\documentclass[3p,a4paper,times]{elsarticle}

\usepackage{amsmath}
\usepackage{amssymb}
\usepackage{amsfonts}
\usepackage{amsthm}
\usepackage{verbatim}
\usepackage{mathrsfs}
\usepackage{graphicx}
\usepackage{caption}
\usepackage{subcaption}
\usepackage{tabularx}
\usepackage{longtable}
\usepackage{tikz}
\usepackage{pgfplots}
\usepackage{xcolor}
\usepackage{pgfplots}
\usepackage{pgfplotstable}
\usepackage{enumitem}
\usepackage{layout}
\allowdisplaybreaks
\usepackage[toc]{appendix}

\usepackage[colorlinks=true, bookmarksopen,
pdfauthor={CST},
pdfcreator={pdftex},
pdfsubject={algorithms},
linkcolor={blue},
anchorcolor={black},
citecolor={firebrick},
filecolor={magenta},
menucolor={black},
pagecolor={red},
plainpages=false,pdfpagelabels,
urlcolor={db} ]{hyperref}

\usepackage[capitalise]{cleveref}
\usepackage[noadjust]{cite}
\usepackage{color}

\usepackage{bm}

\newtheorem{lemma}{Lemma}
\newtheorem{proposition}[lemma]{Proposition}

\newtheorem{theorem}[lemma]{Theorem}

\newtheorem{example}[lemma]{Example}

\newtheorem{algorithm}{Algorithm}

\numberwithin{equation}{section}
\numberwithin{lemma}{section}
\newcommand{\N}{\mathbb{N}}    %natural numbers
\newcommand{\NN}{\mathbb{N}_0} %Natural Nonnegative numbers
\newcommand{\R}{\mathbb{R}}    %real number field

\newcommand{\bo}{\mathcal{O}}

\newcommand{\Qu}{\textsf{u}}
\newcommand{\Qa}{\textsf{a}}
\newcommand{\Qb}{\textsf{b}}
\newcommand{\Qc}{\textsf{c}}
\newcommand{\Qd}{\textsf{d}}

\newcommand{\Qk}{\textsf{k}}
\newcommand{\Qr}{\textsf{r}}
\newcommand{\be}{ \begin{equation} }
	\newcommand{\ee}{ \end{equation} }

\newcommand{\ind}{\Lambda}
\newcommand{\nab}{\nabla}

\journal{XXX}
\begin{document}
	
	\begin{frontmatter}

	\title{Fourth-Order Compact FDMs for  Steady and  Time-Dependent Nonlinear Convection-Diffusion Equations}

		\author[rvt1]{Qiwei Feng\fnref{fn1}\corref{cor}}
\ead{qif21@pitt.edu,qfeng@ualberta.ca}
\address[rvt1]{Department of Mathematics, University of Pittsburgh, Pittsburgh, PA 15260 USA.}

\author[rvt1]{Catalin Trenchea\fnref{fn1}}
\ead{trenchea@pitt.edu}

\cortext[cor]{Corresponding author.}
\fntext[fn1]{Qiwei Feng is partially supported by  the Mathematics Research Center, Department of Mathematics, University of Pittsburgh, Pittsburgh, PA, USA.
Catalin Trenchea is partially supported by the National Science Foundation under grant DMS-2208220.
}

	\makeatletter \@addtoreset{equation}{section} \makeatother

	\begin{abstract}
		In this paper, we discuss the  steady and time-dependent nonlinear convection-diffusion (advection-diffusion) equations  with the Dirichlet boundary condition. For the steady  nonlinear equation, we use an iteration method to reformulate the nonlinear equation into its linear counterpart, and derive a fourth-order compact 9-point finite difference method (FDM) to solve the  reformulated equation on a uniform Cartesian grid. To increase the accuracy, we modify the FDM to reduce the pollution effect. The linear system of the FDM generates an M-matrix,  provided the mesh size $h$ is sufficiently small. For the  time-dependent nonlinear equation, we discrete the temporal domain using the Crank-Nicolson (CN), BDF3, BDF4 time stepping methods, and apply a similar iterative method to rewrite the nonlinear equation as the same linear convection-diffusion equation. Then we propose the second-order to fourth-order compact 9-point FDMs with the reduced pollution effects on a uniform Cartesian grid. We prove that all FDMs satisfy the discrete maximum principle for  sufficiently small $h$.
		Several examples with the variable and  time-dependent  diffusion coefficients and  challenging nonlinear terms (not limited to the Burgers equation) are provided to verify the accuracy and the desired convergence rates in the $l_2$ and $l_{\infty}$ norms in space and time. We also compare our second-order CN method with the third-order BDF3 method and the discontinuous Galerkin (DG) method, and the numerical results demonstrate that our FDM with the coarse time step generates the small error. Especially, if the same BDF3 scheme is applied, our error is 1.6\% of that obtained from the DG method.
		The proposed methods can be easily  extended to a 3D spatial domain and more general nonlinear convection-diffusion-reaction equations.

	\end{abstract}


\end{frontmatter}	
	
	\section{Introduction}

	The nonlinear convection-diffusion-reaction equations (advection-diffusion-reaction equations, the terms 'convection' and 'advection' are used indiscriminately in numerical
	analysis \citep[p.2]{Hundsdorfer2003}) arise in a wide range of important physical and engineering applications. They serve as the foundation for several well-known models, such as the porous medium equation for the investigation of the fluid transport  in porous materials, the Navier-Stokes equation for modeling flows in pipes and channels, and the Fisher-Kolmogorov-Petrovsky-Piskunov equation to describe population growth. To analyze and numerically solve various nonlinear convection-diffusion-reaction equations, we present the following literature review. Additional related works can be found in their references.

	In \citep{Burman2002},  Burman et al. proposed the stabilized Galerkin approximation  for the steady nonlinear  convection-diffusion-reaction equation
	$-\epsilon \Delta u	+ {\bm v} \cdot \nab u +r(u) = f$  with the Dirichlet and Neumann boundary conditions, and  the nonlinear reaction term $r(u)$, on a square in the numerical examples.
	In \citep{Clain2024}, Clain et al.  established the second-order to sixth-order FDMs  with the uniform Cartesian grids for the steady nonlinear  convection-diffusion-reaction equation $	-\kappa(u)\Delta u +   {\bm F}(u)\cdot  \nab(u)+r(u) u = f$ with the nonlinear terms of  diffusion $\kappa(u)$,  convection $F(u)$, and  reaction $r(u)$, the Dirichlet, Neumann, and nonlinear Robin conditions in a curved domain. 
	For $u_t-\epsilon \Delta u	+ \nab\cdot{\bm F} (u)  = f$  with the nonlinear convective term ${\bm F} (u)$, the Dirichlet and/or Neumann boundary conditions on a  bounded polyhedral domain in 2D and/or 3D, the priori asymptotic error estimate of the discontinuous Galerkin (DG) method was deduced in \citep{Dolej2007,Dolej2005,DolejVlas2008,Feistauer2011}, and Feistauer et al.  proposed in \citep{Feistauer1997} the convergence analysis of the combined finite volume-finite element (FV-FE) method. 
	In \citep{Dolej2007,Dolej2005,DolejVlas2008}, Dolej\v{s}\'{i} et al. tested the 2D viscous Burgers equation with  the Dirichlet boundary condition on a square in the numerical example.
	Nguyen et al. introduced in \citep{Nguyen2009}  the fourth-order DG method for  $-\nab\cdot (\kappa \nab u) +   \nab\cdot{\bm F} (u)  = f$, and the third-order DG method for  $u_t-\nab\cdot (\kappa \nab u) +   \nab\cdot{\bm F} (u)  = f$ with the constant diffusion $\kappa$, the nonlinear convection ${\bm F}(u)=(u^2/2,u^2/2)$, and the Dirichlet boundary condition on a rectangle in the numerical experiments.
	For $u_t-\nab\cdot (K(u) \nab u) 	+ \nab\cdot{\bm F} (u)  = 0$ with the periodic boundary condition and the nonlinear diffusion term $K(u)$, where $K(u)$ is a symmetric
	and positive definite matrix of the variable $u$,   
	Xu et al. in \citep{XuShu2007}  constructed the error estimate for the semi-discrete local DG method on a rectangular domain, and 
	Yan  provided in \citep{Yan2013} the error analysis for the nonsymmetric DG method using the uniform and nonuniform meshes on a bounded domain in 1D and 2D. The equation $u_t-(u^2)_{xx}-(u^2)_{yy}=0$ and the 2D strongly degenerate parabolic problem in a square domain were examined in \citep{Yan2013}.
	New central schemes that maintain the high-resolution independent of $\bo(1/\tau)$ for the general nonlinear diffusion equation $u_t-(P(u,u_x,u_y))_x-(Q(u,u_x,u_y))_x	+ \nab\cdot{\bm F} (u)  = 0$ were offered by Kurganov et al. in \citep{Kurganov2000}, where $\tau$ is the time step of the temporal discretization.
	In \citep{Eymard2010}, Eymard et al.  described the combined FV-FE method for $(\alpha(u))_t-\nab\cdot (K \nab u) 	+ \nab \cdot ({\bm v}   u)+r(u)  = f$ with nonlinear $\alpha(u)$ and $r(u)$,  the Dirichlet and Neumann boundary conditions, and the nonmatching grids on a bounded  domain  in 2D and 3D. 
	In \citep{Tezduyar1986}, Tezduyar et al.  built the streamline-upwind/Petrov-Galerkin method  for the steady  nonlinear  convection-diffusion-reaction equation
	systems. The numerical approximated solutions from \citep{Tezduyar1986} are 
	accurate with the minimal numerical dissipation and oscillations. 
	For the systems of the  time-dependent nonlinear 
	convection-diffusion-reaction equations,
	Bause et al.  presented in \citep{Bause2012} a conforming FEM of up to the fourth order using the regular triangular mesh, Cockburn et al. in \citep{Cockburn1998} extended the Runge-Kutta DG method of the purely hyperbolic system to design  the local DG method, and Michoski et al. extended in \citep{Michoski2017}  the von Neumann analysis
	to prove the  nonlinear stability of the high-order  DG method.
	Furthermore,
	Canc\'{e}s et al. analyzed in \citep{Cances2020} the large time behavior of the nonlinear FVMs for the linear  anisotropic convection-diffusion equations with the Dirichlet and Neumann boundary conditions in a bounded domain.

	In this paper,  we consider the steady   and  time-dependent nonlinear convection-diffusion equations in the spatial  domain $\Omega=(0,1)^2$ and the temporal domain $I=[0,1]$ as follows
	\be\label{Model：Original:Elliptic}
	\begin{cases}
		-\nab\cdot (\kappa \nab u) +   \nab\cdot{\bm F} (u)  = f, \qquad \qquad  &  (x,y)\in \Omega,\\
		u =g,  & (x,y)\in \partial \Omega,
	\end{cases}
	\ee
	and
	\be\label{Model：Original:Parabolic}
	\begin{cases}
		u_t-\nab\cdot (\kappa \nab u) +   \nab\cdot{\bm F} (u)  = f, \qquad \qquad  &  (x,y)\in \Omega  \hspace{0.63cm} \text{and}  \quad t\in I, \\	
		u= u^0,     & (x,y)\in \Omega \hspace{0.63cm} \text{and} \quad t=0,\\
		u =g,    & (x,y)\in \partial \Omega \quad \text{and} \quad t\in I,
	\end{cases}
	\ee
	where $u$ is the unknown scalar variable function; $\kappa$ and $f$ denote the available scalar-valued diffusion coefficient and the source term, respectively; $g$ and $u^0$ represent the given Dirichlet boundary function of $u$ on $\Omega$ and the initial value of $u$ at $t = 0$, respectively. Here  $\bm F$ is the vector-valued nonlinear function of $u$, i.e., 
	\[
	{\bm F(u)}=(\alpha(u),\beta(u)), \quad \text{with two nonlinear scalar functions } \alpha(u),\ \beta(u)  \text{ of the variable } u.
	\]
	For example, if 
	\[
	(\alpha(u),\beta(u))=(u^2/2, u^2/2),\quad (\sin(u),\cos(u)),\quad  (\cos(u),\exp(u)),
	\]
	then
	\[
	\nab\cdot{\bm F} (u)=uu_x+uu_y, \quad \cos(u)u_x-\sin(u)u_y, \quad -\sin(u)u_x+\exp(u)u.
	\]
	In this paper, we assume that $\kappa>0$, and $u,\kappa,f, \alpha, \beta$ are smooth in $\Omega$ and $I$.
The rest of the paper is organized as follows:
	
	In \cref{sec:elliptic} we consider the steady nonlinear  convection-diffusion equation \eqref{Model：Original:Elliptic}. We first reformulate \eqref{Model：Original:Elliptic} into a linear convection-diffusion equation in \cref{sec:reformulate:1}. Then, we provide the explicit expression of the fourth-order compact 9-point FDM  for the linear equation  in \cref{subsec:elliptic} by using the techniques in \cref{subsec:tech}. To increase the accuracy of our FDM, we reduce the pollution effect  in \cref{sec:FDM:reduced}.

	In \cref{sec:parabolic} we discuss the time-dependent nonlinear convection-diffusion equation \eqref{Model：Original:Parabolic}. We apply the Crank-Nicolson (CN), BDF3, and BDF4 time stepping methods, and rewrite \eqref{Model：Original:Parabolic} as the linear convection-diffusion equation in \cref{sec:reformulate:2}. The second- to fourth-order compact 9-point FDMs with the reduced pollution effects are proposed in \cref{subsec:parabolic}.
	
	In \cref{sec:Numeri} we provide various examples with the variable $\kappa(x,y),\kappa(x,y,t)$, and the challenging $\alpha(u),\beta(u)$, and compare our scheme with the  discontinuous Galerkin (DG) and the BDF3 methods in \citep{Nguyen2009}.
	
	In \cref{sec:Contribu} we highlight the main conclusion of this paper.
	\section{Fourth-order compact 9-point FDMs for the  steady nonlinear convection-diffusion equation}\label{sec:elliptic}

	In this section we convert the steady nonlinear  convection-diffusion equation \eqref{Model：Original:Elliptic} into its linear counterpart by a fixed point method, and derive the fourth-order compact 9-point FDMs.
	\subsection{Reformulation of  the steady nonlinear  convection-diffusion equation}\label{sec:reformulate:1}

	Clearly, \eqref{Model：Original:Elliptic} results in 
	\[
	-  \kappa \Delta u -\kappa_x u_x -\kappa_y u_y +   \alpha_u(u)u_x+\beta_u(u)u_y  = f.
	\]
	Let 
	\be\label{a:b:psi}
	a:=\frac{\kappa_x -\alpha_u(u)} {  \kappa}, \qquad b:=\frac{\kappa_y -\beta_u(u)} {  \kappa}, \qquad \psi :=-\frac{f} {  \kappa}.
	\ee
	Then \eqref{Model：Original:Elliptic} is equivalent to
	\be\label{Model：simplify:Elliptic}
	\begin{cases}
		\Delta u +a u_x +b u_y    = \psi, \qquad \qquad  &  (x,y)\in \Omega,\\
		u =g,  &  (x,y)\in  \partial \Omega.
	\end{cases}
	\ee
	In this paper,  an iteration method is applied  to numerically solve the nonlinear  equation \eqref{Model：simplify:Elliptic}:
	\be\label{Model：iteration:Elliptic}
	\begin{cases}
		\Delta u_{\Qk+1} +a_{\Qk} (u_{\Qk+1})_x +b_{\Qk} (u_{\Qk+1})_y   =\psi, \qquad \qquad &  (x,y)\in \Omega,\\
		u_{\Qk+1} =g,  &  (x,y)\in \partial \Omega,\\
		\Qk:=\Qk+1, \ \text{the initial  } \Qk=0,\ u_0= 0,   
	\end{cases}
	\ee
	where $\Qk$ denotes the index of the $\Qk$-th iteration, and 
	\be\label{aibi:elliptic}
	a_{\Qk}:=\frac{\kappa_x -\alpha_u(u_\Qk)} {  \kappa}, \qquad b_\Qk:=\frac{\kappa_y -\beta_u(u_\Qk)} {  \kappa}.
	\ee
	As $u_\Qk$ is computed in the $\Qk$-th iteration, \eqref{Model：iteration:Elliptic} can be considered as the linear problem in the $(\Qk+1)$-th iteration. Precisely, we rewrite the nonlinear problem \eqref{Model：Original:Elliptic} as the following steady linear convection-diffusion equation:
	\be\label{Model：linear:Elliptic}
	\begin{cases}
		\Delta \Qu  +\Qa \Qu_x +\Qb \Qu_y   = \psi, \qquad \qquad &  (x,y)\in \Omega,\\
		\Qu =g,  &  (x,y)\in  \partial \Omega,
	\end{cases}
	\ee
	where
	\be\label{QuQaQb}
	\Qu:=u_{\Qk+1}, \qquad \Qa:=a_\Qk=\frac{\kappa_x -\alpha_u(u_\Qk)} {  \kappa}, \qquad \Qb:=b_\Qk=\frac{\kappa_y -\beta_u(u_\Qk)} {  \kappa}.
	\ee
	To solve \eqref{Model：linear:Elliptic} efficiently, we propose the fourth-order compact 9-point FDMs in the following \cref{subsec:tech,subsec:elliptic,sec:FDM:reduced}.

	\subsection{Techniques to derive the high-order FDMs }\label{subsec:tech}
	\begin{figure}[h]
		\centering	
		\begin{subfigure}[b]{0.3\textwidth}
			\hspace{0.3cm}
			\begin{tikzpicture}[scale = 0.45]
				\node (A) at (0,7) {$\Qu^{(0,0)}$};
				\node (A) at (0,6) {$\Qu^{(0,1)}$};
				\node (A) at (0,5) {$\Qu^{(0,2)}$};
				\node (A) at (0,4) {$\Qu^{(0,3)}$};
				\node (A) at (0,3) {$\Qu^{(0,4)}$};
				\node (A) at (0,2) {$\Qu^{(0,5)}$};
				\node (A) at (0,1) {$\Qu^{(0,6)}$};
				\node (A) at (0,0) {$\Qu^{(0,7)}$};
				%%%%%%%%%%%%%%%%%%%%%%%%%%%%%%%%%%%%%
				\node (A) at (2,7) {$\Qu^{(1,0)}$};
				\node (A) at (2,6) {$\Qu^{(1,1)}$};
				\node (A) at (2,5) {$\Qu^{(1,2)}$};
				\node (A) at (2,4) {$\Qu^{(1,3)}$};
				\node (A) at (2,3) {$\Qu^{(1,4)}$};
				\node (A) at (2,2) {$\Qu^{(1,5)}$};
				\node (A) at (2,1) {$\Qu^{(1,6)}$};	
				%%%%%%%%%%%%%%%%%%%%%%%%%%%%%%%%%%%%%
				\node (A) at (4,7) {$\Qu^{(2,0)}$};
				\node (A) at (4,6) {$\Qu^{(2,1)}$};
				\node (A) at (4,5) {$\Qu^{(2,2)}$};
				\node (A) at (4,4) {$\Qu^{(2,3)}$};
				\node (A) at (4,3) {$\Qu^{(2,4)}$};
				\node (A) at (4,2) {$\Qu^{(2,5)}$};
				%%%%%%%%%%%%%%%%%%%%%%%%%%%%%%%%%%%%%
				\node (A) at (6,7) {$\Qu^{(3,0)}$};
				\node (A) at (6,6) {$\Qu^{(3,1)}$};
				\node (A) at (6,5) {$\Qu^{(3,2)}$};
				\node (A) at (6,4) {$\Qu^{(3,3)}$};
				\node (A) at (6,3) {$\Qu^{(3,4)}$};
				%%%%%%%%%%%%%%%%%%%%%%%%%%%%%%%%%%%%%
				\node (A) at (8,7) {$\Qu^{(4,0)}$};
				\node (A) at (8,6) {$\Qu^{(4,1)}$};
				\node (A) at (8,5) {$\Qu^{(4,2)}$};
				\node (A) at (8,4) {$\Qu^{(4,3)}$};
				%%%%%%%%%%%%%%%%%%%%%%%%%%%%%%%%%%%%%
				\node (A) at (10,7) {$\Qu^{(5,0)}$};
				\node (A) at (10,6) {$\Qu^{(5,1)}$};
				\node (A) at (10,5) {$\Qu^{(5,2)}$};
				%%%%%%%%%%%%%%%%%%%%%%%%%%%%%%%%%%%%%
				\node (A) at (12,7) {$\Qu^{(6,0)}$};
				\node (A) at (12,6) {$\Qu^{(6,1)}$};
				%%%%%%%%%%%%%%%%%%%%%%%%%%%%%%%%%%%%%
				\node (A) at (14,7) {$\Qu^{(7,0)}$};
				\node (A) at (15,3.8) {\color{blue}\huge{$=$}};
			\end{tikzpicture}
		\end{subfigure}
		%%%%%%%%%%%%%%%%%%%%%
		\begin{subfigure}[b]{0.3\textwidth}
			\hspace{3.3cm}
			\begin{tikzpicture}[scale = 0.45]
				\node (A) at (0,7) {$\Qu^{(0,0)}$};
				\node (A) at (0,6) {$\Qu^{(0,1)}$};
				\node (A) at (0,5) {$\Qu^{(0,2)}$};
				\node (A) at (0,4) {$\Qu^{(0,3)}$};
				\node (A) at (0,3) {$\Qu^{(0,4)}$};
				\node (A) at (0,2) {$\Qu^{(0,5)}$};
				\node (A) at (0,1) {$\Qu^{(0,6)}$};
				\node (A) at (0,0) {$\Qu^{(0,7)}$};
				%%%%%%%%%%%%%%%%%%%%%%%%%%%%%%%%%%%%%
				\node (A) at (2,7) {$\Qu^{(1,0)}$};
				\node (A) at (2,6) {$\Qu^{(1,1)}$};
				\node (A) at (2,5) {$\Qu^{(1,2)}$};
				\node (A) at (2,4) {$\Qu^{(1,3)}$};
				\node (A) at (2,3) {$\Qu^{(1,4)}$};
				\node (A) at (2,2) {$\Qu^{(1,5)}$};
				\node (A) at (2,1) {$\Qu^{(1,6)}$};	 
				\node (A) at (3.8,4) {\color{blue}\huge{$\cup$}};     	
			\end{tikzpicture}
		\end{subfigure}
		\begin{subfigure}[b]{0.3\textwidth}
			\hspace{0.3cm}
			\vspace{0.9cm}
			\begin{tikzpicture}[scale = 0.45]
				%%%%%%%%%%%%%%%%%%%%%%%%%%%%%%%%%%%%%
				\node (A) at (4,7) {$\Qu^{(2,0)}$};
				\node (A) at (4,6) {$\Qu^{(2,1)}$};
				\node (A) at (4,5) {$\Qu^{(2,2)}$};
				\node (A) at (4,4) {$\Qu^{(2,3)}$};
				\node (A) at (4,3) {$\Qu^{(2,4)}$};
				\node (A) at (4,2) {$\Qu^{(2,5)}$};
				%%%%%%%%%%%%%%%%%%%%%%%%%%%%%%%%%%%%%
				\node (A) at (6,7) {$\Qu^{(3,0)}$};
				\node (A) at (6,6) {$\Qu^{(3,1)}$};
				\node (A) at (6,5) {$\Qu^{(3,2)}$};
				\node (A) at (6,4) {$\Qu^{(3,3)}$};
				\node (A) at (6,3) {$\Qu^{(3,4)}$};
				%%%%%%%%%%%%%%%%%%%%%%%%%%%%%%%%%%%%%
				\node (A) at (8,7) {$\Qu^{(4,0)}$};
				\node (A) at (8,6) {$\Qu^{(4,1)}$};
				\node (A) at (8,5) {$\Qu^{(4,2)}$};
				\node (A) at (8,4) {$\Qu^{(4,3)}$};
				%%%%%%%%%%%%%%%%%%%%%%%%%%%%%%%%%%%%%
				\node (A) at (10,7) {$\Qu^{(5,0)}$};
				\node (A) at (10,6) {$\Qu^{(5,1)}$};
				\node (A) at (10,5) {$\Qu^{(5,2)}$};
				%%%%%%%%%%%%%%%%%%%%%%%%%%%%%%%%%%%%%
				\node (A) at (12,7) {$\Qu^{(6,0)}$};
				\node (A) at (12,6) {$\Qu^{(6,1)}$};
				%%%%%%%%%%%%%%%%%%%%%%%%%%%%%%%%%%%%%
				\node (A) at (14,7) {$\Qu^{(7,0)}$};
			\end{tikzpicture}
		\end{subfigure}	
		\caption
		{The illustration for $\ind_M=\ind_M^1 \cup \ind_M^2$ in \eqref{ind}--\eqref{ind12} with $M=7$. }
		\label{fig:umn:1}
	\end{figure}
	\begin{figure}[h]
		\centering	
		\begin{subfigure}[b]{0.3\textwidth}
			\hspace{0cm}
			\vspace{0.9cm}
			\begin{tikzpicture}[scale = 0.45]
				%%%%%%%%%%%%%%%%%%%%%%%%%%%%%%%%%%%%%
				\node (A) at (4,7) {$\Qu^{(2,0)}$};
				\node (A) at (4,6) {$\Qu^{(2,1)}$};
				\node (A) at (4,5) {$\Qu^{(2,2)}$};
				\node (A) at (4,4) {$\Qu^{(2,3)}$};
				\node (A) at (4,3) {$\Qu^{(2,4)}$};
				\node (A) at (4,2) {$\Qu^{(2,5)}$};
				%%%%%%%%%%%%%%%%%%%%%%%%%%%%%%%%%%%%%
				\node (A) at (6,7) {$\Qu^{(3,0)}$};
				\node (A) at (6,6) {$\Qu^{(3,1)}$};
				\node (A) at (6,5) {$\Qu^{(3,2)}$};
				\node (A) at (6,4) {$\Qu^{(3,3)}$};
				\node (A) at (6,3) {$\Qu^{(3,4)}$};
				%%%%%%%%%%%%%%%%%%%%%%%%%%%%%%%%%%%%%
				\node (A) at (8,7) {$\Qu^{(4,0)}$};
				\node (A) at (8,6) {$\Qu^{(4,1)}$};
				\node (A) at (8,5) {$\Qu^{(4,2)}$};
				\node (A) at (8,4) {$\Qu^{(4,3)}$};
				%%%%%%%%%%%%%%%%%%%%%%%%%%%%%%%%%%%%%
				\node (A) at (10,7) {$\Qu^{(5,0)}$};
				\node (A) at (10,6) {$\Qu^{(5,1)}$};
				\node (A) at (10,5) {$\Qu^{(5,2)}$};
				%%%%%%%%%%%%%%%%%%%%%%%%%%%%%%%%%%%%%
				\node (A) at (12,7) {$\Qu^{(6,0)}$};
				\node (A) at (12,6) {$\Qu^{(6,1)}$};
				%%%%%%%%%%%%%%%%%%%%%%%%%%%%%%%%%%%%%
				\node (A) at (14,7) {$\Qu^{(7,0)}$};
				\node (A) at (15,3.8) {\color{blue}\huge{$=$}};
			\end{tikzpicture}
		\end{subfigure}
		%%%%%%%%%%%%%%%%%%%%%
		\begin{subfigure}[b]{0.3\textwidth}
			\hspace{1cm}
			\begin{tikzpicture}[scale = 0.45]
				\node (A) at (0,7) {$\Qu^{(0,0)}$};
				\node (A) at (0,6) {$\Qu^{(0,1)}$};
				\node (A) at (0,5) {$\Qu^{(0,2)}$};
				\node (A) at (0,4) {$\Qu^{(0,3)}$};
				\node (A) at (0,3) {$\Qu^{(0,4)}$};
				\node (A) at (0,2) {$\Qu^{(0,5)}$};
				\node (A) at (0,1) {$\Qu^{(0,6)}$};
				\node (A) at (0,0) {$\Qu^{(0,7)}$};
				%%%%%%%%%%%%%%%%%%%%%%%%%%%%%%%%%%%%%
				\node (A) at (2,7) {$\Qu^{(1,0)}$};
				\node (A) at (2,6) {$\Qu^{(1,1)}$};
				\node (A) at (2,5) {$\Qu^{(1,2)}$};
				\node (A) at (2,4) {$\Qu^{(1,3)}$};
				\node (A) at (2,3) {$\Qu^{(1,4)}$};
				\node (A) at (2,2) {$\Qu^{(1,5)}$};
				\node (A) at (2,1) {$\Qu^{(1,6)}$};	      	
			\end{tikzpicture}
		\end{subfigure}
		\begin{subfigure}[b]{0.3\textwidth}
			\hspace{-0.8cm}
			\vspace{0.9cm}
			\begin{tikzpicture}[scale = 0.45]
			    \node (A) at (1.1,4) {\color{blue}\huge{$+$}};
				%%%%%%%%%%%%%%%%%%%%%%%%%%%%%%%%%%%%%
				\node (A) at (4,7) {$\Qa^{(0,0)}$};
				\node (A) at (4,6) {$\Qa^{(0,1)}$};
				\node (A) at (4,5) {$\Qa^{(0,2)}$};
				\node (A) at (4,4) {$\Qa^{(0,3)}$};
				\node (A) at (4,3) {$\Qa^{(0,4)}$};
				\node (A) at (4,2) {$\Qa^{(0,5)}$};
				%%%%%%%%%%%%%%%%%%%%%%%%%%%%%%%%%%%%%
				\node (A) at (6,7) {$\Qa^{(1,0)}$};
				\node (A) at (6,6) {$\Qa^{(1,1)}$};
				\node (A) at (6,5) {$\Qa^{(1,2)}$};
				\node (A) at (6,4) {$\Qa^{(1,3)}$};
				\node (A) at (6,3) {$\Qa^{(1,4)}$};
				%%%%%%%%%%%%%%%%%%%%%%%%%%%%%%%%%%%%%
				\node (A) at (8,7) {$\Qa^{(2,0)}$};
				\node (A) at (8,6) {$\Qa^{(2,1)}$};
				\node (A) at (8,5) {$\Qa^{(2,2)}$};
				\node (A) at (8,4) {$\Qa^{(2,3)}$};
				%%%%%%%%%%%%%%%%%%%%%%%%%%%%%%%%%%%%%
				\node (A) at (10,7) {$\Qa^{(3,0)}$};
				\node (A) at (10,6) {$\Qa^{(3,1)}$};
				\node (A) at (10,5) {$\Qa^{(3,2)}$};
				%%%%%%%%%%%%%%%%%%%%%%%%%%%%%%%%%%%%%
				\node (A) at (12,7) {$\Qa^{(4,0)}$};
				\node (A) at (12,6) {$\Qa^{(4,1)}$};
				%%%%%%%%%%%%%%%%%%%%%%%%%%%%%%%%%%%%%
				\node (A) at (14,7) {$\Qa^{(5,0)}$};
			\end{tikzpicture}
		\end{subfigure}	
			\begin{subfigure}[b]{0.3\textwidth}
			\hspace{1.65cm}
			\begin{tikzpicture}[scale = 0.45]
				\node (A) at (2.2,4.5) {\color{blue}\huge{$+$}};
				%%%%%%%%%%%%%%%%%%%%%%%%%%%%%%%%%%%%%
				\node (A) at (4,7) {$\Qb^{(0,0)}$};
				\node (A) at (4,6) {$\Qb^{(0,1)}$};
				\node (A) at (4,5) {$\Qb^{(0,2)}$};
				\node (A) at (4,4) {$\Qb^{(0,3)}$};
				\node (A) at (4,3) {$\Qb^{(0,4)}$};
				\node (A) at (4,2) {$\Qb^{(0,5)}$};
				%%%%%%%%%%%%%%%%%%%%%%%%%%%%%%%%%%%%%
				\node (A) at (6,7) {$\Qb^{(1,0)}$};
				\node (A) at (6,6) {$\Qb^{(1,1)}$};
				\node (A) at (6,5) {$\Qb^{(1,2)}$};
				\node (A) at (6,4) {$\Qb^{(1,3)}$};
				\node (A) at (6,3) {$\Qb^{(1,4)}$};
				%%%%%%%%%%%%%%%%%%%%%%%%%%%%%%%%%%%%%
				\node (A) at (8,7) {$\Qb^{(2,0)}$};
				\node (A) at (8,6) {$\Qb^{(2,1)}$};
				\node (A) at (8,5) {$\Qb^{(2,2)}$};
				\node (A) at (8,4) {$\Qb^{(2,3)}$};
				%%%%%%%%%%%%%%%%%%%%%%%%%%%%%%%%%%%%%
				\node (A) at (10,7) {$\Qb^{(3,0)}$};
				\node (A) at (10,6) {$\Qb^{(3,1)}$};
				\node (A) at (10,5) {$\Qb^{(3,2)}$};
				%%%%%%%%%%%%%%%%%%%%%%%%%%%%%%%%%%%%%
				\node (A) at (12,7) {$\Qb^{(4,0)}$};
				\node (A) at (12,6) {$\Qb^{(4,1)}$};
				%%%%%%%%%%%%%%%%%%%%%%%%%%%%%%%%%%%%%
				\node (A) at (14,7) {$\Qb^{(5,0)}$};
			\end{tikzpicture}
		\end{subfigure}	
			\begin{subfigure}[b]{0.3\textwidth}
			\hspace{2cm}
			\begin{tikzpicture}[scale = 0.45]
				\node (A) at (1,4.5) {\color{blue}\huge{$+$}};
				%%%%%%%%%%%%%%%%%%%%%%%%%%%%%%%%%%%%%
\node (A) at (4,7) {$\psi^{(0,0)}$};
\node (A) at (4,6) {$\psi^{(0,1)}$};
\node (A) at (4,5) {$\psi^{(0,2)}$};
\node (A) at (4,4) {$\psi^{(0,3)}$};
\node (A) at (4,3) {$\psi^{(0,4)}$};
\node (A) at (4,2) {$\psi^{(0,5)}$};
%%%%%%%%%%%%%%%%%%%%%%%%%%%%%%%%%%%%%
\node (A) at (6,7) {$\psi^{(1,0)}$};
\node (A) at (6,6) {$\psi^{(1,1)}$};
\node (A) at (6,5) {$\psi^{(1,2)}$};
\node (A) at (6,4) {$\psi^{(1,3)}$};
\node (A) at (6,3) {$\psi^{(1,4)}$};
%%%%%%%%%%%%%%%%%%%%%%%%%%%%%%%%%%%%%
\node (A) at (8,7) {$\psi^{(2,0)}$};
\node (A) at (8,6) {$\psi^{(2,1)}$};
\node (A) at (8,5) {$\psi^{(2,2)}$};
\node (A) at (8,4) {$\psi^{(2,3)}$};
%%%%%%%%%%%%%%%%%%%%%%%%%%%%%%%%%%%%%
\node (A) at (10,7) {$\psi^{(3,0)}$};
\node (A) at (10,6) {$\psi^{(3,1)}$};
\node (A) at (10,5) {$\psi^{(3,2)}$};
%%%%%%%%%%%%%%%%%%%%%%%%%%%%%%%%%%%%%
\node (A) at (12,7) {$\psi^{(4,0)}$};
\node (A) at (12,6) {$\psi^{(4,1)}$};
%%%%%%%%%%%%%%%%%%%%%%%%%%%%%%%%%%%%%
\node (A) at (14,7) {$\psi^{(5,0)}$};
			\end{tikzpicture}
		\end{subfigure}	
		\caption
		{The illustration for \eqref{u:pq} with $M=7$. }
		\label{fig:upq:1}
	\end{figure}

	In this section, we present the key techniques for deriving the high-order (not limited to the fourth-order) FDMs in 2D, which can be naturally extended to the three-dimensional spatial domain.
	For $M\in \NN:=\N\cup\{0\}$, we define (see \cref{fig:umn:1} for an illustration)
	\be \label{ind}
	\ind_{M}:=\{ (m,n)\in \NN^2 \; : \;  m+n\le M \}, 
	\ee
	\be \label{ind12}
	\ind_{M}^{ 2}:=\ind_{M}\setminus \ind_{M}^{ 1}\quad \mbox{with}\quad
	\ind_{M}^{ 1}:=\{ (m,n)\in \ind_{M} \; : \; m=0,1\}.
	\ee
	For any smooth function $\varphi$, we also define
	\be\label{varphi:mn}
	\begin{split}
		\varphi^{(m,n)}:=\frac{\partial^{m+n} \varphi(x,y)}{ \partial^m x \partial^n y}, \text{ i.e., the } (m,n)\text{-th order partial derivative of } \varphi(x,y).
	\end{split}
	\ee
	The main technique to derive the high-order FDM is the following formula \eqref{u:GH}. Note that \eqref{u:GH} can be obtained by \citep[(2.2)-(2.12)]{Feng2024}. For the readers' convenience and to further simplify the corresponding expression in \citep[(2.12)]{Feng2024}, we provide more details in the following relations \eqref{u20}-\eqref{u:GH}.
	It follows from \eqref{Model：linear:Elliptic} and \eqref{varphi:mn}
	that
	\be\label{u20}
	\Qu^{(2,0)}  =-\Qu^{(0,2)}-\Qa^{(0,0)} \Qu^{(1,0)} -\Qb^{(0,0)} \Qu^{(0,1)} +  \psi^{(0,0)},
	\ee
	where $\Qa,\Qb, \Qu$ are defined in \eqref{QuQaQb}, and $\psi$ is defined in \eqref{a:b:psi}.
Next,	we take the $m$-th derivative with respect to $x$ of \eqref{u20}, 
	\be\label{u:m20}
	\Qu^{(m+2,0)}  =-\Qu^{(m,2)}- \sum_{i=0}^m{m \choose i}  \Qa^{(m-i,0)} \Qu^{(i+1,0)} -\sum_{i=0}^m{m \choose i} \Qb^{(m-i,0)} \Qu^{(i,1)} +  \psi^{(m,0)}, \quad m\ge 0.
	\ee
We also	take the $n$-th derivative with respect to $y$ of \eqref{u:m20} to obtain
	\be\label{u:m2n}
	\begin{split}
		\Qu^{(m+2,n)}  =&-\Qu^{(m,n+2)}- \sum_{j=0}^n   \sum_{i=0}^m {n \choose j} {m \choose i}  \Qa^{(m-i,n-j)} \Qu^{(i+1,j)}\\ 
		& - \sum_{j=0}^n \sum_{i=0}^m {n \choose j} {m \choose i} \Qb^{(m-i,n-j)} \Qu^{(i,j+1)} +  \psi^{(m,n)}, \quad m,n\ge 0.
	\end{split}
	\ee
Then	\eqref{u:m2n} implies
	\be\label{recursive:1}
	\Qu^{(m+2,n)} \longrightarrow (\Qu^{(m,n+2)}, \Qu^{(i+1,j)},  \Qu^{(i,j+1)}, \psi^{(m,n)}), \quad \text{with} \quad 0\le i \le m,\quad 0\le j\le n, \quad m,n\ge 0.
	\ee
	Applying \eqref{recursive:1} recursively $m+1$ times yields 
	\be\label{recursive:2}
	\Qu^{(m+2,n)} \longrightarrow (\Qu^{(0,j_0)},  \Qu^{(1,j_1)}, \psi^{(i,j)}),
	\ee
	with $0\le j_0 \le m+n+2$, $0\le j_1\le m+n+1$, $m,n\ge 0$, $0 \le i,j \le m+n$, and $i+j\le m+n$ (examples of \eqref{recursive:2} for $\Qu^{(m+2,n)}$ with $m=0,1,2$ are explicitly given in the following identities \eqref{u2q:example}--\eqref{u4q:xi}). Namely,  
	\be\label{u:m2n:2}
	\Qu^{(m+2,n)}  =\sum_{j=0}^{m+2} \xi_{m+2,n,0,j} \Qu^{(0,j)}+\sum_{j=0}^{m+1}  \xi_{m+2,n,1,j} \Qu^{(1,j)}+\mathop{   \sum_{i,j=0}}_{i+j\le m+n}^{m+n}   \eta_{m+2,n,i,j} \psi^{(i,j)},\qquad m,n\ge 0,
	\ee
	where  $ \xi_{m+2,n,0,j},  \xi_{m+2,n,1,j},  \eta_{m+2,n,i,j}  $  are uniquely determined by the above recursive algorithms \eqref{u:m2n}--\eqref{recursive:2}, and only depend on the high-order partial derivatives of $\Qa$ and $\Qb$. Now, we replace $(m+2,n)$ by $(p,q)$ in \eqref{u:m2n:2} to get
	\be\label{u:pq:step1}
	\Qu^{(p,q)}  =\sum_{j=0}^{p} \xi_{p,q,0,j} \Qu^{(0,j)}+\sum_{j=0}^{p-1}  \xi_{p,q,1,j} \Qu^{(1,j)}+\mathop{   \sum_{i,j=0}}_{i+j\le p+q-2}^{p+q-2}   \eta_{p,q,i,j} \psi^{(i,j)}, \qquad p\ge 2, q\ge 0.
	\ee
	According to the definitions of $\ind_{M}, \ind_{M}^{1}, \ind_{M}^{2}$ in \eqref{ind} and \eqref{ind12}, 
	\be\label{u:pq}
	\Qu^{(p,q)}  =\sum_{(m,n)\in \ind_{p+q}^{1} } \xi_{p,q,m,n} \Qu^{(m,n)}+\sum_{(m,n)\in \ind_{p+q-2} } \eta_{p,q,m,n}   \psi^{(m,n)}, \qquad (p,q)\in \ind_{M}^{2},
	\ee
	where $\psi, \Qa,\Qb, \Qu$ are defined in \eqref{a:b:psi} and \eqref{QuQaQb}, and $\xi_{p,q,m,n},   \eta_{p,q,m,n}$ are uniquely determined by the high-order partial derivatives of $\Qa$ and $\Qb$ (see \cref{fig:upq:1} for an illustration). 
	By  the derivations  \eqref{u:m2n}--\eqref{u:pq:step1} of \eqref{u:pq},
	we can say that  $\xi_{p,q,m,n}$ and  $\eta_{p,q,m,n}$ are uniquely determined by the multiplications of the high-order partial derivatives of $\Qa$ and $\Qb$, i.e.,   $\xi_{p,q,m,n}$   and $\eta_{p,q,m,n}$ in \eqref{u:pq} belong to 
	\be\label{Space:S}
	S:=\text{span}\bigg\{ \mathop{\prod_{ i_1,j_1 \in \N_0}}_{v_1,w_1\in \N_0}\Qa^{(i_1,j_1)}\Qb^{(v_1,w_1)},  \mathop{\prod_{ i_2,j_2 \in \N_0 }}_{v_2,w_2\in \N_0}\Qa^{(i_2,j_2)}\Qb^{(v_2,w_2)}, \dots,  \mathop{\prod_{ i_k,j_k \in \N_0 }}_{v_k,w_k\in \N_0}\Qa^{(i_k,j_k)}\Qb^{(v_k,w_k)}\bigg\}, 
	\ee
	with $k\in \N$, and the coefficients in $\R$. Precisely,
	\be\label{xi:and:eta}
	\begin{split}
		& \xi_{p,q,m,n}= \sum_{k \in \N} c_{p,q,m,n,k}\prod_{ i_k,j_k,v_k,w_k\in \N_0}\Qa^{(i_k,j_k)}\Qb^{(v_k,w_k)}, \quad \text{with} \quad c_{p,q,m,n,k}\in \R,\\
		& \eta_{p,q,m,n}= \sum_{k \in \N} d_{p,q,m,n,k}\prod_{ i_k,j_k,v_k,w_k\in \N_0}\Qa^{(i_k,j_k)}\Qb^{(v_k,w_k)}, \quad \text{with} \quad d_{p,q,m,n,k}\in \R.
	\end{split}
	\ee

	For example, $\Qu^{(p,q)}$ with $(p,q)\in \ind_{M}^{2}$, $M\ge 4$, and $p=2,3,4$ in \eqref{u:pq} can be explicitly given by
	\be\label{u2q:example}
	\Qu^{(2,q)}  =-\Qu^{(0,q+2)}- \sum_{k=0}^q{q \choose k}  \left[ \Qa^{(0,q-k)} \Qu^{(1,k)} + \Qb^{(0,q-k)} \Qu^{(0,k+1)}\right] +  \psi^{(0,q)}, \quad  0 \le q \le M-2,
	\ee	
	\be\label{u3q:example}
	\begin{split}	
		\Qu^{(3,q)}  =&	-\Qu^{(1, q+2)}+\sum_{k=0}^q{q \choose k} \big[ (\Qa\Qb - \Qb^{(1, 0)})^{(0,q-k)}\Qu^{(0, k+1)}		+ \Qa^{(0,q-k)}\Qu^{(0, k+2)} \\
		& 	+(\Qa^2 - \Qa^{(1, 0)})^{(0,q-k)}\Qu^{(1, k)}\big]- \sum_{k=0}^q{q \choose k} \Qb^{(0,q-k)}\Qu^{(1, k+1)}  \\
		&-\sum_{k=0}^q{q \choose k}\Qa^{(0,q-k)} \psi^{(0,k)}+ \psi^{(1,q)},\quad   0 \le q \le M-3,
	\end{split}
	\ee
	\be\label{u4q:example}
	\begin{split}
		\Qu^{(4,q)}  =&	
		\sum_{k=0}^q{q \choose k} \Big[ 	\xi_{4,0,0,1}^{(0,q-k)} \Qu^{(0, k+1)}+ \xi_{4,0,0,2}^{(0,q-k)} \Qu^{(0, k+2)}
		+2\Qb^{(0,q-k)}\Qu^{(0, k+3)} 
		+  \xi_{4,0,1,0}^{(0,q-k)} \Qu^{(1, k)} 
		\\
		& +  \xi_{4,0,1,1}^{(0,q-k)} \Qu^{(1, k+1)} +2\Qa^{(0,q-k)}\Qu^{(1, k+2)}\Big]	+\Qu^{(0, q+4)}+ \sum_{k=0}^q{q \choose k} \Big[ (\Qa^2 - 2\Qa^{(1, 0)})^{(0,q-k)}\psi^{(0,k)} \\
		& - \Qa^{(0,q-k)}\psi^{(1, k)}- \Qb^{(0,q-k)}\psi^{(0, k+1)}\Big] +\psi^{(2, q)} - \psi^{(0, q+2)}, \quad  0 \le q \le M-4,
	\end{split}
	\ee
	where
	\be\label{u4q:xi}
	\begin{split}
		& \xi_{4,0,0,1}=2\Qa^{(1, 0)}\Qb-\Qa^2\Qb  + \Qb^{(0, 1)}\Qb + \Qa\Qb^{(1, 0)} + \Qb^{(0, 2)} - \Qb^{(2, 0)}, \qquad  \qquad
		\xi_{4,0,0,2}=2\Qa^{(1, 0)}-\Qa^2 + \Qb^2  + 2\Qb^{(0, 1)},\\
		& \xi_{4,0,1,0}= 3\Qa\Qa^{(1, 0)}-\Qa^3  + \Qa^{(0, 1)}\Qb + \Qa^{(0, 2)} - \Qa^{(2, 0)},\qquad \hspace{2.15cm} \xi_{4,0,1,1}=2\Qa\Qb + 2\Qa^{(0, 1)} - 2\Qb^{(1, 0)}.
	\end{split}
	\ee
	
	To derive a fourth-order compact 9-point FDM, we use the uniform Cartesian grid in the spatial  domain $\Omega=(0,1)^2$  as follows:
	\be \label{xiyj}
	x_i:=i h, \qquad i=0,\ldots,N_1, \qquad y_j:=j h, \qquad j=0,\ldots,N_1,  \qquad \text{and} \qquad h:=1/N_1,
	\ee
	for some integer $N_1\in \N$.
	To simplify the presentation, we adapt the following notations throughout the remainder of this paper:
	\be\label{mf:mn}
	\begin{split}
		&	\Qa^{(m,n)}:=\frac{\partial^{m+n} \Qa(x_i,y_j)}{ \partial^m x \partial^n y}, \qquad  \qquad 	\Qb^{(m,n)}:=\frac{\partial^{m+n} \Qb(x_i,y_j)}{ \partial^m x \partial^n y}, \\
		&	\Qu^{(m,n)}:=\frac{\partial^{m+n} \Qu(x_i,y_j)}{ \partial^m x \partial^n y}, \qquad  \qquad  \psi^{(m,n)}:=\frac{\partial^{m+n} \psi(x_i,y_j)}{ \partial^m x \partial^n y}.
	\end{split}
	\ee
The	Taylor approximation at the base point  $(x_i,y_j)$ and the definitions of $\ind_{M}, \ind_{M}^{1}, \ind_{M}^{2}$ in \eqref{ind} and \eqref{ind12} yield
	\begin{align}
		\Qu(x+x_i,y+y_j)
		=&
		\sum_{(m,n)\in \ind_{M}}
		\Qu^{(m,n)}\frac{x^my^n}{m!n!}+\bo(h^{M+1}) \label{u:original:0}\\
		=&
		\sum_{(m,n)\in \ind^1_{M}}
		\Qu^{(m,n)}\frac{x^my^n}{m!n!}+\sum_{(p,q)\in \ind^2_{M}}
		\Qu^{(p,q)}\frac{x^py^q}{p!q!}+\bo(h^{M+1}), \label{u:original}
	\end{align}
	where $x,y\in [-h,h]$.
	From \eqref{u:pq}, we see that
	\be\label{sum:pq:2}
	\begin{split}
		\sum_{(p,q)\in \ind^2_{M}}
		\Qu^{(p,q)}\frac{x^py^q}{p!q!}=&\sum_{(p,q)\in \ind^2_{M}} \Bigg[ \sum_{(m,n)\in \ind_{p+q}^{1} } \xi_{p,q,m,n} \Qu^{(m,n)}+\sum_{(m,n)\in \ind_{p+q-2} } \eta_{p,q,m,n}   \psi^{(m,n)} \Bigg] \frac{x^py^q}{p!q!}\\
		=&\mathop{   \sum_{p=2,q=0}}_{p+q\le M}^{M}  \Bigg[ \sum_{n=0}^{p+q} \xi_{p,q,0,n} \Qu^{(0,n)}+\sum_{n=0}^{p+q-1} \xi_{p,q,1,n} \Qu^{(1,n)}+\mathop{   \sum_{m,n=0}}_{m+n\le p+q-2}^{p+q-2} \eta_{p,q,m,n}   \psi^{(m,n)} \Bigg] \frac{x^py^q}{p!q!}.
	\end{split}
	\ee
	Clearly,
	\be\label{part:1}
	\mathop{   \sum_{p=2,q=0}}_{p+q\le M}^{M}  \Bigg[ \sum_{n=0}^{p+q} \xi_{p,q,0,n} \Qu^{(0,n)} \Bigg] \frac{x^py^q}{p!q!} =  \sum_{n=0}^{M} \Bigg[\mathop{   \sum_{p=2,q=0}}_{n \le p+q\le M}^{M} \xi_{p,q,0,n} \frac{x^py^q}{p!q!} \Bigg]\Qu^{(0,n)},
	\ee
	\be\label{part:2}
	\mathop{   \sum_{p=2,q=0}}_{p+q\le M}^{M}  \Bigg[ \sum_{n=0}^{p+q-1} \xi_{p,q,1,n} \Qu^{(1,n)}  \Bigg]  \frac{x^py^q}{p!q!} =  \sum_{n=0}^{M-1} \Bigg[\mathop{   \sum_{p=2,q=0}}_{n+1 \le p+q\le M}^{M} \xi_{p,q,1,n} \frac{x^py^q}{p!q!} \Bigg]\Qu^{(1,n)}, 
	\ee
	\be\label{part:3}
	\mathop{   \sum_{p=2,q=0}}_{p+q\le M}^{M}  \Bigg[  \mathop{   \sum_{m,n=0}}_{m+n\le p+q-2}^{p+q-2} \eta_{p,q,m,n}   \psi^{(m,n)}  \Bigg]  \frac{x^py^q}{p!q!}=\mathop{   \sum_{m,n=0}}_{m+n\le M-2}^{M-2} \Bigg[ \mathop{   \sum_{p=2,q=0}}_{m+n+2 \le p+q\le M}^{M}    \eta_{p,q,m,n}    \frac{x^py^q}{p!q!}\Bigg]\psi^{(m,n)}. 
	\ee
	Plugging \eqref{part:1}--\eqref{part:3} into \eqref{sum:pq:2} gives
	\be\label{sum:pq:2:simplied}
	\begin{split}
		\sum_{(p,q)\in \ind^2_{M}}
		\Qu^{(p,q)}\frac{x^py^q}{p!q!}=&\sum_{(m,n)\in \ind^1_{M}} \Big[\sum_{(p,q)\in \ind^2_{M} \setminus \ind^2_{m+n-1} } \xi_{p,q,m,n} \frac{x^py^q}{p!q!} \Big]	\Qu^{(m,n)}\\
		&+\sum_{(m,n)\in \ind_{M-2}} \Big[\sum_{(p,q)\in \ind^2_{M} \setminus \ind^2_{m+n+1} } \eta_{p,q,m,n} \frac{x^py^q}{p!q!} \Big]	\psi^{(m,n)}.
	\end{split}
	\ee
Then,	we substitute \eqref{sum:pq:2:simplied} into \eqref{u:original},  
	\be \label{u:GH}
	\Qu(x+x_i,y+y_j)
	=
	\sum_{(m,n)\in \ind^1_{M}}
	\Qu^{(m,n)}G_{M,m,n}(x,y)+\sum_{(m,n)\in \ind_{M-2}}	\psi^{(m,n)} H_{M,m,n}(x,y)+\bo(h^{M+1}),
	\ee
	where  $x,y\in [-h,h]$, and
	\be\label{GHmn}
	\begin{split}
		& G_{M,m,n}(x,y):=\frac{x^my^n}{m!n!}+\sum_{(p,q)\in \ind^2_{M} \setminus \ind^2_{m+n-1} } \xi_{p,q,m,n} \frac{x^py^q}{p!q!}, \qquad  H_{M,m,n}(x,y):=\sum_{(p,q)\in \ind^2_{M} \setminus \ind^2_{m+n+1} } \eta_{p,q,m,n} \frac{x^py^q}{p!q!},
	\end{split}
	\ee
	where the expressions of  $\xi_{p,q,m,n}$ and  $\eta_{p,q,m,n}$ are provided in \eqref{xi:and:eta}. Finally, we use \eqref{u:GH} to construct the fourth-order compact 9-point FDMs for \eqref{Model：linear:Elliptic} in the following \cref{subsec:elliptic}.
	\subsection{Fourth-order compact 9-point FDMs }\label{subsec:elliptic}
	Recall that $\Qu$ denotes the exact solution of \eqref{Model：linear:Elliptic}, and the grid point $(x_i,y_j)$ with the uniform mesh size $h$ is defined in \eqref{xiyj}.
	We define  $\Qu_h$ as the numerical solution computed by the FDM,  $\Qu_{i,j}=\Qu(x_i,y_j)$, and $(\Qu_h)_{i,j}$ is the value of $\Qu_h$ at $(x_i,y_j)$.
	The fourth-order compact 9-point FDM at the grid point $(x_i,y_j)$  is expressed as
	\be\label{L:h:u:h}
	\begin{split}
		h^{-2}	\mathcal{L}_h \Qu_h:= h^{-2} \sum_{k=-1}^{1} \sum_{\ell=-1}^{1} C_{k,\ell} (\Qu_h)_{i+k,j+\ell}=F_{i,j},
	\end{split}
	\ee
	with
	\be\label{Ckl:1}
	C_{k,\ell}:=\sum_{p=0}^{M+1}c_{k,\ell,p} h^p, \quad c_{k,\ell,p}\in S,\quad M\in \N_0,
	\ee
	where $S$ is defined in \eqref{Space:S}. The coefficients $\{ c_{k,\ell,p} \}_{k,\ell=-1,0,1,p=0,\dots,M+1}$ and the right-hand side $F_{i,j}$ are determined next via relations \eqref{L:h:u}--\eqref{Ckl:elliptic:special}. We restrict each $ c_{k,\ell,p}\in S$ such that the high-order partial derivatives $\Qa^{(m,n)}$ and $\Qb^{(m,n)}$ do not appear in the denominator, ensuring that each $c_{k,\ell,p}$ is well defined. We say that the set of nine elements $\{C_{k,\ell} \}_{k,\ell=-1,0,1}$ is nontrivial if $C_{k,\ell}|_{h=0}\ne0 $  for at least some $k,\ell=-1,0,1$. \\
	{\bf{ The algorithm to derive the fourth-order  compact 9-point FDMs:}}
	We define the linear operator $\mathcal{L}_h \Qu$
	\be\label{L:h:u}
	\begin{split}
		& h^{-2}	\mathcal{L}_h \Qu:= h^{-2} \sum_{k=-1}^{1} \sum_{\ell=-1}^{1} C_{k,\ell} \Qu(kh+x_i, \ell h+y_j).
	\end{split}
	\ee
	Plugging \eqref{u:GH} into  \eqref{L:h:u} 
	and replacing $M$ by $M+1$ generate 
	\begin{align}\label{Lh:u}
		h^{-2}	\mathcal{L}_h \Qu   = &	 h^{-2}  \sum_{k=-1}^{1} \sum_{\ell=-1}^{1} C_{k,\ell}	\sum_{(m,n)\in \ind^1_{M+1}}
		\Qu^{(m,n)}G_{M+1,m,n}(kh,\ell h) \notag \\
		&+h^{-2}  \sum_{k=-1}^{1} \sum_{\ell=-1}^{1} C_{k,\ell}\sum_{(m,n)\in \ind_{M-1}}	\psi^{(m,n)} H_{M+1,m,n}(kh,\ell h)+\bo(h^{M}), \notag  \\
		=& h^{-2} \sum_{(m,n)\in \ind_{M+1}^{1}} \Qu^{(m,n)} I_{M+1,m,n} +h^{-2} \sum_{(m,n) \in \ind_{	 M-1}} \psi^{(m,n)} J_{M-1,m,n}   +\bo(h^{M}),
	\end{align}
	where
	\be\label{Imn}
	\begin{split}
		I_{M+1,m,n}:=\sum_{k=-1}^{1} \sum_{\ell=-1}^{1} C_{k,\ell}G_{M+1,m,n}  (kh, \ell h),\qquad J_{M-1,m,n}:= \sum_{k=-1}^{1} \sum_{\ell=-1}^{1}	 C_{k,\ell} H_{M+1,m,n}  (kh, \ell h),
	\end{split}
	\ee	
	with $G_{M+1,m,n}(x,y)$ and  $H_{M+1,m,n}(x,y)$ defined by \eqref{GHmn}.  
	Owing to \eqref{L:h:u:h} and \eqref{Lh:u}, the local spatial truncation error
	\be\label{L:h:u:uh} 
	h^{-2}	\mathcal{L}_h (\Qu-\Qu_h)=	 h^{-2} \sum_{(m,n)\in \ind_{M+1}^{1}} \Qu^{(m,n)} I_{M+1,m,n} +h^{-2} \sum_{(m,n) \in \ind_{	 M-1}} \psi^{(m,n)} J_{M-1,m,n}   -F_{i,j} +\bo(h^{M}).
	\ee
	Let
	\be \label{Fij}
	F_{i,j}:=\text{the terms of } \Big(h^{-2} \sum_{(m,n) \in \ind_{	 M-1}} \psi^{(m,n)} J_{M-1,m,n}\Big) \text{ with degree}   \le M-1 \text{ in } h,
	\ee
	where $J_{M-1,m,n}$ is defined in \eqref{Imn} with  $H_{M+1,m,n}(x,y)$ in \eqref{GHmn}.   Then \eqref{L:h:u:uh}  leads to
	\be\label{L:h:u:uh:2} 
	h^{-2}	\mathcal{L}_h (\Qu-\Qu_h)=	 h^{-2} \sum_{(m,n)\in \ind_{M+1}^{1}} \Qu^{(m,n)} I_{M+1,m,n} +\bo(h^{M}).
	\ee
	By \eqref{ind}, \eqref{ind12}, and \eqref{GHmn}, we have that 
	smallest degrees of $h$ among the nonzero terms in $G_{M+1,m,n}  (kh, \ell h)$ and  $H_{M+1,m,n}  (kh, \ell h)$ are 0 and 2, respectively. So, \eqref{Ckl:1}, \eqref{Imn}, and \eqref{Fij} imply that  the smallest degree of $h$  among the nonzero terms in both $I_{M+1,m,n}$ and $F_{i,j}$ is 0.

	Now, if  $C_{k,\ell}$ in \eqref{Ckl:1},   $I_{M+1,m,n}$ in \eqref{Imn}, and $F_{i,j}$ in  \eqref{Fij} satisfy
	\be \label{EQ:1:explicit}
	I_{M+1,m,n}=\sum_{k=-1}^{1} \sum_{\ell=-1}^{1} C_{k,\ell}G_{M+1,m,n}  (kh, \ell h)=\bo(h^{M+2}) \quad \mbox{for all}  \quad (m,n)\in \ind_{M+1}^1, 
	\ee
	and
	\be \label{EQ:2:explicit}
	C_{0,0}|_{h=0}\ne 0,  \qquad  F_{i,j}|_{h=0}=\psi,
	\ee
	where $G_{M+1,m,n}(x,y)$ is defined in \eqref{GHmn},
	then \eqref{L:h:u:uh:2} generates
	\be\label{elliptic:order:M}
	h^{-2}	\mathcal{L}_h (\Qu- \Qu_h)=\bo(h^{M}),
	\ee
	and
	$\mathcal{L}_h \Qu_h$ approximates $	\Delta \Qu  +\Qa \Qu_x +\Qb \Qu_y   = \psi$ in \eqref{Model：linear:Elliptic} with  the consistency order $M$ at the grid point $(x_i,y_j)$.

	By the definitions of $J_{M-1,m,n}$ and $F_{i,j}$  in \eqref{Imn} and \eqref{Fij}, $F_{i,j}$ can be obtained immediately if $\{c_{k,\ell,p}\}|_{k,\ell=-1,0,1}^{p=0,M+1}$ in \eqref{Ckl:1} are fixed. So the main task of deriving the high-order FDM is to find $\{c_{k,\ell,p}\}|_{k,\ell=-1,0,1}^{p=0,M+1}$ to satisfy \eqref{EQ:1:explicit} and \eqref{EQ:2:explicit}.
	Using the symbolic calculation in Maple, the largest $M$ such that the nontrivial $\{C_{k,\ell}\}_{k,\ell=-1,0,1}$ solving the corresponding linear system of \eqref{EQ:1:explicit} with the variables   $\{c_{k,\ell,p}\}|_{k,\ell=-1,0,1}^{p=0,M+1}$ is, $M=4$. Therefore,  the highest  consistency order
	of the compact 9-point FDM based on our technique using the uniform Cartesian mesh grid for $\Delta \Qu  +\Qa \Qu_x +\Qb \Qu_y   = \psi$, is 4.  Furthermore, there  exist nine coefficients in $\{C_{k,\ell}\}_{k,\ell=-1,0,1}$  satisfying \eqref{EQ:1:explicit} with $M=4$, for any free variables in the following set:
	\be\label{free:para}
	\{c_{0,0,0}, c_{0,0,1}, c_{-1,-1,2}, c_{1,1,2}, c_{0,1,3}, c_{1,0,3}, c_{1,0,2}, c_{1,1,3} \} \cup  \{ c_{k,\ell,p} \}_{k,\ell=-1,0,1}^{p=4,5}.
	\ee
	The symbolic calculation in Maple also reveals that: if we fix $c_{0,0,0}=-10/3$, $c_{0,0,1}=c_{-1,-1,2}=c_{1,1,2}=c_{0,1,3}=c_{1,0,3}=0$,  $\{ c_{k,\ell,p} \}_{k,\ell=-1,0,1,p=4,5}=\{0\}$, $c_{1,0,2}=(\Qa^2 +\Qa\Qb +  \Qa^{(0, 1)} +\Qa^{(1, 0)}+\Qb^{(1, 0)} - \Qb^{(0, 1)})/12$, and $c_{1,1,3}=\Qa^{(0, 1)}\Qb /24$ in \eqref{free:para}, then there exist the unique $\{ C_{k,\ell}\}_{k,\ell=-1,0,1}$ and $F_{i,j}$ fulfilling \eqref{EQ:1:explicit} and \eqref{EQ:2:explicit} with $M=4$ as follows  (every free variable  $c_{k,\ell,p}$ in \eqref{free:para} is chosen to ensure \eqref{EQ:2:explicit} and  make the expression of  $C_{k,\ell}$  concise):
	\begin{align}\label{Ckl:elliptic:general}
		& C_{-1,-1}= \tfrac{1}{6} - \tfrac{r_1h}{12}+ ( r_3\Qa - (2\Qa^{(0, 1)}+ \Qa^{(1, 0)})\Qb +  \Delta r_1 ) \tfrac{h^3}{24},\notag \\
		&  C_{-1,0}=  \tfrac{2}{3} - \tfrac{\Qa h}{3}+(\Qa^2 + \Qa\Qb + r_2 +r_3) \tfrac{h^2}{12} + (r_2\Qb - r_3\Qa -\Delta r_1 ) \tfrac{h^3}{12}, \notag\\
		& C_{-1,1}= \tfrac{1}{6} - \tfrac{r_5h}{12} -(\Qa\Qb +r_4) \tfrac{h^2}{12} + (\Qa\Qb^{(1, 0)}-r_2\Qb  + \Delta \Qb)\tfrac{h^3}{24}, \notag\\
		& C_{0,-1}=\tfrac{2}{3}-  \tfrac{\Qb h}{3}+(\Qa\Qb + \Qb^2 +r_6+r_7)  \tfrac{h^2}{12}  + (r_2\Qb -r_3\Qa  - \Delta r_1  )   \tfrac{h^3}{12},\notag \\
		& C_{0,0}=-\tfrac{10}{3} -(\Qa^2 +\Qa\Qb +\Qb^2 +r_4)  \tfrac{h^2}{6}  + (r_3\Qa  - r_2\Qb +\Delta r_1  )   \tfrac{h^3}{12}, \\
		& C_{0,1}=\tfrac{2}{3}+  \tfrac{\Qb h}{3}+(\Qa\Qb + \Qb^2 + r_6+r_7)  \tfrac{h^2}{12},\notag\\
		& C_{1,-1}= \tfrac{1}{6} + \tfrac{r_5 h}{12} -(\Qa\Qb +r_4) \tfrac{h^2}{12} + (\Delta \Qa-\Qa\Qb^{(0, 1)} )\tfrac{h^3}{24}, \notag\\
		& C_{1,0}=\tfrac{2}{3}+  \tfrac{\Qa h}{3}+(\Qa^2 +\Qa\Qb +  r_2+r_3)  \tfrac{h^2}{12},\notag\\
		& C_{1,1}= \tfrac{1}{6} + \tfrac{r_1 h}{12} +\Qa^{(0, 1)}\Qb \tfrac{h^3}{24},\notag
	\end{align}
	where 
	\be \label{r1:r7}
	\begin{split}
		& r_1=\Qa+\Qb, \qquad r_2=\Qa^{(0, 1)} +\Qa^{(1, 0)}, \qquad r_3=\Qb^{(1, 0)} - \Qb^{(0, 1)},\qquad r_4=\Qa^{(0, 1)} +\Qb^{(1, 0)},\\
		& r_5=\Qa-\Qb, \qquad r_6=\Qa^{(0, 1)} -\Qa^{(1, 0)}, \qquad r_7=\Qb^{(1, 0)} + \Qb^{(0, 1)},
	\end{split}
	\ee
	and
	\be \label{Fij:elliptic:general}
	F_{i,j}:=\psi-( (\Qa^{(1, 0)} + \Qb^{(0, 1)})\psi - \Qa\psi^{(1, 0)} - \Qb\psi^{(0, 1)} - \Delta \psi ) \tfrac{h^2}{12}.
	\ee

	In particular, if $\Qa=\Qb$ (e.g, if $\kappa=$ constant and $\alpha(u)=\beta(u)=u^2/2$, then \eqref{QuQaQb} yields the steady viscous Burgers equation), we obtain a more concise expression by the following steps:
	\eqref{EQ:1:explicit} and \eqref{EQ:2:explicit} with $\Qa=\Qb$, $M=4$, $c_{0,0,0}=-10/3$, $c_{-1,1,1}=c_{0,0,1}=c_{1,-1,1}=c_{-1,-1,2}=c_{1,1,2}=c_{0,1,3}=c_{1,0,3}=c_{1,1,3}=0$,  $\{ c_{k,\ell,p} \}_{k,\ell=-1,0,1,p=4,5}=\{0\}$, and $c_{1,0,2}=(\Qa^2+\Qa^{(1,0)})/6$  result in
	\be\label{Ckl:elliptic:special}
	\begin{split}
		& C_{-1,-1}=  \tfrac{1}{6} - \tfrac{\Qa h}{6}+  \tfrac{\Delta \Qa h^3}{12}, \qquad \qquad C_{-1,0}=  \tfrac{2}{3} - \tfrac{\Qa h}{3}+ (\Qa^2+\Qa^{(1,0)})\tfrac{ h^2}{6}- \tfrac{\Delta \Qa h^3}{6},   \\
		&  C_{-1, 1}=  \tfrac{1}{6} - (\Qa^2+\Qa^{(0, 1)} +\Qa^{(1, 0)})\tfrac{ h^2}{12}+  \tfrac{\Delta \Qa h^3}{24},\qquad \qquad  C_{0,-1}=  \tfrac{2}{3} - \tfrac{\Qa h}{3}+ (\Qa^2+\Qa^{(0,1)})\tfrac{ h^2}{6}- \tfrac{\Delta \Qa h^3}{6},  \\
		&   C_{0, 0}=  -\tfrac{10}{3} - (3\Qa^2+\Qa^{(0, 1)} +\Qa^{(1, 0)})\tfrac{ h^2}{6}+  \tfrac{\Delta \Qa h^3}{6},\qquad \qquad C_{0,1}=  \tfrac{2}{3} + \tfrac{\Qa h}{3}+ (\Qa^2+\Qa^{(0,1)})\tfrac{ h^2}{6},\\  & C_{1,-1}=C_{-1,1}, \qquad \qquad C_{1,0}=  \tfrac{2}{3} + \tfrac{\Qa h}{3}+ (\Qa^2+\Qa^{(1,0)})\tfrac{ h^2}{6}, \qquad \qquad C_{1,1}= \tfrac{1}{6} + \tfrac{\Qa h}{6},
	\end{split}
	\ee
	and $F_{i,j}$ is obtained by replacing $\Qb$ by $\Qa$ in \eqref{Fij:elliptic:general}.

	Note that  the nine coefficients in $\{ C_{k,\ell}\}_{k,\ell=-1,0,1}$  satisfying \eqref{EQ:1:explicit}--\eqref{EQ:2:explicit} with $M=4$ are not unique. We choose each free variable $c_{k,\ell,p}$ to endow every $C_{k,\ell}$ with the simplest explicit expression in \eqref{Ckl:elliptic:general} and \eqref{Ckl:elliptic:special}.  The coefficients in \eqref{Ckl:elliptic:general} with ${\textup \Qa}={\textup\Qb}$ can also give the left-hand side of a fourth-order FDM for $\Delta {\textup \Qu}  +{\textup \Qa} {\textup \Qu}_x +{\textup \Qa} {\textup \Qu}_y   = \psi$, but the corresponding nine coefficients $\{C_{k,\ell}\}_{k,\ell=-1,0,1}$ are not of the simplest forms based on our technique in \cref{subsec:tech}.

	Summarizing the above relations \eqref{L:h:u:h}--\eqref{Ckl:elliptic:special},  the following  \cref{thm:elliptic:special,thm:elliptic:general} provide the fourth-order compact 9-point FDMs for $	\Delta \Qu  +\Qa \Qu_x +\Qa \Qu_y   = \psi$ and $\Delta \Qu  +\Qa \Qu_x +\Qb \Qu_y   = \psi$, respectively. 
	\begin{theorem}\label{thm:elliptic:special}
		Assume $\kappa, u, \alpha=\beta, f$ are all smooth, and $\kappa_x=\kappa_y$ in \eqref{Model：Original:Elliptic},  $\mathcal{L}_h {\textup \Qu}_h$ is defined in \eqref{L:h:u:h}, 
the	nine coefficients $\{ C_{k,\ell}\}_{k,\ell=-1,0,1}$ are defined in \eqref{Ckl:elliptic:special}, the right-hand side $F_{i,j}$ is obtained by replacing ${\textup \Qb}$ by ${\textup \Qa}$ in \eqref{Fij:elliptic:general}, and  all the functions in $\mathcal{L}_h {\textup \Qu}_h$ are evaluated at the grid point $(x_i,y_j)$ in \eqref{xiyj}.  Then $h^{-2}	\mathcal{L}_h  {\textup \Qu}_h$
		approximates $	\Delta {\textup \Qu}  +{\textup \Qa} {\textup \Qu}_x +{\textup \Qb} {\textup \Qu}_y   = \psi$ in \eqref{Model：linear:Elliptic} with  the fourth-order of consistency at  $(x_i,y_j)$ for the case ${\textup \Qa}={\textup \Qb}$, where  $\psi, {\textup \Qa},{\textup \Qb}, {\textup \Qu}$ are defined in \eqref{a:b:psi} and \eqref{QuQaQb}.	
	\end{theorem}	
	\begin{proof}
		From the derivations of $\{ C_{k,\ell}\}_{k,\ell=-1,0,1}$ and $F_{i,j}$ in identities \eqref{L:h:u:h}--\eqref{Ckl:elliptic:special}, this result can be proved directly. For the readers' convenience and to make our FDM more rigorous, we also provide the proof in the standard perspective as follows.  
		From \eqref{Model：linear:Elliptic} with $\Qa=\Qb$ and \eqref{mf:mn}, we have
		\be\label{first:eq:proof}
		\Delta \Qu+(\Qu^{(1, 0)}+\Qu^{(0, 1)} )\Qa = \psi.
		\ee
		Taking the first-order partial derivatives of \eqref{first:eq:proof} with respect to $x$ and $y$ yields 
		\be\label{eq:1x}
		\Delta \Qu^{(1,0)}+(\Qu^{(2, 0)}+\Qu^{(1, 1)} )\Qa+(\Qu^{(1, 0)}+\Qu^{(0, 1)} )\Qa^{(1,0)} = \psi ^{(1,0)},
		\ee
		and
		\be\label{eq:1y}
		\Delta \Qu^{(0,1)}+(\Qu^{(1, 1)}+\Qu^{(0, 2)} )\Qa+(\Qu^{(1, 0)}+\Qu^{(0, 1)} )\Qa^{(0,1)}  = \psi^{(0,1)},
		\ee
		so that
		\eqref{eq:1x}+\eqref{eq:1y} gives
		\be\label{second:eq:proof}
		\Delta (\Qu^{(1,0)}+\Qu^{(0,1)})=-(\Delta \Qu+2\Qu^{(1, 1)} )\Qa-(\Qu^{(1, 0)}+\Qu^{(0, 1)} )(\Qa^{(1,0)}+\Qa^{(0,1)})+\psi ^{(1,0)}+\psi^{(0,1)}. 
		\ee
	Now, we	take the first-order partial derivatives of \eqref{eq:1x} and \eqref{eq:1y}  with respect to $x$ and $y$, respectively,
		\be\label{eq:2x}
		\Qu^{(4, 0)} +\Qu^{(2, 2)}+(\Qu^{(3, 0)}+\Qu^{(2, 1)} )\Qa+2(\Qu^{(2, 0)}+\Qu^{(1, 1)} )\Qa^{(1,0)}+(\Qu^{(1, 0)}+\Qu^{(0, 1)} )\Qa^{(2,0)} = \psi ^{(2,0)},
		\ee
		\be\label{eq:2y}
		\Qu^{(2, 2)} +\Qu^{(0, 4)}+(\Qu^{(1, 2)}+\Qu^{(0, 3)} )\Qa+2(\Qu^{(1, 1)}+\Qu^{(0, 2)} )\Qa^{(0,1)}+(\Qu^{(1, 0)}+\Qu^{(0, 1)} )\Qa^{(0,2)}  = \psi^{(0,2)}, 
		\ee
		hence
		\eqref{eq:2x}+\eqref{eq:2y} gives
		\be\label{third:eq:proof}
		\begin{split}
			\Delta^2 \Qu+(\Qu^{(1, 0)}+\Qu^{(0, 1)} )\Delta \Qa+2\Qu^{(1, 1)} (\Qa^{(0,1)}+\Qa^{(1,0)})=&-  \Qa\Delta (\Qu^{(1,0)}+\Qu^{(0,1)})  -2\Qu^{(2, 0)}\Qa^{(1,0)}\\
			&-2\Qu^{(0, 2)} \Qa^{(0,1)}  +  \Delta \psi.
		\end{split}
		\ee
		Next, we
		plug $C_{k,\ell}$ of \eqref{Ckl:elliptic:special} into $\mathcal{L}_h \Qu$ defined in \eqref{L:h:u}, and apply \eqref{u:original:0} with $M=5$. After a direct simplification, we obtain
		\[
		\begin{split} 
			\mathcal{L}_h \Qu=&	\sum_{k=-1}^1 \sum_{\ell=-1}^1  C_{k,\ell}  \Qu(kh+x_i, \ell h+y_j)	\\
			=& h^2[\Delta \Qu+\Qa(\Qu^{(1, 0)}+\Qu^{(0, 1)} )]+\tfrac{ h^4}{12} [ \Qa^2 \Delta \Qu +2\Qa^2 \Qu^{(1, 1)} + 2\Qa(\Delta \Qu^{(1,0)} + \Delta \Qu^{(0,1)}) \\
			&  + \Delta^2 \Qu  + (\Qu^{(0, 1)} + \Qu^{(1, 0)})\Delta \Qa+ 2(\Qa^{(0, 1)} + \Qa^{(1, 0)})\Qu^{(1, 1)} + (\Qu^{(2, 0)} - \Qu^{(0, 2)})(\Qa^{(1, 0)}-\Qa^{(0, 1)} ) ]+\bo(h^6).
		\end{split}	
		\]
		Using \eqref{first:eq:proof}, \eqref{second:eq:proof}, and \eqref{third:eq:proof}, after algebraic manipulation, we have
		\begin{align*} 
			\mathcal{L}_h \Qu
			=& h^2 \psi+\tfrac{ h^4}{12} [\Qa^2 \Delta \Qu+2\Qa^2 \Qu^{(1, 1)} -2\Qa^2 \Delta \Qu -4\Qa^2\Qu^{(1, 1)}-2\Qa (\Qu^{(1, 0)}+\Qu^{(0, 1)} )(\Qa^{(1,0)}+\Qa^{(0,1)})\\
			& +2\Qa(\psi ^{(1,0)}+\psi^{(0,1)}) -  \Qa\Delta (\Qu^{(1,0)}+\Qu^{(0,1)})  -2\Qu^{(2, 0)}\Qa^{(1,0)}-2\Qu^{(0, 2)} \Qa^{(0,1)}  \\
			& +  \Delta \psi+ (\Qu^{(2, 0)} - \Qu^{(0, 2)})(\Qa^{(1, 0)}-\Qa^{(0, 1)} )] +\bo(h^6),\hspace{6cm} \text{by}\ \eqref{first:eq:proof}, \eqref{second:eq:proof}, \eqref{third:eq:proof},\\
			=& h^2 \psi+\tfrac{ h^4}{12} [-\Qa^2 \Delta \Qu -2\Qa^2\Qu^{(1, 1)} -2\Qa (\Qu^{(1, 0)}+\Qu^{(0, 1)} )(\Qa^{(1,0)}+\Qa^{(0,1)})+2\Qa(\psi ^{(1,0)}+\psi^{(0,1)})\\
			&  +\Qa^2 \Delta \Qu+2\Qa^2\Qu^{(1, 1)}+\Qa(\Qu^{(1, 0)}+\Qu^{(0, 1)} )(\Qa^{(1,0)}+\Qa^{(0,1)})-\Qa(\psi ^{(1,0)}+\psi^{(0,1)})   \\
			&  -2\Qu^{(2, 0)}\Qa^{(1,0)}-2\Qu^{(0, 2)} \Qa^{(0,1)}+  \Delta \psi+ (\Qu^{(2, 0)} - \Qu^{(0, 2)})(\Qa^{(1, 0)}-\Qa^{(0, 1)} )] +\bo(h^6),\hspace{2.2cm} \text{by}\ \eqref{second:eq:proof},\\
			=& h^2 \psi+\tfrac{ h^4}{12} [ -\Qa (\Qu^{(1, 0)}+\Qu^{(0, 1)} )(\Qa^{(1,0)}+\Qa^{(0,1)})+\Qa(\psi ^{(1,0)}+\psi^{(0,1)})   \\
			&  -2\Qu^{(2, 0)}\Qa^{(1,0)}-2\Qu^{(0, 2)} \Qa^{(0,1)}+  \Delta \psi+ (\Qu^{(2, 0)} - \Qu^{(0, 2)})(\Qa^{(1, 0)}-\Qa^{(0, 1)} )] +\bo(h^6),\hspace{2.23cm} \text{by simplification},\\
			=& h^2 \psi+\tfrac{ h^4}{12} [ -\Qa (\Qu^{(1, 0)}+\Qu^{(0, 1)} )(\Qa^{(1,0)}+\Qa^{(0,1)})+\Qa(\psi ^{(1,0)}+\psi^{(0,1)})   \\
			&  -\Qu^{(2, 0)}(\Qa^{(1, 0)}+\Qa^{(0, 1)} )-\Qu^{(0, 2)} (\Qa^{(1, 0)}+\Qa^{(0, 1)} ) +  \Delta \psi]+\bo(h^6),\hspace{4.3cm} \text{by simplification},\\
			=& h^2 \psi+\tfrac{ h^4}{12} [ -\Qa (\Qu^{(1, 0)}+\Qu^{(0, 1)} )(\Qa^{(1,0)}+\Qa^{(0,1)})+\Qa(\psi ^{(1,0)}+\psi^{(0,1)})   \\
			&  +(\Qa^{(1, 0)}+\Qa^{(0, 1)} )(\Qa(\Qu^{(1, 0)}+\Qu^{(0, 1)} ) - \psi) +  \Delta \psi]+\bo(h^6), \hspace{5.05cm} \text{by}\ \eqref{first:eq:proof}, \\
			=& h^2 \psi+\tfrac{ h^4}{12} [\Qa(\psi ^{(1,0)}+\psi^{(0,1)})   -(\Qa^{(1, 0)}+\Qa^{(0, 1)} ) \psi +  \Delta \psi]+\bo(h^6),\hspace{4.24cm}  \text{by simplification}.
		\end{align*} 
		Finally, \eqref{L:h:u:h} and \eqref{Fij:elliptic:general} with $\Qb=\Qa$ complete the argument.
	\end{proof}
	\begin{theorem}\label{thm:elliptic:general}
		Assume $\kappa, u, \alpha, \beta, f$ are all smooth  in \eqref{Model：Original:Elliptic},  $\mathcal{L}_h  {\textup \Qu}_h$ is defined in \eqref{L:h:u:h}, the
	nine coefficients $\{ C_{k,\ell}\}_{k,\ell=-1,0,1}$ are defined in \eqref{Ckl:elliptic:general}--\eqref{r1:r7}, the right-hand side  $F_{i,j}$ is defined in \eqref{Fij:elliptic:general}, and  all the functions in $\mathcal{L}_h  {\textup \Qu}_h$ are evaluated at the grid point $(x_i,y_j)$ in \eqref{xiyj}.  Then $h^{-2}	\mathcal{L}_h  {\textup \Qu}_h$
		approximates $	\Delta {\textup \Qu}  +{\textup \Qa} {\textup \Qu}_x +{\textup \Qb} {\textup \Qu}_y   = \psi$ in \eqref{Model：linear:Elliptic} with  the fourth-order of consistency at  $(x_i,y_j)$, where  $\psi, {\textup \Qa},{\textup \Qb}, {\textup \Qu}$ are defined in \eqref{a:b:psi} and \eqref{QuQaQb}.	
	\end{theorem}	
	\begin{proof}
		The proof follows similar arguments as in \cref{thm:elliptic:special}.
	\end{proof}

	From \eqref{EQ:1:explicit}--\eqref{free:para}, there  exist the nine coefficients $\big\{C_{k,\ell}=\sum_{p=0}^{M+1}c_{k,\ell,p} h^p : k,\ell=-1,0,1 \big\}$  fulfilling  \eqref{EQ:1:explicit} with $M=4$ for any free variables $c_{k,\ell,p}$ in \eqref{free:para}. Hence, we  use the free variables from \eqref{free:para} to reduce the pollution effects in the following \cref{sec:FDM:reduced}. 
	\subsection{Reduction of the pollution effects of $\bo(h^4)$ and $\bo(h^5)$}\label{sec:FDM:reduced}
	Recall that  $M=4$ is the largest integer such that the nontrivial $\{C_{k,\ell}\}_{k,\ell=-1,0,1}$ satisfying \eqref{EQ:1:explicit}, i.e.,
	the maximum  consistency order
	of the FDM in \cref{thm:elliptic:general} for $\Delta \Qu  +\Qa \Qu_x +\Qb \Qu_y   = \psi$ in \eqref{Model：linear:Elliptic}, is 4. In this section we use the free variables $c_{k,\ell,p}$ in \eqref{free:para} 
	to reduce the pollution effects of $\bo(h^4)$ and $\bo(h^5)$. Let us choose $M=6$ in \eqref{Ckl:1}, \eqref{Imn}, and \eqref{Fij}, so that
	\be\label{Ckl:degree:7}
	C_{k,\ell}:=\sum_{p=0}^{7}c_{k,\ell,p} h^p, \quad  c_{k,\ell,p}\in S, \qquad 	I_{7,m,n}:=\sum_{k=-1}^{1} \sum_{\ell=-1}^{1} C_{k,\ell}G_{7,m,n}  (kh, \ell h),
	\ee
	\be\label{Fij:order6}
	F_{i,j}=\text{the terms of } \Big( h^{-2} \sum_{(m,n) \in \ind_{	 5}} \psi^{(m,n)} \sum_{k=-1}^{1} \sum_{\ell=-1}^{1}	 C_{k,\ell} H_{7,m,n}  (kh, \ell h)\Big) \text{ with degree}   \le 5 \text{ in } h,
	\ee
	where $G_{7,m,n}(x,y)$ and  $H_{7,m,n}(x,y)$ are defined in \eqref{GHmn}, and $S$ is defined in \eqref{Space:S}. To maintain the fourth-order of consistency, and  reduce the truncation errors of $\bo(h^4)$ and $\bo(h^5)$, we
	choose $M=4$, and replace $I_{5,m,n}$ by  $I_{7,m,n}$ and $G_{5,m,n}$ by  $G_{7,m,n}$  in \eqref{EQ:1:explicit}--\eqref{EQ:2:explicit}. Using the symbolic computation from Maple, we verify that there exist $C_{k,\ell}=\sum_{p=0}^{7}c_{k,\ell,p} h^p$ with $k,\ell=-1,0,1$  satisfying
	\be\label{QE:order:4:reduce}
	I_{7,m,n}=\sum_{k=-1}^{1} \sum_{\ell=-1}^{1} C_{k,\ell}G_{7,m,n}  (kh, \ell h)=\bo(h^{6}) \quad \mbox{for all}  \quad (m,n)\in \ind_{5}^1,\quad C_{0,0}|_{h=0}\ne 0,  \quad F_{i,j}|_{h=0}=\psi,
	\ee
	for any free variables in the following set
	\be\label{free:parameter:1}
	\begin{split}
		& \{c_{-1,-1,p}\}_{p=6,7} \cup \{c_{-1,0,p}\}_{p=5,6,7}  \cup \{c_{-1,1,p}\}_{p=4,\dots,7} \cup \{c_{0,-1,p}\}_{p=5,6,7}  \\
		& \cup \{c_{0,0,p}\}_{p=4,\dots,7} \cup \{c_{0,1,p}\}_{p=3,\dots,7}   \cup \{c_{1,-1,p}\}_{p=3,\dots,7} \cup \{c_{1,0,p}\}_{p=2,\dots,7}  \cup \{c_{1,1,p}\}_{p=1,\dots,7}.
	\end{split}
	\ee
	Furthermore, we also observe that $\{ C_{k,\ell}\}_{k,\ell=-1,0,1}$ solving \eqref{QE:order:4:reduce} leads to
	\be\label{reduce:together}
	\begin{split}
		\sum_{(m,n)\in \ind_{7}^1}\Qu^{(m,n)}I_{7,m,n}&=\sum_{(m,n)\in \ind_{7}^1}\Qu^{(m,n)}\sum_{k=-1}^{1} \sum_{\ell=-1}^{1} C_{k,\ell}G_{7,m,n}  (kh, \ell h)\\
		&=\tfrac{h^6}{90}(\Qa^{(0,1)}-\Qb^{(1,0)})\Qu^{(1,3)}+C_6h^6+C_7h^7+\bo(h^8),
	\end{split}
	\ee
	where $C_6$ and $C_7$ depend on the free variables in \eqref{free:parameter:1}. To reduce the truncation errors of $\bo(h^4)$ and $\bo(h^5)$, we consider
	\be\label{QE:order:4:reduce:2}
	\begin{split}
		& \sum_{k=-1}^{1} \sum_{\ell=-1}^{1} C_{k,\ell}G_{7,m,n}  (kh, \ell h)=\bo(h^{8}) \quad \mbox{for all}  \quad (m,n)\in \ind_{7}^1 \setminus \{(1,3)\},\\
		& \sum_{k=-1}^{1} \sum_{\ell=-1}^{1} C_{k,\ell}G_{7,1,3}  (kh, \ell h)-\tfrac{h^6}{90}(\Qa^{(0,1)}-\Qb^{(1,0)})=\bo(h^{8}),\quad C_{0,0}|_{h=0}\ne 0,  \quad F_{i,j}|_{h=0}=\psi.
	\end{split}
	\ee
	Again using the symbolic calculation from Maple, there exist $C_{k,\ell}=\sum_{p=0}^{7}c_{k,\ell,p} h^p$ with $k,\ell=-1,0,1$ satisfying \eqref{QE:order:4:reduce:2} for any free variables in the following set:
	\be\label{free:parameter:2}
	\begin{split}
		&\{c_{-1,0,7}\} \cup \{c_{-1,1,p}\}_{p=6,7} \cup \{c_{0,-1,7}\}  \cup \{c_{0,0,p}\}_{p=6,7} \cup \{c_{0,1,p}\}_{p=5,6,7} \cup \{c_{1,-1,p}\}_{p=5,6,7}  \\
		&\  \cup \{c_{1,0,p}\}_{p=4,\dots,7} \cup \{c_{1,1,p}\}_{p=2,\dots,7}.
	\end{split}
	\ee
	Furthermore,  $C_{k,\ell}=\sum_{p=0}^{7}c_{k,\ell,p} h^p$ fulfilling \eqref{QE:order:4:reduce:2}  yields
	\be\label{ckl0:elliptic} 
	\begin{split} 
		& c_{-1,-1,0}=c_{-1,1,0}=c_{1,-1,0}=c_{1,1,0}=\tfrac{1}{6}, \qquad  c_{-1,0,0}=c_{1,0,0}=c_{0,-1,0}=c_{0,1,0}=\tfrac{2}{3}, \qquad  c_{0,0,0}=-\tfrac{10}{3}.	
	\end{split}	
	\ee
	Similarly to \eqref{L:h:u:uh}--\eqref{elliptic:order:M},  $C_{k,\ell}=\sum_{p=0}^{7}c_{k,\ell,p} h^p$ meeting \eqref{QE:order:4:reduce:2} implies
	\be\label{elliptic:order:6}
	h^{-2}	\mathcal{L}_h (\Qu- \Qu_h)=\tfrac{h^4}{90}(\Qa^{(0,1)}-\Qb^{(1,0)})\Qu^{(1,3)}+\bo(h^{6}),
	\ee
	and
	$\mathcal{L}_h \Qu_h$ approximates $	\Delta \Qu  +\Qa \Qu_x +\Qb \Qu_y   = \psi$ with  the fourth-order of consistency at $(x_i,y_j)$, where $	\mathcal{L}_h \Qu$ is defined in \eqref{L:h:u}, $	\mathcal{L}_h \Qu_h$ is defined in \eqref{L:h:u:h}, and $F_{i,j}$ is defined in \eqref{Fij:order6}. In our numerical examples, we set all free variables in \eqref{free:parameter:2} to zero, so $C_{k,\ell}=\sum_{p=0}^{7}c_{k,\ell,p} h^p$ with $k,\ell=-1,0,1$ can be uniquely determined by solving \eqref{QE:order:4:reduce:2}. In summary, we have the FDM with the reduced pollution effects of $\bo(h^4)$ and $\bo(h^5)$ in the following theorem:
	\begin{theorem}\label{thm:elliptic:general:reduce}
		Assume $\kappa, u, \alpha, \beta, f$ are all smooth in \eqref{Model：Original:Elliptic},  $\mathcal{L}_h  {\textup \Qu}_h$ is defined in \eqref{L:h:u:h}, the nine coefficients
		$C_{k,\ell}=\sum_{p=0}^{7}c_{k,\ell,p} h^p$ with $k,\ell=-1,0,1$ are uniquely determined by solving \eqref{QE:order:4:reduce:2} with all free variables  being zero in \eqref{free:parameter:2}, the right-hand side  $F_{i,j}$ is defined in \eqref{Fij:order6}, and  all the functions in $\mathcal{L}_h  {\textup \Qu}_h$ are evaluated at the grid point $(x_i,y_j)$ in \eqref{xiyj}.  Then $h^{-2}	\mathcal{L}_h  {\textup \Qu}_h$
		approximates $	\Delta  {\textup \Qu}  +{\textup \Qa}  {\textup \Qu}_x +{\textup \Qb}  {\textup \Qu}_y   = \psi$ with  the fourth-order of consistency at $(x_i,y_j)$, and the truncation error  $\tfrac{h^4}{90}({\textup \Qa}^{(0,1)}-{\textup \Qb}^{(1,0)}) {\textup \Qu}^{(1,3)}+\bo(h^{6})$.
	\end{theorem}	
	\begin{proof}
		The proof follows directly from \eqref{Ckl:degree:7}--\eqref{elliptic:order:6}.
	\end{proof}
We note that the theoretical results of the discrete maximum principle in \citep{Feng2022,FHM2023,Feng2024,LiIto2001,LiZhang2020} imply that, if the nine coefficients $\{ C_{k,\ell}\}_{k,\ell=-1,0,1}$  satisfy 
	\be\label{sign:and:sum}
	\begin{split}
	& C_{0,0}>0, \qquad C_{k,\ell}\le 0, \quad \text{if} \quad (k,\ell)\neq (0,0), \quad \text{the sign condition}; \\ 
	& \sum_{k=-1}^1\sum_{\ell=-1}^1C_{k,\ell}\ge 0,  \quad \text{the sum condition};
	\end{split}
	\ee
	then the corresponding FDM with the Dirichlet boundary condition yields a numerical solution which satisfies the discrete maximum principle, and generates an M-matrix (a real square matrix with the non-positive off-diagonal entries and the positive diagonal entries such that all row sums are non-negative with at least one row sum being positive).

	\begin{proposition}\label{prop:1:ckl0}
		All $\{C_{k,\ell}\}_{k,\ell=-1,0,1}$ in \cref{thm:elliptic:special,thm:elliptic:general,thm:elliptic:general:reduce}
		with $C_{k,\ell}|_{h=0}=c_{k,\ell,0}$ satisfy
		\be\label{prop:ckl0:elliptic} 
		\begin{split} 
			& c_{\pm1,\pm1,0}=c_{\pm1,\mp1,0}=\tfrac{1}{6}, \qquad c_{\pm1,0,0}=c_{0,\pm1,0}=\tfrac{2}{3},\qquad  c_{0,0,0}=-\tfrac{10}{3},	\qquad \sum_{k=1}^{-1}\sum_{\ell=1}^{-1} c_{k,\ell,0}=0.
		\end{split}	
		\ee
		That is, each FDM in \cref{thm:elliptic:special,thm:elliptic:general,thm:elliptic:general:reduce} satisfies the discrete maximum principle, and generates an M-matrix,  when the mesh size $h$ is sufficiently small.
	\end{proposition}
	\begin{proof}
	The expression	\eqref{prop:ckl0:elliptic}   stems from \eqref{Ckl:elliptic:general}, \eqref{Ckl:elliptic:special}, and \eqref{ckl0:elliptic}. By $C_{k,\ell}|_{h=0}=c_{k,\ell,0}$ and \eqref{L:h:u:h}, we see that 
	\be \label{negative:Lh}
	-h^{-2}	\mathcal{L}_h := h^{-2} \sum_{k=-1}^{1} \sum_{\ell=-1}^{1} (-C_{k,\ell}) (\Qu_h)_{i+k,j+\ell}=-F_{i,j},
	\ee
	  generates an M-matrix, and the numerical solution computed by \eqref{negative:Lh} satisfies the discrete maximum principle,  if $h$ is sufficiently small. 
	\end{proof}

	Now, we use the fourth-order compact 9-point FDMs in \cref{thm:elliptic:special,thm:elliptic:general,thm:elliptic:general:reduce}, and the iteration method \eqref{Model：iteration:Elliptic}, to numerically solve \eqref{Model：Original:Elliptic} in the following \cref{algm1}. Recall that $u$ denotes the exact solution of \eqref{Model：Original:Elliptic}.
	Now, we define  $u_h$ as the numerical solution computed by the FDMs in \cref{thm:elliptic:special,thm:elliptic:general,thm:elliptic:general:reduce} and the iteration method \eqref{Model：iteration:Elliptic}.
	\begin{algorithm}\label{algm1}
		We use the fourth-order compact 9-point  FDMs in \cref{thm:elliptic:special,thm:elliptic:general,thm:elliptic:general:reduce} 
		to solve \eqref{Model：linear:Elliptic} to obtain ${\textup \Qu}_h$ with the iteration method \eqref{Model：iteration:Elliptic} after 40 iterations.  Then $u_h={\textup \Qu}_h$.
	\end{algorithm}
	In \cref{algm1},  we consider $u_0= 0$ in \eqref{Model：iteration:Elliptic} as the initial guess in the iteration method, and we observe that 40 iterations are sufficient for convergence. For the  time-dependent nonlinear convection-diffusion equation \eqref{Model：Original:Parabolic} in \cref{sec:parabolic}, we use $u^{n+1/2}_0=u^{n}$ in \eqref{iteration:CN} for the CN method, $u^{n+3}_0=u^{n+2}$ in \eqref{iteration:BDF3} for the BDF3 method, and $u^{n+4}_0=u^{n+3}$  in \eqref{iteration:BDF4} for the BDF4 method, so we only need 20 iterations in \cref{algm2,algm3,algm4}.

	\subsection{Approximations of the high-order partial derivatives of $\textnormal{\Qa}^{(m,n)}$, $\textnormal{\Qb}^{(m,n)}$, and $\psi^{(m,n)}$ }\label{amn:bmn}
	
For	the explicit expressions of $\{C_{k,\ell}\}_{k,\ell=-1,0,1}$ and $F_{i,j}$ in \eqref{Ckl:elliptic:general}, \eqref{Fij:elliptic:general},  and \eqref{Ckl:elliptic:special} of \cref{thm:elliptic:special,thm:elliptic:general}, the high-order partial derivatives $\{ \Qa^{(m,n)}, \Qb^{(m,n)}, \psi^{(m,n)} : m+n\le 2\}$ are required. Furthermore, the symbolic computation in Maple also reveals that the unique $\{C_{k,\ell}\}_{k,\ell=-1,0,1}$ and $F_{i,j}$ of the FDM in \cref{thm:elliptic:general:reduce} need  $\{ \Qa^{(m,n)}, \Qb^{(m,n)} : m+n\le 4\}$ and $\{ \psi^{(m,n)}  :  m+n\le 5\}$. Since $\kappa$ and $f$ are available in \eqref{Model：Original:Elliptic}, the expressions of \eqref{a:b:psi} and \eqref{QuQaQb} imply that, we only need to approximate
	\be\label{auk:mn}
	\left(\frac{\alpha_u(u)} {  \kappa}\right)^{(m,n)} \quad \text{and} \quad \left(\frac{\beta_u(u)} {  \kappa}\right)^{(m,n)}, \quad \text{with} \quad m+n\le4.
	\ee
	So we use the following FDMs to evaluate \eqref{auk:mn}. 
	For any smooth 1D function $\rho(x)\in C^5(\R)$,
	\\
	{\bf{ The first-order derivatives:}}
	\begin{align*}
		\rho_{x}(x_i)=&\tfrac{1}{h}\big[\tfrac{1}{20} \rho(x_{i-2})-  \tfrac{1}{2} \rho(x_{i-1})-\tfrac{1}{3}\rho(x_{i})+  \rho(x_{i+1})-\tfrac{1}{4}  \rho(x_{i+2})+\tfrac{1}{30}  \rho(x_{i+3})\big]+\bo(h^{5}), \\
		&  \text{ if} \quad  0 \le x_{i\pm 2}, x_{i+3} \le 1,\\
		%%%%%%%%%%%%%%%%%%%%%%%%%%%%%%%%%%%%%%%%
		\rho_{x}(x_i)=&\tfrac{1}{h}\big[-\tfrac{1}{30} \rho(x_{i-3})+  \tfrac{1}{4} \rho(x_{i-2})-\rho(x_{i-1})+  \tfrac{1}{3} \rho(x_{i})+\tfrac{1}{2}  \rho(x_{i+1})-\tfrac{1}{20}  \rho(x_{i+2})\big]+\bo(h^{5}), \\
		&  \text{ if} \quad  0 \le x_{i\pm 2}, x_{i-3} \le 1,\\
		%%%%%%%%%%%%%%%%%%%%%%%%%%%%%%%%%%%%%%%%	
		\rho_{x}(x_i)=&\tfrac{\pm 1}{h}\big[  \tfrac{1}{5} \rho(x_{i\pm5}) -\tfrac{5}{4} \rho(x_{i\pm4}) +\tfrac{10}{3}\rho(x_{i\pm3})-5\rho(x_{i\pm2})   	+5  \rho(x_{i\pm1})  -\tfrac{137}{60}  \rho(x_{i}) \big]+\bo(h^{5}), \\
		& \text{ if} \quad 0 \le x_{i\pm5} \text{ and } x_{i}\le 1, \\
		%%%%%%%%%%%%%%%%%%%%%%%%%%%%%%%%%%%%%%%%
		\rho_{x}(x_i)=& \tfrac{\pm1}{h}\big[ -\tfrac{1}{20} \rho(x_{i\pm4})+ \tfrac{1}{3} \rho(x_{i\pm3})  -\rho(x_{i\pm2})+  2 \rho(x_{i\pm1}) - \tfrac{13}{12} \rho(x_{i})-\tfrac{1}{5} \rho(x_{i\mp1})\big]+\bo(h^{5}), \\
		& \text{ if} \quad 0 \le x_{i\pm4} \text{ and } x_{i\mp1}\le 1.
	\end{align*}
	{\bf{ The second-order derivatives:}}
	\begin{align*}
		\rho_{xx}(x_i)=&\tfrac{1}{h^2}\big[-\tfrac{1}{12} \rho(x_{i-2})+  \tfrac{4}{3} \rho(x_{i-1})-  \tfrac{5}{2} \rho(x_{i})+\tfrac{4}{3}  \rho(x_{i+1})-\tfrac{1}{12}  \rho(x_{i+2})\big]+\bo(h^{4}), \\
		&  \text{ if} \quad  0 \le x_{i\pm 2} \le 1,\\
		%%%%%%%%%%%%%%%%%%%%%%%%%%%%%%%%%%%%%%%%
		\rho_{xx}(x_i)=&\tfrac{1}{h^2}\big[ - \tfrac{5}{6}\rho(x_{i\pm5})+\tfrac{61}{12}\rho(x_{i\pm4})-13 \rho(x_{i\pm3})+  \tfrac{107}{6} \rho(x_{i\pm2})   	-  \tfrac{77}{6}   \rho(x_{i\pm1})+\tfrac{15}{4}\rho(x_{i}) \big]+\bo(h^{4}), \\
		& \text{ if} \quad 0 \le x_{i\pm5} \text{ and } x_{i}\le 1, \\
		%%%%%%%%%%%%%%%%%%%%%%%%%%%%%%%%%%%%%%%%
		\rho_{xx}(x_i)=&\tfrac{1}{h^2}\big[\tfrac{1}{12} \rho(x_{i\pm4})-\tfrac{1}{2} \rho(x_{i\pm3})  +\tfrac{7}{6}\rho(x_{i\pm2})- \tfrac{1}{3} \rho(x_{i\pm1})-  \tfrac{5}{4} \rho(x_{i})+\tfrac{5}{6}   \rho(x_{i\mp1})\big]+\bo(h^{4}), \\
		& \text{ if} \quad 0 \le x_{i\pm4} \text{ and } x_{i\mp1}\le 1.
	\end{align*}
	{\bf{ The third-order derivatives:}}
	\begin{align*}
		\rho_{xxx}(x_i)=&\tfrac{1}{h^3}\big[-\tfrac{1}{4} \rho(x_{i-2})-\tfrac{1}{4}  \rho(x_{i-1})+\tfrac{5}{2} \rho(x_{i})-  \tfrac{7}{2} \rho(x_{i+1})+\tfrac{7}{4} \rho(x_{i+2}) -  \tfrac{1}{4}\rho(x_{i+3})\big]+\bo(h^{3}), \\
		&  \text{ if} \quad  0 \le x_{i\pm 2}, x_{i+3} \le 1,\\
		%%%%%%%%%%%%%%%%%%%%%%%%%%%%%%%%%%%%%%%%
		\rho_{xxx}(x_i)=&\tfrac{1}{h^3}\big[\tfrac{1}{4} \rho(x_{i-3})-  \tfrac{7}{4} \rho(x_{i-2})+\tfrac{7}{2}\rho(x_{i-1})-  \tfrac{5}{2} \rho(x_{i})+\tfrac{1}{4}  \rho(x_{i+1})+\tfrac{1}{4}  \rho(x_{i+2})\big]+\bo(h^{3}), \\
		&  \text{ if} \quad  0 \le x_{i\pm 2}, x_{i-3} \le 1,\\
		%%%%%%%%%%%%%%%%%%%%%%%%%%%%%%%%%%%%%%%%
		\rho_{xxx}(x_i)=&\tfrac{\pm 1}{h^3}\big[ \tfrac{7}{4} \rho(x_{i\pm5}) -\tfrac{41}{4} \rho(x_{i\pm4})+ \tfrac{49}{2} \rho(x_{i\pm3})-\tfrac{59}{2} \rho(x_{i\pm2})   	+\tfrac{71}{4} \rho(x_{i\pm1})   -\tfrac{17}{4}  \rho(x_{i}) \big]+\bo(h^{3}), \\
		& \text{ if} \quad 0 \le x_{i\pm5} \text{ and } x_{i}\le 1, \\
		%%%%%%%%%%%%%%%%%%%%%%%%%%%%%%%%%%%%%%%%
		\rho_{xxx}(x_i)=&\tfrac{\pm 1}{h^3}\big[ \tfrac{1}{4} \rho(x_{i\pm4}) -\tfrac{7}{4} \rho(x_{i\pm3})+\tfrac{11}{2}\rho(x_{i\pm2})-\tfrac{17}{2} \rho(x_{i\pm1})+\tfrac{25}{4} \rho(x_{i}) -\tfrac{7}{4} \rho(x_{i\mp1})\big]+\bo(h^{3}), \\
		& \text{ if} \quad 0 \le x_{i\pm4} \text{ and } x_{i\mp1}\le 1.
	\end{align*}
	{\bf{ The fourth-order derivatives:}}
	\begin{align*}
		\rho_{xxxx}(x_i)=&\tfrac{1}{h^4}\big[ \rho(x_{i-2})-4 \rho(x_{i-1})+6 \rho(x_{i})-4  \rho(x_{i+1})+  \rho(x_{i+2})\big]+\bo(h^{2}), \\
		&  \text{ if} \quad  0 \le x_{i\pm 2} \le 1,\\
		%%%%%%%%%%%%%%%%%%%%%%%%%%%%%%%%%%%%%%%%
		\rho_{xxxx}(x_i)=&\tfrac{1}{h^4}\big[  -2\rho(x_{i\pm5})+11 \rho(x_{i\pm4})-  24 \rho(x_{i\pm3})+  26 \rho(x_{i\pm2})   	 -14 \rho(x_{i\pm1})+3\rho(x_{i}) \big]+\bo(h^2), \\
		& \text{ if} \quad 0 \le x_{i\pm5} \text{ and } x_{i}\le 1, \\
		%%%%%%%%%%%%%%%%%%%%%%%%%%%%%%%%%%%%%%%%
		\rho_{xxxx}(x_i)=&\tfrac{1}{h^4}\big[   -\rho(x_{i\pm4})+ 6\rho(x_{i\pm3})-  14 \rho(x_{i\pm2})+  16 \rho(x_{i\pm1})-9 \rho(x_{i})+2\rho(x_{i\mp1})\big]+\bo(h^2), \\
		& \text{ if} \quad 0 \le x_{i\pm4} \text{ and } x_{i\mp1}\le 1. 
	\end{align*}
	{\bf{ The fifth-order derivatives:}}
	\begin{align*}
		\rho_{xxxxx}(x_i)=&\tfrac{1}{h^5}\big[- \rho(x_{i-2})+ 5 \rho(x_{i-1})-10 \rho(x_{i})+10\rho(x_{i+1})-5 \rho(x_{i+2}) +\rho(x_{i+3})\big]+\bo(h), \\
		&  \text{ if} \quad  0 \le x_{i\pm 2}, x_{i+3} \le 1,\\
		%%%%%%%%%%%%%%%%%%%%%%%%%%%%%%%%%%%%%%%%
		\rho_{xxxxx}(x_i)=&\tfrac{1}{h^5}\big[- \rho(x_{i-3})+ 5 \rho(x_{i-2})-10 \rho(x_{i-1})+10\rho(x_{i})-5 \rho(x_{i+1}) +\rho(x_{i+2})\big]+\bo(h), \\
		&  \text{ if} \quad  0 \le x_{i\pm 2}, x_{i-3} \le 1,\\
		%%%%%%%%%%%%%%%%%%%%%%%%%%%%%%%%%%%%%%%%
		\rho_{xxxxx}(x_i)=&\tfrac{\pm 1}{h^5}\big[ \rho(x_{i\pm5}) -5 \rho(x_{i\pm4})+ 10 \rho(x_{i\pm3})-10 \rho(x_{i\pm2})   	+5 \rho(x_{i\pm1})   -  \rho(x_{i}) \big]+\bo(h), \\
		& \text{ if} \quad 0 \le x_{i\pm5} \text{ and } x_{i}\le 1, \\
		%%%%%%%%%%%%%%%%%%%%%%%%%%%%%%%%%%%%%%%%
		\rho_{xxxxx}(x_i)=&\tfrac{\pm 1}{h^5}\big[ \rho(x_{i\pm4}) -5 \rho(x_{i\pm3})+ 10 \rho(x_{i\pm2})-10 \rho(x_{i\pm1})   	+5 \rho(x_{i})   -  \rho(x_{i\mp1}) \big]+\bo(h), \\
		& \text{ if} \quad 0 \le x_{i\pm4} \text{ and } x_{i\mp1}\le 1.
	\end{align*}
	Similarly, 	for any smooth 2D function $\rho(x,y)\in C^5(\R^2)$, we can evaluate  $\rho_{x}(x_i,y_j)$,	$\rho_{xx}(x_i,y_j)$, 	$\rho_{xxx}(x_i,y_j)$, 	$\rho_{xxxx}(x_i,y_j)$, $\rho_{xxxxx}(x_i,y_j)$, $\rho_{y}(x_i,y_j)$, 	$\rho_{yy}(x_i,y_j)$, 	$\rho_{yyy}(x_i,y_j)$, 	$\rho_{yyyy}(x_i,y_j)$, $\rho_{yyyyy}(x_i,y_j)$. Then we can evaluate the high-order mixed derivatives as follows: 
	\[
	\begin{split}
		&	\rho_{xy}=(\rho_{x})_{y}, \quad \rho_{xxy}=(\rho_{xx})_{y}, \quad \rho_{xyy}=(\rho_{x})_{yy}, \quad \rho_{xxxy}=(\rho_{xxx})_{y}, \quad \rho_{xxyy}=(\rho_{xx})_{yy},  \\
		&\rho_{xyyy}=(\rho_{x})_{yyy}, \quad   \rho_{xxxxy}=(\rho_{xxxx})_{y}, \quad \rho_{xxxyy}=(\rho_{xxx})_{yy}, \quad \rho_{xxyyy}=(\rho_{xx})_{yyy}, \quad \rho_{xyyyy}=(\rho_{x})_{yyyy}.
	\end{split}
	\]
	Note that the fifth-order partial derivatives are not required in this section, but they are necessary for the FDMs in \cref{sec:parabolic}. 
	\section{Second- to fourth-order compact 9-point FDMs for the time-dependent nonlinear  convection-diffusion equation}\label{sec:parabolic}

	Similarly to \cref{sec:elliptic}, we rewrite the  time-dependent nonlinear convection-diffusion equation \eqref{Model：Original:Parabolic}  as the linear problem via a fixed point method in \cref{sec:reformulate:2}, and the compact 9-point FDMs are constructed in \cref{subsec:parabolic} to solve the reformulated linear problem.

	\subsection{Reformulation of  the  time-dependent nonlinear convection-diffusion equation}\label{sec:reformulate:2}

	\eqref{Model：Original:Parabolic} leads to
	\be\label{parabolic:expand}
	u_t-  \kappa \Delta u -\kappa_x u_x -\kappa_y u_y +   \alpha_u(u)u_x+\beta_u(u)u_y  = f.
	\ee
	Recall that  temporal domain $I=[0,1]$, and $u$ is the exact solution of \eqref{Model：Original:Parabolic} and \eqref{parabolic:expand}. Here, we define that 
	\be\label{def:tau}
	u^{n}:=u|_{t=t_{n}}, \qquad t_{n}:=n \tau, \qquad n=0,\ldots,N_2, \qquad \tau:=1/N_2, \qquad N_2\in \N.
	\ee
	\textbf{The second-order Crank-Nicolson method:} \eqref{Model：Original:Parabolic} and \eqref{parabolic:expand} with the CN method in \citep{Burkardt2020} imply
	\be\label{time:order:2}
	\begin{cases}
		& \hspace{-0.3cm} \displaystyle{\frac{u^{n+1/2}-u^{n}}{ \tau/2}-  \kappa^{n+1/2} \Delta u^{n+1/2}} \displaystyle{ +[\alpha_u(u^{n+1/2})-\kappa^{n+1/2}_x] u^{n+1/2}_x +[\beta_u(u^{n+1/2})-\kappa^{n+1/2}_y] u^{n+1/2}_y}= f^{n+1/2}, \\
		& \hspace{-0.3cm} u^{n+1}=2u^{n+1/2}-u^{n}, \quad \text{with the initial } n=0, \text{ the given } u^0\in \Omega, \ u^{n+1/2}\in \partial \Omega.
	\end{cases}
	\ee
	The first identity in \eqref{time:order:2} yields
	\be\label{parabolic:CN}
	\begin{split}
		&   \Delta u^{n+1/2} + \frac{\kappa^{n+1/2}_x-\alpha_u(u^{n+1/2})}{\kappa^{n+1/2}} u^{n+1/2}_x + \frac{\kappa^{n+1/2}_y-\beta_u(u^{n+1/2})}{\kappa^{n+1/2}} u^{n+1/2}_y-\frac{2u^{n+1/2}}{\tau \kappa^{n+1/2}  } \\
		& = \frac{-1}{\kappa^{n+1/2} } \left[f^{n+1/2}+\frac{2}{\tau }u^{n} \right], \quad \text{with the initial } n=0, \text{ the given } u^0\in \Omega, \ u^{n+1/2}\in \partial \Omega.
	\end{split}
	\ee
	We denote  the solution at $t=(n+1/2) \tau$ in the $\Qk$-iteration by $u^{n+1/2}_\Qk$. Then we use the following iteration method to rewrite \eqref{time:order:2}  as the  linear convection-diffusion equation:
	\be\label{iteration:CN}
	\begin{cases}
		& \hspace{-0.3cm}  \displaystyle{\Delta u_{\Qk+1}^{n+1/2} + \frac{\kappa^{n+1/2}_x-\alpha_u(u_\Qk^{n+1/2})}{\kappa^{n+1/2}} (u_{\Qk+1}^{n+1/2})_x + \frac{\kappa^{n+1/2}_y-\beta_u(u_\Qk^{n+1/2})}{\kappa^{n+1/2}} (u_{\Qk+1}^{n+1/2})_y-\frac{2u_{\Qk+1}^{n+1/2}}{\tau \kappa^{n+1/2}  }} \\
		& \hspace{-0.3cm} \displaystyle{= \frac{-1}{\kappa^{n+1/2}}\left[ f^{n+1/2}+\frac{2}{\tau  }u^{n}\right],}\    \\
		& \hspace{-0.3cm} \Qk:=\Qk+1, \ \text{the initial  } \Qk=n=0, \text{ the given } u^0\in \Omega, \  u_{\Qk+1}^{n+1/2}\in \partial \Omega,\  u^{n+1/2}_{0}= u^{n}  \text{ in } \Omega, \\
		& \hspace{-0.3cm} u^{n+1}=2u^{n+1/2}-u^{n}, 
	\end{cases}
	\ee
	where $u^{n}$ with $n\ge 1$ is calculated at $t=n\tau $ by the same iteration method.
	Let 
	\be\label{sub:c:1}
	\begin{split}
		& \tau:=rh, \qquad 
		\Qu:= u_{\Qk+1}^{n+1/2}, \qquad \Qa:=\frac{\kappa^{n+1/2}_x-\alpha_u(u_\Qk^{n+1/2})}{\kappa^{n+1/2}}, \qquad \Qb:=\frac{\kappa^{n+1/2}_y-\beta_u(u_\Qk^{n+1/2})}{\kappa^{n+1/2}}, \\	
		&  \Qc:=-\frac{2}{ r\kappa^{n+1/2}  }, \qquad \Qd:= \frac{\Qc}{h}, \qquad \varphi:=-\frac{ f^{n+1/2}}{\kappa^{n+1/2} }, \qquad \phi:=-\frac{2u^{n}}{ r\kappa^{n+1/2}  },\qquad \psi:=  \varphi+\frac{\phi}{h},
	\end{split}
	\ee
	where $r$ is a positive constant.
	Then  \eqref{Model：Original:Parabolic} and \eqref{iteration:CN}  indicate
	\be\label{elliptic:CN:2}
	\Delta \Qu  + \Qa \Qu_x + \Qb \Qu_y + \Qd \Qu   =   \psi \text{ in }  \Omega \quad \text{and} \quad \Qu=g \text{ on } \partial\Omega.
	\ee
	\textbf{The third-order backward difference formula:}
	\eqref{Model：Original:Parabolic} and \eqref{parabolic:expand} with the BDF3 method in \citep[p.366]{Hairer1993} give
	\be\label{time:order:3}
	\begin{split}
		&  \frac{11u^{n+3}-18u^{n+2}+9u^{n+1}-2u^{n} }{ 6 \tau}-  \kappa^{n+3} \Delta u^{n+3}+(\alpha_u(u^{n+3})-\kappa^{n+3}_x) u^{n+3}_x\\
		&  +(\beta_u(u^{n+3})-\kappa^{n+3}_y) u^{n+3}_y= f^{n+3}, \quad \text{with the given } u^0\in \Omega \text{ and } u^{n+3}\in \partial \Omega.
	\end{split}
	\ee
	After a direct calculation, we rewrite
	\be\label{nonlinear:BDF3}
	\begin{split}
		&   \Delta u^{n+3} + \frac{\kappa^{n+3}_x-\alpha_u(u^{n+3})}{\kappa^{n+3}} u^{n+3}_x + \frac{\kappa^{n+3}_y-\beta_u(u^{n+3})}{\kappa^{n+3}} u^{n+3}_y-\frac{11u^{n+3}}{6\tau \kappa^{n+3}  } \\
		& =\frac{-1}{\kappa^{n+3}} \left[  f^{n+3}+\frac{3}{\tau   }u^{n+2}-\frac{3}{2\tau  }u^{n+1}+\frac{1}{3\tau  }u^{n}\right], \quad \text{with the given } u^0\in \Omega \text{ and } u^{n+3}\in \partial \Omega. 
	\end{split}
	\ee
	Similarly to \eqref{iteration:CN}, the above nonlinear equation \eqref{nonlinear:BDF3} is solved by the following linear convection-diffusion equation, via the iteration method: 
	\be\label{iteration:BDF3}
	\begin{cases}
		&   \displaystyle{ \hspace{-0.3cm} \Delta u_{\Qk+1}^{n+3} + \frac{\kappa^{n+3}_x-\alpha_u(u_\Qk^{n+3})}{\kappa^{n+3}} (u_{\Qk+1}^{n+3})_x + \frac{\kappa^{n+3}_y-\beta_u(u_\Qk^{n+3})}{\kappa^{n+3}} (u_{\Qk+1}^{n+3})_y-\frac{11u_{\Qk+1}^{n+3} }{6\tau \kappa^{n+3}  }} \\
		&  \hspace{-0.3cm}  \displaystyle{= \frac{-1}{\kappa^{n+3}}\left[ f^{n+3}+\frac{3}{\tau  }u^{n+2}-\frac{3}{2\tau   }u^{n+1}+\frac{1}{3\tau   }u^{n}\right],} \\
		& \hspace{-0.3cm} \Qk:=\Qk+1, \ \text{the initial  } \Qk=n=0, \text{ the given } u^0\in \Omega,\  u_{\Qk+1}^{n+3}\in \partial \Omega, \text{ and }  u^{n+3}_{0}= u^{n+2} \text{ in } \Omega,\\
	\end{cases}
	\ee
	where $u^{n+3}_\Qk$ with $n\ge 0$ is computed in the $\Qk$-iteration at $t=(n+3)\tau$; $u^{n},u^{n+1}$, and $u^{n+2}$ with $n\ge3$ are calculated at $t=n\tau, (n+1)\tau$, and $(n+2)\tau $ by the same iteration method; $u^{1}$ and $u^{2}$ are computed by the CN method.
	Let 
	\be\label{sub:c:2}
	\begin{split}
		& \tau:=rh, \qquad 
		\Qu:= u_{\Qk+1}^{n+3}, \qquad \Qa:=\frac{\kappa^{n+3}_x-\alpha_u(u_\Qk^{n+3})}{\kappa^{n+3}}, \qquad \Qb:=\frac{\kappa^{n+3}_y-\beta_u(u_\Qk^{n+3})}{\kappa^{n+3}},\qquad \Qc:=-\frac{11}{ 6r\kappa^{n+3}  },\\
		&  \Qd:= \frac{\Qc}{h},  \qquad \psi:=  \varphi+\frac{\phi}{h} , \qquad \varphi:=-\frac{ f^{n+3}}{\kappa^{n+3} }, \qquad \phi:=\frac{-1}{\kappa^{n+3}}\left[\frac{3}{r   }u^{n+2}-\frac{3}{2r   }u^{n+1}+\frac{1}{3r   }u^{n}\right],
	\end{split}
	\ee
	where $r$ is a positive constant.
	Then
	\be\label{elliptic:BDF3:1}
	\Delta \Qu  + \Qa \Qu_x + \Qb \Qu_y + \Qd \Qu   =   \psi \text{ in }  \Omega \quad \text{and} \quad \Qu=g \text{ on } \partial\Omega.
	\ee
	\textbf{The fourth-order backward difference formula:}
	\eqref{Model：Original:Parabolic} and \eqref{parabolic:expand} with the BDF4 method in \citep[p.366]{Hairer1993}  generate
	\be\label{time:order:4}
	\begin{split}
		&  \frac{25u^{n+4}-48u^{n+3}+36u^{n+2}-16u^{n+1}+3u^{n} }{ 12 \tau}-  \kappa^{n+4} \Delta u^{n+4}+(\alpha_u(u^{n+4})-\kappa^{n+4}_x) u^{n+4}_x\\
		&  +(\beta_u(u^{n+4})-\kappa^{n+4}_y) u^{n+4}_y= f^{n+4}, \quad \text{with the given } u^0\in \Omega \text{ and } u^{n+4}\in \partial \Omega.
	\end{split}
	\ee
	Then,
	\be\label{nonlinear:BDF4}
	\begin{split}
		&   \Delta u^{n+4} + \frac{\kappa^{n+4}_x-\alpha_u(u^{n+4})}{\kappa^{n+4}} u^{n+4}_x + \frac{\kappa^{n+4}_y-\beta_u(u^{n+4})}{\kappa^{n+4}} u^{n+4}_y-\frac{25u^{n+4}}{12\tau \kappa^{n+4}  } \\
		& = \frac{-1}{\kappa^{n+4}}\left[  f^{n+4}+\frac{4}{\tau  }u^{n+3}-\frac{3}{\tau  }u^{n+2}+\frac{4}{3\tau   }u^{n+1} -\frac{1}{4\tau   }u^{n}\right], \quad \text{with given } u^0\in \Omega \text{ and } u^{n+4}\in \partial \Omega. 
	\end{split}
	\ee
	Similarly to \eqref{iteration:BDF3}, we linearize \eqref{nonlinear:BDF4} as
	\be\label{iteration:BDF4}
	\begin{cases}
		&   \displaystyle{ \hspace{-0.3cm} \Delta u_{\Qk+1}^{n+4} + \frac{\kappa^{n+4}_x-\alpha_u(u_\Qk^{n+4})}{\kappa^{n+4}} (u_{\Qk+1}^{n+4})_x + \frac{\kappa^{n+4}_y-\beta_u(u_\Qk^{n+4})}{\kappa^{n+4}} (u_{\Qk+1}^{n+4})_y-\frac{25u_{\Qk+1}^{n+4}}{12\tau \kappa^{n+4}  } } \\
		&  \hspace{-0.3cm}  \displaystyle{=  \frac{-1}{\kappa^{n+4}}\left[ f^{n+4}+\frac{4}{\tau   }u^{n+3}-\frac{3}{\tau   }u^{n+2}+\frac{4}{3\tau   }u^{n+1}-\frac{1}{4\tau   }u^{n}\right],} \\
		& \hspace{-0.3cm} \Qk:=\Qk+1, \ \text{the initial  } \Qk=n=0, \text{ the given } u^0\in \Omega,\  u_{\Qk+1}^{n+4}\in \partial \Omega, \text{ and }  u^{n+4}_{0}= u^{n+3} \text{ in } \Omega,\\
	\end{cases}
	\ee
	where $u^{1}$ and $u^{2}$ are computed by the CN method, and $u^{3}$ is computed by the BDF3 method.
	Let 
	\be\label{sub:c:3}
	\begin{split}
		& \tau:=rh, \qquad 
		\Qu:= u_{\Qk+1}^{n+4}, \qquad \Qa:=\frac{\kappa^{n+4}_x-\alpha_u(u_\Qk^{n+4})}{\kappa^{n+4}}, \qquad \Qb:=\frac{\kappa^{n+4}_y-\beta_u(u_\Qk^{n+4})}{\kappa^{n+4}},\qquad \Qc:=-\frac{25}{ 12r\kappa^{n+4}  },\\
		&  \Qd:= \frac{\Qc}{h},  \qquad \psi:=  \varphi+\frac{\phi}{h} , \qquad \varphi:=-\frac{ f^{n+4}}{\kappa^{n+4} }, \qquad \phi:=-\frac{1}{\kappa^{n+4}} \left[\frac{4}{r  }u^{n+3}-\frac{3}{r  }u^{n+2}+\frac{4}{3r   }u^{n+1}-\frac{1}{4r   }u^{n}\right],
	\end{split}
	\ee
	where $r$ is a positive constant.
	Then
	\be\label{elliptic:BDF4:1}
	\Delta \Qu  + \Qa \Qu_x + \Qb \Qu_y + \Qd \Qu   =   \psi \text{ in }  \Omega \quad \text{and} \quad \Qu=g \text{ on } \partial\Omega.
	\ee
	So, the CN, BDF3, and BDF4 methods with the iteration methods \eqref{iteration:CN}, \eqref{iteration:BDF3}, and \eqref{iteration:BDF4}  yield the same linear convection-diffusion equation in \eqref{elliptic:CN:2}, \eqref{elliptic:BDF3:1}, and \eqref{elliptic:BDF4:1}. 
	
	Next, we propose the fourth-order compact 9-point FDM for the linear convection-diffusion equation in the following \cref{subsec:parabolic}. 
	\subsection{Fourth-order compact 9-point FDM }\label{subsec:parabolic}
	Since the  CN, BDF3, and BDF4 methods yield the same linear convection-diffusion equation  \eqref{elliptic:CN:2}, \eqref{elliptic:BDF3:1}, and \eqref{elliptic:BDF4:1}, in this section we construct the fourth-order compact 9-point FDM for   
	\be\label{reformulate}
	\Delta \Qu  + \Qa \Qu_x + \Qb \Qu_y +\Qd \Qu   = \psi,
	\ee
	where $\Qa, \Qb, \Qd, \psi$ are defined in \eqref{sub:c:1}, \eqref{sub:c:2}, and \eqref{sub:c:3}.
	Similarly to \eqref{u20}--\eqref{u:pq},
	\be\label{u:pq:parabolic}
	\begin{split}
		\Qu^{(p,q)}  =&\sum_{(m,n)\in \ind_{p+q}^{1} }\xi_{p,q,m,n} \Qu^{(m,n)}+\sum_{(m,n)\in \ind_{p+q-2} } \eta_{p,q,m,n}   \psi^{(m,n)}, \qquad (p,q)\in \ind_{M}^{2},
	\end{split}
	\ee
	where $ \xi_{p,q,m,n}$ and $\eta_{p,q,m,n}$ are uniquely determined by the high-order partial derivatives of $\Qa$, $\Qb$, and $\Qd$. 
	Similarly to \eqref{Space:S}--\eqref{GHmn},
	\be \label{u:GH:parabolic}
	\Qu(x+x_i,y+y_j)
	=
	\sum_{(m,n)\in \ind^1_{M}}
	\Qu^{(m,n)}G_{M,m,n}(x,y)+\sum_{(m,n)\in \ind_{M-2}}	\psi^{(m,n)} H_{M,m,n}(x,y)+\bo(h^{M+1}),
	\ee
	where  $x,y\in [-h,h]$, and
	\be\label{GHmn:parabolic}
	\begin{split}
		& G_{M,m,n}(x,y):=\frac{x^my^n}{m!n!}+\sum_{(p,q)\in \ind^2_{M} \setminus \ind^2_{m+n-1} } \xi_{p,q,m,n} \frac{x^py^q}{p!q!}, \quad  H_{M,m,n}(x,y):=\sum_{(p,q)\in \ind^2_{M} \setminus \ind^2_{m+n+1} } \eta_{p,q,m,n} \frac{x^py^q}{p!q!},
	\end{split}
	\ee
	each $\xi_{p,q,m,n},   \eta_{p,q,m,n}$ in \eqref{u:pq:parabolic} and \eqref{GHmn:parabolic} belongs to 
	\be\label{Space:S:parabolic}
	S:=\text{span}\left\{ \prod_{ i_1,j_1,v_1,w_1,r_1,s_1\in \N_0}\Qa^{(i_1,j_1)}\Qb^{(v_1,w_1)}\Qd^{(r_1,s_1)}, \dots,  \prod_{ i_k,j_k,v_k,w_k,r_k,s_k\in \N_0}\Qa^{(i_k,j_k)}\Qb^{(v_k,w_k)}\Qd^{(r_k,s_k)}\right\}, 
	\ee
	with $k\in \N$, and the coefficients in $\R$.
	Similarly to  \eqref{L:h:u:h}--\eqref{L:h:u},  we define the linear operator $\mathcal{L}_h \Qu$ and the stencil $\mathcal{L}_h \Qu_h$ of the FDM  as follows:
	\be\label{L:h:u:h:parabolic}
	\begin{split}
		h^{-2}	\mathcal{L}_h \Qu:= h^{-2} \sum_{k,\ell =-1}^{1} C_{k,\ell} \Qu(kh+x_i, \ell h+y_j), \quad 
		h^{-2}	\mathcal{L}_h \Qu_h:= h^{-2} \sum_{k,\ell =-1}^{1}  C_{k,\ell} (\Qu_h)_{i+k,j+\ell}=F_{i,j},
	\end{split}
	\ee
	\be\label{Ckl:parabolic}
	C_{k,\ell}:=\sum_{p=0}^7c_{k,\ell,p} h^p, \quad c_{k,\ell,p}\in \tilde{S},
	\ee
	where
	\be\label{Space:S:tilde:parabolic}
	\tilde{S}:=\text{span}\left\{ \prod_{ i_1,j_1,v_1,w_1,r_1,s_1\in \N_0}\Qa^{(i_1,j_1)}\Qb^{(v_1,w_1)}\Qc^{(r_1,s_1)}, \dots,  \prod_{ i_k,j_k,v_k,w_k,r_k,s_k\in \N_0}\Qa^{(i_k,j_k)}\Qb^{(v_k,w_k)}\Qc^{(r_k,s_k)}\right\},
	\ee
	$\Qa, \Qb$, and $\Qc$ are defined in \eqref{sub:c:1}, \eqref{sub:c:2}, and \eqref{sub:c:3}.
	We restrict each $ c_{k,\ell,p}\in \tilde{S}$ in \eqref{Ckl:parabolic} such that the high-order partial derivatives $\Qa^{(m,n)}$, $\Qb^{(m,n)}$, and $\Qc^{(m,n)}$ do not appear in the denominator, ensuring that each $c_{k,\ell,p}$ is well defined.
	Similarly to \eqref{Fij:order6},
	\be\label{Fij:order6:parabolic}
	F_{i,j}=\text{the terms of } \Big( h^{-2} \sum_{(m,n) \in \ind_{	 5}} \psi^{(m,n)} \sum_{k,\ell =-1}^{1} 	 C_{k,\ell} H_{7,m,n}  (kh, \ell h)\Big) \text{ with degree}   \le 5 \text{ in } h,
	\ee
	where $H_{7,m,n}  (x, y)$ is defined in \eqref{GHmn:parabolic}, and $\psi$ is defined in \eqref{sub:c:1}, \eqref{sub:c:2}, and \eqref{sub:c:3}.
	Similarly to \eqref{L:h:u:h}--\eqref{elliptic:order:M} and \eqref{Ckl:degree:7}--\eqref{ckl0:elliptic}, by the symbolic calculation in Maple, there exist $\{c_{k,\ell,p}\}_{k,\ell=-1,0,1}^{p=0,7}$ with $c_{k,\ell,p}\in \tilde{S}$ satisfying
	\be\label{QE:order:4:reduce:parabolic}
	\begin{split}
		& \sum_{(m,n)\in \ind_{7}^{1}} \Qu^{(m,n)}\sum_{k,\ell=-1}^{1} C_{k,\ell}G_{7,m,n}  (kh, \ell h)-\tfrac{h^6}{90}(\Qa^{(0,1)}-\Qb^{(1,0)})\Qu^{(1,3)}-h^7\zeta=\bo(h^{8}),\\
		& \quad c_{0,0,0}=-\tfrac{10}{3},  \qquad  F_{i,j}|_{h=0}=\psi, \qquad c_{-1,-1,0}=c_{-1,1,0}=c_{1,-1,0}=c_{1,1,0}=\tfrac{1}{6},\\
		& \quad c_{-1,0,0}=c_{1,0,0}=c_{0,-1,0}=c_{0,1,0}=\tfrac{2}{3},
	\end{split}
	\ee
	for any free variables in the following set:
	\be \label{free:parabolic}
	\begin{split}
		&\{c_{-1,0,7}\} \cup \{c_{-1,1,p}\}_{p=6,7}  \cup \{c_{0,-1,7}\} \cup \{c_{0,0,p}\}_{p=6,7}  \\
		&\ \cup \{c_{0,1,p}\}_{p=5,6,7} \cup \{c_{1,-1,p}\}_{p=5,6,7} \cup \{c_{1,0,p}\}_{p=4,\dots,7} \cup \{c_{1,1,p}\}_{p=1,\dots,7},
	\end{split}
	\ee
	where
	\be\label{zeta:parabolic}
	\begin{split}
		\zeta:=&\tfrac{1}{37800}(10\Qa  (6\Qb  \Qc   - 21(\Qa^{(0, 1)} -\Qb^{(1, 0)}) - 8\Qc^{(0, 1)})  -210(\Qa^{(0, 1)} -\Qb^{(1, 0)} +\Qc^{(1, 0)})\Qb   \\
		&   - (49\Qa^{(0, 1)}+ 91\Qb^{(1, 0)})\Qc + 2520c_{1,1,1}(\Qa^{(0, 1)} - \Qb^{(1, 0)}))u^{(1, 3)} - \tfrac{1}{7560} (\Qa  \Qc  +14 \Qc^{(1, 0)} )u^{(1, 4)}\\
		&  + \tfrac{1}{540}(\Qb  \Qc    - \Qc^{(0, 1)})u^{(0, 5)}.
	\end{split}
	\ee
	Similarly to \eqref{elliptic:order:6},  $C_{k,\ell}=\sum_{p=0}^{7}c_{k,\ell,p} h^p$  with $k,\ell=-1,0,1$ satisfying \eqref{QE:order:4:reduce:parabolic} implies
	\be\label{parabolic:order:6}
	h^{-2}	\mathcal{L}_h (\Qu- \Qu_h)=\tfrac{h^4}{90}(\Qa^{(0,1)}-\Qb^{(1,0)})\Qu^{(1,3)}+h^5\zeta+\bo(h^{6}),
	\ee
	and
	$\mathcal{L}_h \Qu_h$ approximates $	\Delta \Qu  +\Qa \Qu_x +\Qb \Qu_y +\Qd \Qu   = \psi$ in \eqref{reformulate}  with  the fourth-order of consistency at $(x_i,y_j)$ in \eqref{xiyj}, where $	\mathcal{L}_h \Qu$ and $\mathcal{L}_h \Qu_h$ are defined in \eqref{L:h:u:h:parabolic}, and $F_{i,j}$ is defined in \eqref{Fij:order6:parabolic}. In our numerical examples, we set all free variables in \eqref{free:parabolic} to zero, so $C_{k,\ell}=\sum_{p=0}^{7}c_{k,\ell,p} h^p$ with $k,\ell=-1,0,1$ can be uniquely determined by solving \eqref{QE:order:4:reduce:parabolic}. 
	
	In summary, we have the fourth-order compact 9-point FDM with the reduced pollution effects of $\bo(h^4)$ and $\bo(h^5)$ in the following theorem:
	\begin{theorem}\label{thm:parabolic:general:reduce}
		Assume $\kappa, u, \alpha, \beta, f$ are all smooth in \eqref{Model：Original:Parabolic},  $\mathcal{L}_h  {\textup \Qu}_h$ is defined in \eqref{L:h:u:h:parabolic}, the nine coefficients
		$C_{k,\ell}=\sum_{p=0}^{7}c_{k,\ell,p} h^p$ with $k,\ell=-1,0,1$ are uniquely determined by solving \eqref{QE:order:4:reduce:parabolic} with  variables  being zero in \eqref{free:parabolic}, the right-hand side $F_{i,j}$ is defined in \eqref{Fij:order6:parabolic}, and  all the functions in $\mathcal{L}_h  {\textup \Qu}_h$ are evaluated at the grid point $(x_i,y_j)$ in \eqref{xiyj}.  Then $h^{-2}	\mathcal{L}_h  {\textup \Qu}_h$
		approximates $	\Delta  {\textup \Qu}  +{\textup \Qa}  {\textup \Qu}_x +{\textup \Qb}  {\textup \Qu}_y  +{\textup \Qd}  {\textup \Qu}  = \psi$ in \eqref{reformulate} with  the fourth-order of consistency at $(x_i,y_j)$, and the truncation error $\tfrac{h^4}{90}({\textup \Qa}^{(0,1)}-{\textup \Qb}^{(1,0)}) {\textup \Qu}^{(1,3)}+h^5\lambda+\bo(h^{6})$, where $\lambda=\zeta$ with $c_{1,1,1}=0$ in \eqref{zeta:parabolic}, and ${\textup \Qa}, {\textup \Qb}, {\textup \Qc}, {\textup \Qd}$, $\psi$ are defined in \eqref{sub:c:1}, \eqref{sub:c:2}, and \eqref{sub:c:3}.
	\end{theorem}	
	\begin{proof}
		The proof follows similar arguments as in \cref{thm:elliptic:general:reduce}.
	\end{proof}
	\begin{proposition}\label{prop:2:ckl0}
		The nine coefficients $C_{k,\ell}=\sum_{p=0}^{7}c_{k,\ell,p} h^p$  with $k,\ell=-1,0,1$ of the FDM	in \cref{thm:parabolic:general:reduce}  satisfy
		\be\label{prop:ckl0:parabolic} 
		\begin{split} 
			& c_{\pm1,\pm1,0}=c_{\pm1,\mp1,0}=\tfrac{1}{6}, \qquad c_{\pm1,0,0}=c_{0,\pm1,0}=\tfrac{2}{3},\qquad  c_{0,0,0}=-\tfrac{10}{3},	\qquad \sum_{k=1}^{-1}\sum_{\ell=1}^{-1} c_{k,\ell,0}=0.
		\end{split}	
		\ee
		That is, the FDM in  \cref{thm:parabolic:general:reduce} satisfies the discrete maximum principle, and generates an M-matrix,   when the mesh size $h$ is sufficiently small.
	\end{proposition}
	\begin{proof}
		The proof follows similar arguments as in \cref{prop:1:ckl0}.
	\end{proof}
	Combining the FDM in \cref{thm:parabolic:general:reduce} with the iteration methods \eqref{iteration:CN}, \eqref{iteration:BDF3}, and \eqref{iteration:BDF4}, we provide the following \cref{algm2} with the CN method,  \cref{algm3} with the BDF3 method, and \cref{algm4} with the BDF4 method to solve \eqref{Model：Original:Parabolic}. Recall that $u$ is the exact solution of \eqref{Model：Original:Parabolic}, and $u^n=u|_{t=t_n=n\tau}$ in \eqref{def:tau}.
	Now, we define that $u^n_h$ is the numerical solution computed by our proposed FDM and the iteration method at $t=t_n=n\tau$, where $\tau$ is defined in \eqref{def:tau} and $N_2\tau=1$.
	\begin{algorithm}(CN method)\label{algm2}
		Step 1: Apply the first and second statements in the iteration method \eqref{iteration:CN} and the fourth-order FDM in \cref{thm:parabolic:general:reduce} to solve \eqref{elliptic:CN:2} with \eqref{sub:c:1}  to compute ${\textup \Qu}_h$ after 20 iterations. 
		Step 2: $u_h^{n+1/2}={\textup \Qu}_h$ and  $u_h^{n+1}=2u_h^{n+1/2}-u_h^{n}$.
		Step 3: Repeat steps 1-2 with $n:=n+1$ until $n=N_2-1$.
	\end{algorithm}
	\begin{algorithm}(BDF3 method)\label{algm3}
		Step 1: Apply the CN method to compute $u^{1}$, $u^{2}$.
		Step 2: Apply the iteration method \eqref{iteration:BDF3} and the fourth-order FDM in  \cref{thm:parabolic:general:reduce}  to solve \eqref{elliptic:BDF3:1} with \eqref{sub:c:2} to compute ${\textup \Qu}_h$ after 20 iterations.
		Step 3:  $u_h^{n+3}={\textup \Qu}_h$ and repeat the step 2 with $n:=n+1$ until $n=N_2-3$.
	\end{algorithm}	
	\begin{algorithm}(BDF4 method)\label{algm4}
		Step 1: Apply the CN method to compute $u^{1}$, $u^{2}$, and use BDF3 method to calculate $u^{3}$.
		Step 2: Apply the iteration method \eqref{iteration:BDF4} and the fourth-order FDM in  \cref{thm:parabolic:general:reduce}  to solve \eqref{elliptic:BDF4:1} with \eqref{sub:c:3} to compute ${\textup \Qu}_h$ after 20 iterations.
		Step 3:  $u_h^{n+4}={\textup \Qu}_h$ and repeat the step 2 with $n:=n+1$ until $n=N_2-4$.
	\end{algorithm}	
	\subsection{Approximations of the high-order partial derivatives of $\textnormal{\Qa}^{(m,n)}$, $\textnormal{\Qb}^{(m,n)}$, $\textnormal{\Qc}^{(m,n)}$, and $\psi^{(m,n)}$ }
	By the symbolic calculation of \eqref{QE:order:4:reduce:parabolic} of FDM in \cref{thm:parabolic:general:reduce}, we  need $\Qa^{(m,n)}$,  $\Qb^{(m,n)}$,  $\Qc^{(m,n)}$ with $m+n\le 4$ and  $\psi^{(m,n)}$ with $m+n\le 5$. As $\kappa$ and $f$ are available in \eqref{Model：Original:Parabolic},  the expressions of   \eqref{sub:c:1}, \eqref{sub:c:2}, and \eqref{sub:c:3} imply that, we need to evaluate
	\be\label{auk:mn:parabolic}
	\left(\frac{\alpha_u(u)} {  \kappa}\right)^{(m,n)} \  \text{and} \  \left(\frac{\beta_u(u)} {  \kappa}\right)^{(m,n)} \  \text{with} \  m+n\le4; \quad \text{and} \quad \left(\frac{u} {  \kappa}\right)^{(m,n)} \  \text{with} \  m+n\le5.
	\ee
	We use the FDM in \cref{amn:bmn} to evaluate the high-order partial derivatives in \eqref{auk:mn:parabolic}.
	\section{Numerical experiments}\label{sec:Numeri}
	Recall that 
	the spatial  domain $\Omega=(0,1)^2$, the temporal domain $I=[0,1]$, $u=u(x,y)$ and $u=u(x,y,t)$ are the exact solutions of the model problems \eqref{Model：Original:Elliptic} and \eqref{Model：Original:Parabolic}, respectively, 
	$u_h$ is the numerical solution of \eqref{Model：Original:Elliptic} computed by \cref{algm1}, and $u^n_h$ is the numerical solution of \eqref{Model：Original:Parabolic} at $t=t_n$ computed by \cref{algm2,algm3,algm4}, where
	\[
	(x_i,y_j)=(i h,jh), \qquad h=1/N_1,\qquad t_{n}=n \tau, \qquad \tau=1/N_2=rh, \qquad N_1/N_2=r, \qquad N_1,N_2\in \N,
	\]
	$i,j=0,\ldots,N_1$ and $n=0,\ldots,N_2$. In our numerical examples, we choose  $r=1/2$ for the CN method, and $r=1$ for the BDF3 and BDF4 methods.

	To verify the accuracy and the convergence rates of \cref{algm1,algm2,algm3,algm4}, we use the following $l_2$ and $l_{\infty}$  norms of errors for the model problems \eqref{Model：Original:Elliptic} and \eqref{Model：Original:Parabolic}, 
	\[
	\begin{split}
		&	\|u_{h}-u\|_{2}^2:= h^2\sum_{i,j=0}^{N_1} \left((u_h)_{i,j}-u(x_i,y_j)\right)^2,\qquad \|u_h-u\|_\infty
		:=\max_{0\le i,j\le N_1} \left|(u_h)_{i,j}-u(x_i,y_j)\right|,  \\ 
		& \text{for the steady nonlinear   convection-diffusion equation }  \eqref{Model：Original:Elliptic};\\
		&	\|u_{h}-u\|_{2}^2:= h^2\sum_{i,j=0}^{N_1} \left((u^{N_2}_h)_{i,j}-u(x_i,y_j,1)\right)^2, \qquad \|u_h-u\|_\infty
		:=\max_{0\le i,j\le N_1} \left|(u^{N_2}_h)_{i,j}-u(x_i,y_j,1)\right|,\\
		& \text{for the  time-dependent nonlinear convection-diffusion equation }  \eqref{Model：Original:Parabolic};
	\end{split}
	\]
	where $(u_{h})_{i,j}$ and $(u^{N_2}_{h})_{i,j}$ are values of $u_h$ and $u^{N_2}_h$ at $(x_i, y_j)$, respectively. 
	In this section, we provide 2 examples for \eqref{Model：Original:Elliptic} and \eqref{Model：Original:Parabolic} in  \cref{sec:example:elliptic} and \cref{sec:example:parabolic}, respectively. 
	\subsection{Two examples of the steady nonlinear   convection-diffusion equation}\label{sec:example:elliptic}
	In the following \cref{Example:1}, we choose the variable diffusion coefficient $\kappa=\kappa(x,y)$ to examine the performance of  \cref{algm1} with \cref{thm:elliptic:general,thm:elliptic:general:reduce}.
	\begin{example}\label{Example:1}
		\normalfont
		The exact solution, the diffusion coefficient, and the nonlinear convection term  in \eqref{Model：Original:Elliptic} are given by
		%%%
		\begin{align*}
			&u=\sin(3x)\cos(7y),\qquad \kappa= 2+\sin (5x-2y),\qquad  \alpha= \cos(u), \qquad \beta= \sin(u), 
		\end{align*}
		$f$ and $g$  are obtained by plugging the above functions into \eqref{Model：Original:Elliptic}.
		The numerical results are presented in \cref{Example:1:table} and \cref{Example:1:fig}.	
		As we reduce the pollution effects of $\bo(h^4)$ and $\bo(h^5)$ in \cref{thm:elliptic:general:reduce},  we observe that \cref{thm:elliptic:general:reduce} produces the smaller errors than \cref{thm:elliptic:general} in \cref{Example:1:table}. Furthermore, the convergence rates of \cref{algm1} with \cref{thm:elliptic:general:reduce} are higher than 4 when $h\ge 1/2^5$.
	\end{example}
	\begin{table}[htbp]
		\caption{The performance in \cref{Example:1} of the proposed  \cref{algm1} with \cref{thm:elliptic:general,thm:elliptic:general:reduce}, where 'RPE' denotes the 'reduced pollution effect'.}
		\centering
		{\renewcommand{\arraystretch}{1.5}
		\scalebox{1.1}{
			\setlength{\tabcolsep}{3mm}{
				\begin{tabular}{c|c|c|c|c|c|c|c|c}
					\hline
					\multicolumn{1}{c|}{} &
					\multicolumn{4}{c|}{\cref{algm1} with \cref{thm:elliptic:general}} &
					\multicolumn{4}{c}{ \cref{algm1} with \cref{thm:elliptic:general:reduce} } \\
					\hline
\multicolumn{1}{c|}{} &
\multicolumn{4}{c|}{The 4th-order FDM without RPE} &
\multicolumn{4}{c}{The 4th-order FDM with  RPE } \\
					\cline{1-9}
					$h$&   ${\|u_{h}-u\|_2}$    &order &   $\|u_{h}-u\|_\infty$    &order &   ${\|u_{h}-u\|_2}$    &order &   $\|u_{h}-u\|_{\infty}$    &order \\
					\hline
					$1/2^3$  &  1.2164E-03  &    &  3.0974E-03  &    &  1.1144E-04  &    &  3.7841E-04  &  \\
					$1/2^4$  &  7.2612E-05  &  4.07  &  2.0035E-04  &  3.95  &  2.2191E-06  &  5.65  &  5.9879E-06  &  5.98\\
					$1/2^5$  &  4.5074E-06  &  4.01  &  1.2469E-05  &  4.01  &  1.1455E-07  &  4.28  &  2.9631E-07  &  4.34\\
					$1/2^6$  &  2.8133E-07  &  4.00  &  7.7980E-07  &  4.00  &  7.1254E-09  &  4.01  &  1.8216E-08  &  4.02\\
					$1/2^7$  &  1.7585E-08  &  4.00  &  4.8776E-08  &  4.00  &  4.4590E-10  &  4.00  &  1.1612E-09  &  3.97\\
					$1/2^8$  &  1.0995E-09  &  4.00  &  3.0505E-09  &  4.00  &  2.7868E-11  &  4.00  &  7.3088E-11  &  3.99\\				
					\hline
		\end{tabular}}}}
		\label{Example:1:table}
	\end{table}
	\begin{figure}[htbp]
		\centering
		\begin{subfigure}[b]{0.24\textwidth}
			\includegraphics[width=4.8cm,height=4.8cm]{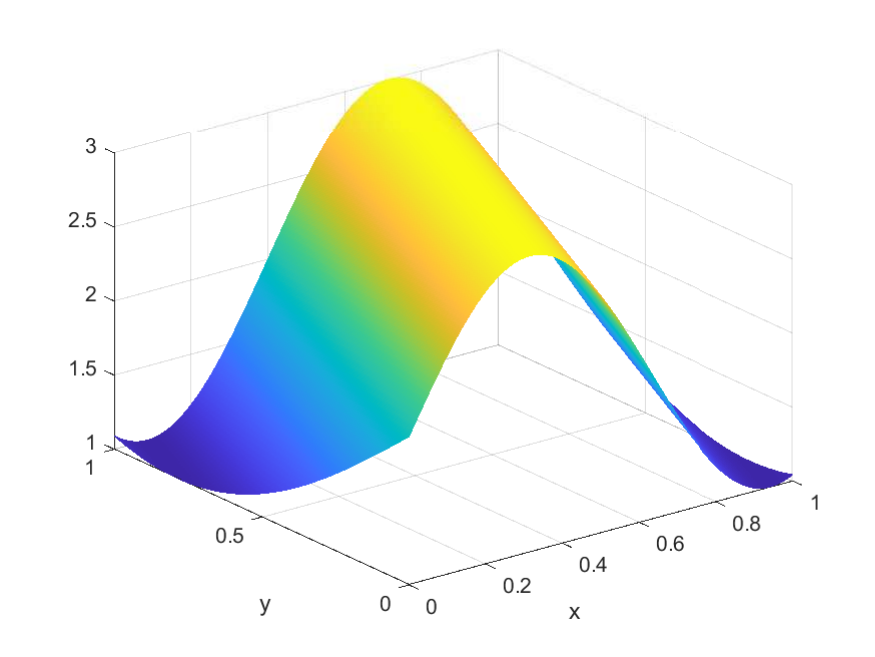}
		\end{subfigure}
		\begin{subfigure}[b]{0.24\textwidth}
			\includegraphics[width=4.8cm,height=4.8cm]{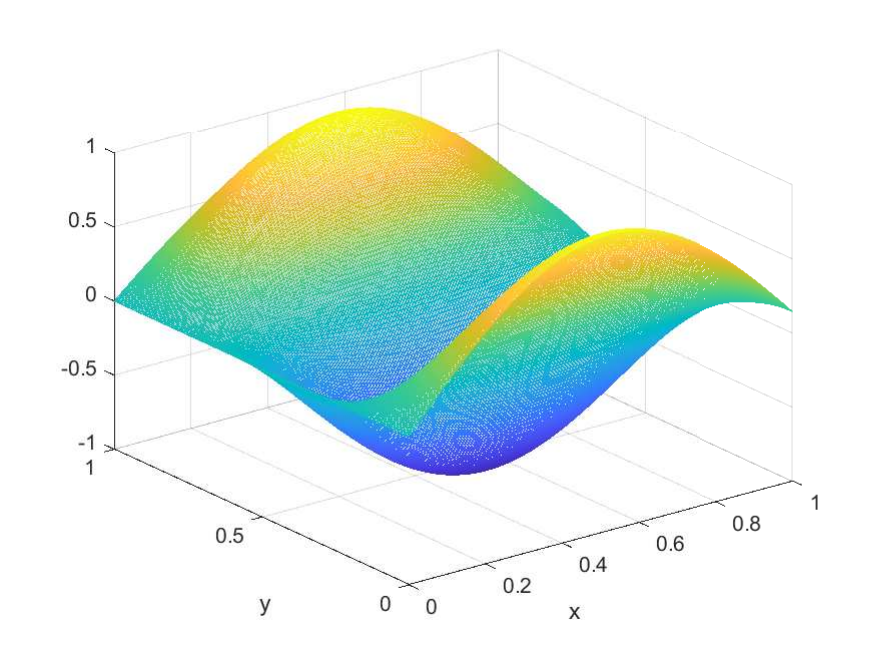}
		\end{subfigure}
		\begin{subfigure}[b]{0.24\textwidth}
			\includegraphics[width=4.8cm,height=4.8cm]{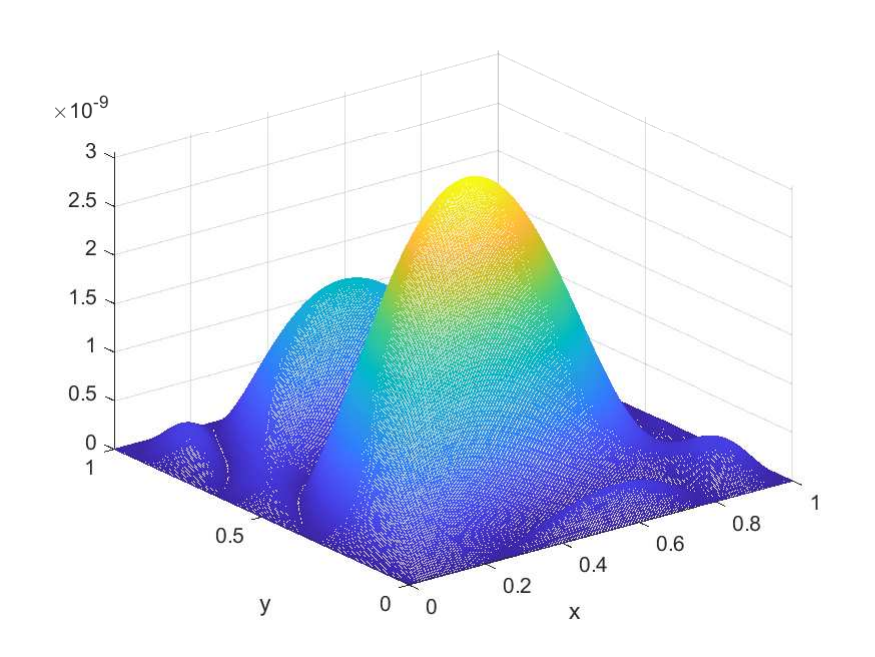}
		\end{subfigure}
		\begin{subfigure}[b]{0.24\textwidth}
			\includegraphics[width=4.8cm,height=4.8cm]{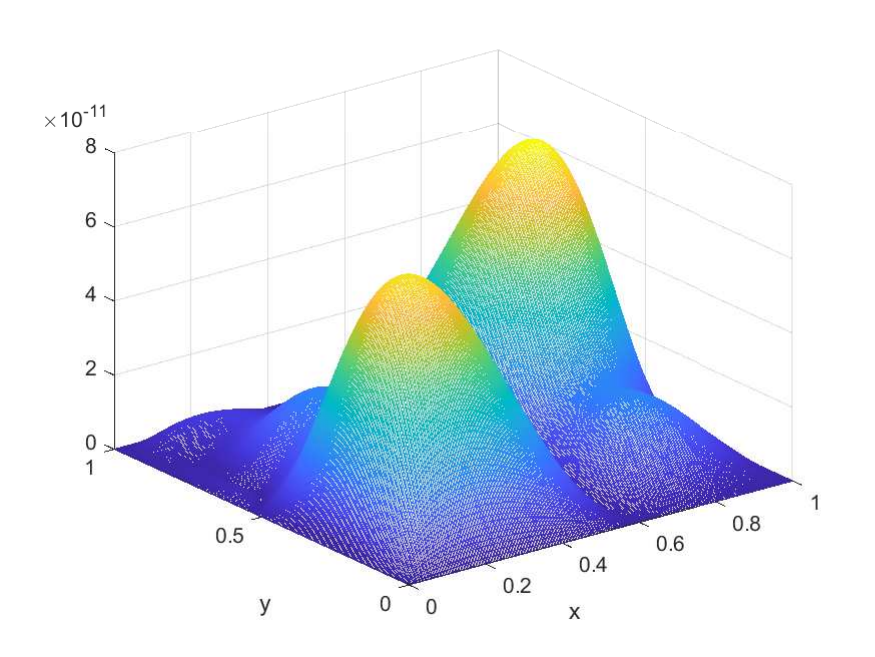}
		\end{subfigure}
		\caption
		{\cref{Example:1}: The diffusion coefficient $\kappa$ (first),  the exact solution $u$ (second), the error $|u_h-u|$ with the numerical solution $u_h$ computed by \cref{algm1} and \cref{thm:elliptic:general} (third), and  the error $|u_h-u|$ with the numerical solution $u_h$ computed by \cref{algm1} and \cref{thm:elliptic:general:reduce} (fourth) on the closure of the spatial domain $[0,1]^2$ with $h=2^{-8}$.}
		\label{Example:1:fig}
	\end{figure}	
	In the following \cref{Example:2}, we choose the constant $\kappa$ to compare \cref{algm1} with the discontinuous Galerkin  method constructed in \citep{Nguyen2009}. Since we do not apply the postprocessing procedure in  \cref{algm1}, we compare the numerical results from \citep{Nguyen2009} without the postprocessing procedure to ensure a fair comparison. 
	\begin{example}\label{Example:2}
		\normalfont
		The exact solution, the diffusion coefficient, and the nonlinear convection term  in \eqref{Model：Original:Elliptic} are given by
		%%%
		\begin{align*}
			&\qquad u=xy\tanh ( (1-x)/{\kappa})\tanh ( (1-y)/{\kappa}),\qquad \kappa=1/10,\qquad \alpha= u^2/2, \qquad \beta= u^2/2.
		\end{align*}
		The numerical results are presented in \cref{Example:2:table} and \cref{Example:2:fig}.	
		The error	from  \cref{algm1} with \cref{thm:elliptic:general:reduce} is less than one-sixth of that in \citep{Nguyen2009} when $h=1/2^6$ in \cref{Example:2:table}.
		We do not have the data of \citep{Nguyen2009} if $h<1/2^6$, but \cref{Example:2:table} indicates that \cref{algm1}  with \cref{thm:elliptic:general:reduce} achieves a stable convergence order of 4 if $h<1/2^6$. So we can expect that the errors of \citep{Nguyen2009} are approximately six times larger than those of our proposed method when $h<1/2^6$.
		Furthermore, \cref{algm1} generates a matrix with only 9 nonzero bands, whereas \citep{Nguyen2009} requires more than 9 nonzero bands to produce the results in \cref{Example:2:table}.
	\end{example}
	\begin{table}[htbp]
		\caption{The performance in \cref{Example:2} of the proposed  \cref{algm1} with \cref{thm:elliptic:special,thm:elliptic:general:reduce}, where 'RPE' denotes the 'reduced pollution effect'. The ratio $\Qr$ is equal to  ${\|u_{h}-u\|_2}$ of \citep{Nguyen2009} divided by ${\|u_{h}-u\|_2}$ of \cref{algm1} with \cref{thm:elliptic:general:reduce}.}
		\centering
			{\renewcommand{\arraystretch}{1.5}
		\scalebox{1}{
			\setlength{\tabcolsep}{1.2mm}{
				\begin{tabular}{c|c|c|c|c|c|c|c}
					\hline
					\multicolumn{1}{c|}{} &
					\multicolumn{2}{c|}{ \cref{thm:elliptic:special} (the 4th-order FDM)  } &
					\multicolumn{2}{c|}{  \cref{thm:elliptic:general:reduce} (the 4th-order FDM)  }& 
					\multicolumn{2}{c|}{ \citep{Nguyen2009}  (the 4th-order DG) }&
					\multicolumn{1}{c}{ } \\
					\hline
					\multicolumn{1}{c|}{} &
					\multicolumn{2}{c|}{ Without RPE } &
					\multicolumn{2}{c|}{ With RPE }& 
					\multicolumn{2}{c|}{  }&
					\multicolumn{1}{c}{ } \\
					\cline{1-8}
					&     &  &   col4 &    &   col6  &   &   $\Qr=$ col6/col4 \\
					\hline
					$h$& \hspace{1cm}  ${\|u_{h}-u\|_2}$ \hspace{1cm}    & order &     \hspace{1cm} ${\|u_{h}-u\|_2}$  \hspace{1cm}  &  order  &    \hspace{1cm} ${\|u_{h}-u\|_2}$  \hspace{1cm}  & order  &   $\Qr$ \\
					\hline
					$1/2^3$  &  1.0157E-02  &    &  1.1651E-02  &    &  5.97E-04  &    &  0.05 \\
					$1/2^4$  &  3.9676E-04  &  4.68  &  6.6129E-05  &  7.46  &  4.14E-05  &  3.85  &  0.63 \\
					$1/2^5$  &  3.3975E-05  &  3.55  &  8.8686E-07  &  6.22  &  2.79E-06  &  3.89  &  3.15 \\
					$1/2^6$  &  2.0859E-06  &  4.03  &  2.8940E-08  &  4.94  &  1.77E-07  &  3.98  &  6.12 \\
					$1/2^7$  &  1.3078E-07  &  4.00  &  1.7415E-09  &  4.05  &    &    &  \\
					$1/2^8$  &  8.2163E-09  &  3.99  &  1.1497E-10  &  3.92  &    &    &  \\
					$1/2^9$  &  5.1521E-10  &  4.00  &  7.6643E-12  &  3.91  &    &    &  \\				
					\hline
		\end{tabular}}}}
		\label{Example:2:table}
	\end{table}
	\begin{figure}[htbp]
		\centering
		\begin{subfigure}[b]{0.24\textwidth}
			\includegraphics[width=4.8cm,height=4.8cm]{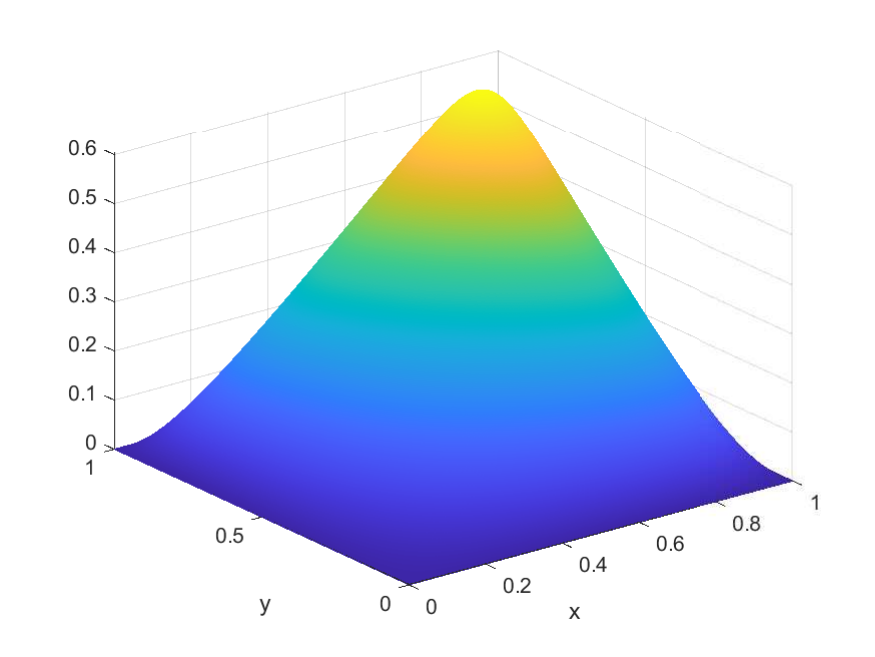}
		\end{subfigure}
		\begin{subfigure}[b]{0.24\textwidth}
			\includegraphics[width=4.8cm,height=4.8cm]{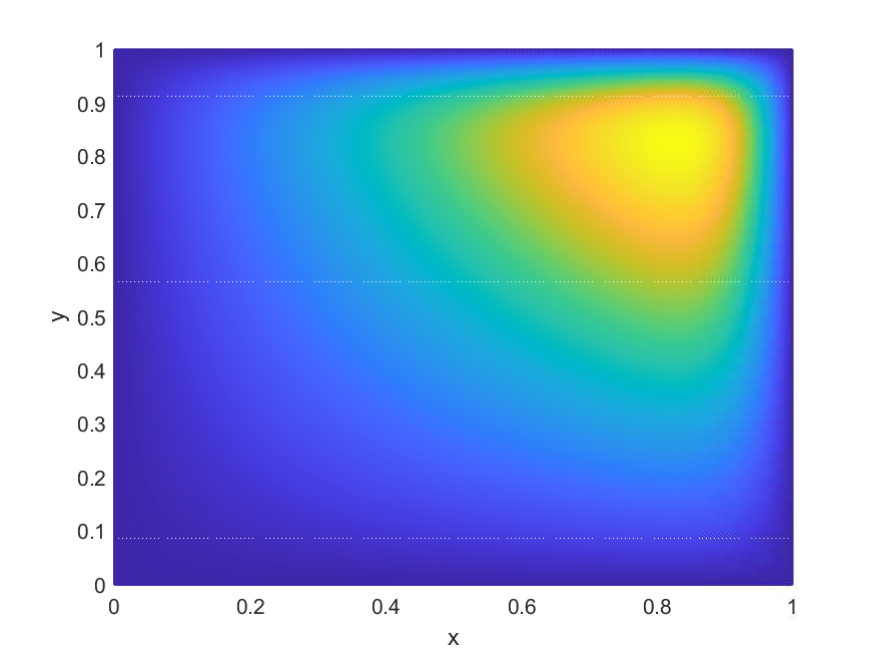}
		\end{subfigure}
		\begin{subfigure}[b]{0.24\textwidth}
			\includegraphics[width=4.8cm,height=4.8cm]{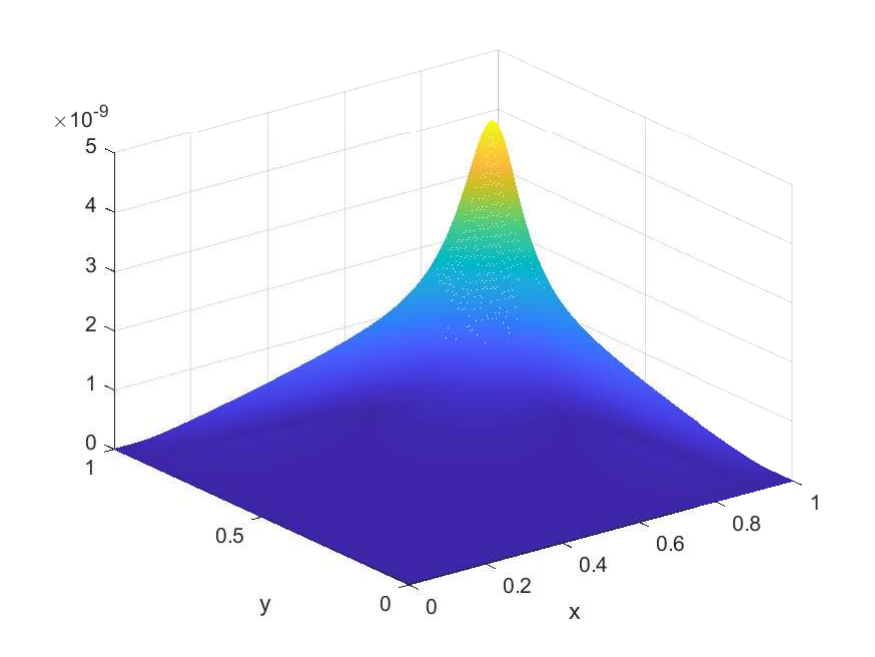}
		\end{subfigure}
		\begin{subfigure}[b]{0.24\textwidth}
			\includegraphics[width=4.8cm,height=4.8cm]{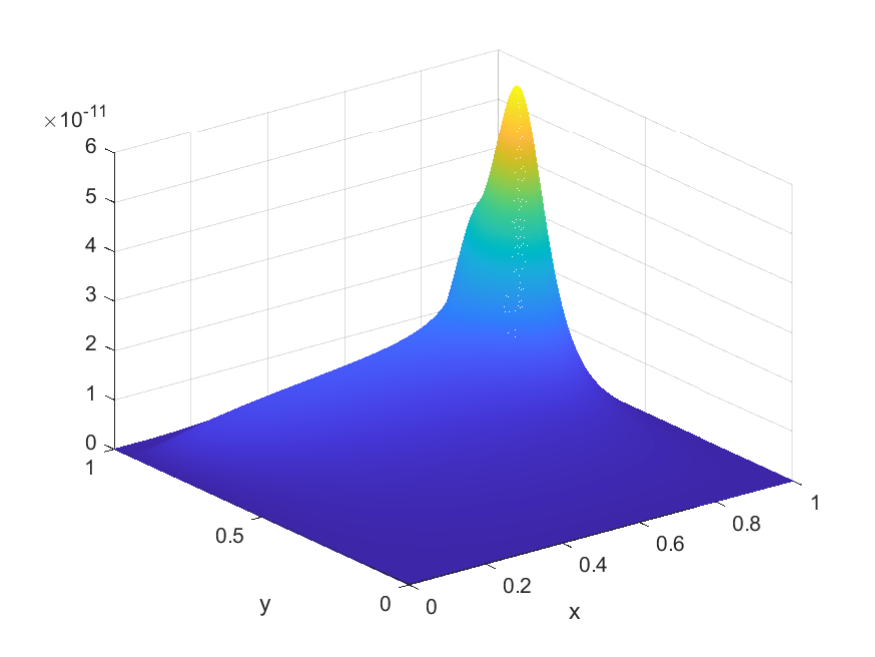}
		\end{subfigure}
		\caption
		{\cref{Example:2}:   The exact solution $u$ (first and second), the error $|u_h-u|$  with the numerical solution $u_h$ computed by \cref{algm1} with \cref{thm:elliptic:special} (third), and  the error $|u_h-u|$  with the numerical solution $u_h$ computed by \cref{algm1} with \cref{thm:elliptic:general:reduce} (fourth) on the closure of the spatial domain $[0,1]^2$ with $h=2^{-9}$.}
		\label{Example:2:fig}
	\end{figure}	
	Next, we present 2 examples to verify the accuracy and the convergence rates of \cref{algm2,algm3,algm4} for the  time-dependent nonlinear convection-diffusion equation \eqref{Model：Original:Parabolic} in the following \cref{sec:example:parabolic}.  
	\subsection{Two examples of the  time-dependent nonlinear convection-diffusion equation}\label{sec:example:parabolic}
	Recall that $\tau=rh$ in \eqref{sub:c:1}, \eqref{sub:c:2}, and \eqref{sub:c:3}. In the following \cref{Example:3,Example:4}, we choose $r=1/2$ in \cref{algm2} and $r=1$ in \cref{algm3,algm4}. 
	\begin{example}\label{Example:3}
		\normalfont
		The exact solution, the diffusion coefficient, and the nonlinear convection term  in \eqref{Model：Original:Parabolic} are given by
		%%%
		\begin{align*}
			&u=\sin(3t)\cos(2x-y),\qquad \kappa= 3+\cos (x+3y+t),\qquad  \alpha= -u^3/3, \qquad \beta= \sin(u).
		\end{align*}
		The numerical results are presented in \cref{Example:3:table} and \cref{Example:3:fig}.	 \cref{Example:3:table} confirms the third-order and fourth-order convergence rates of the BDF3 method in \cref{algm3} and the BDF4 method in \cref{algm4}, respectively, for the diffusion coefficient $\kappa=\kappa(x,y,t)$.
	\end{example}
	\begin{table}[htbp]
		\caption{The performance in \cref{Example:3} of the proposed  \cref{algm3,algm4}.}
		\centering
			{\renewcommand{\arraystretch}{1.5}
		\scalebox{1.2}{
			\setlength{\tabcolsep}{2mm}{
				\begin{tabular}{c|c|c|c|c|c|c|c|c|c}
					\hline
					\multicolumn{1}{c}{} &
					\multicolumn{1}{c|}{} &
					\multicolumn{4}{c|}{\cref{algm3} with $\tau=h$ (BDF3) } &
					\multicolumn{4}{c}{ \cref{algm4} with $\tau=h$ (BDF4)} \\
					\hline
					\multicolumn{1}{c}{} &
					\multicolumn{1}{c|}{} &
					\multicolumn{4}{c|}{Using the 4th-order FDM  } &
					\multicolumn{4}{c}{ Using the 4th-order FDM } \\
					\cline{1-10}
					$h$& $\tau $&  ${\|u_{h}-u\|_2}$    &order &   $\|u_{h}-u\|_\infty$    &order &   ${\|u_{h}-u\|_2}$    &order &   $\|u_{h}-u\|_{\infty}$    &order \\
					\hline
					$1/2^3$  & $1/2^3$  &  3.1777E-04  &    &  6.0055E-04  &    &  2.3851E-04  &    &  4.7264E-04  &  \\
					$1/2^4$  & $1/2^4$  &  2.7632E-05  &  3.52  &  5.4008E-05  &  3.48  &  8.2563E-06  &  4.85  &  1.5931E-05  &  4.89\\
					$1/2^5$  & $1/2^5$  &  2.5943E-06  &  3.41  &  5.0863E-06  &  3.41  &  5.8285E-07  &  3.82  &  1.1277E-06  &  3.82\\
					$1/2^6$  & $1/2^6$  &  2.6848E-07  &  3.27  &  5.2770E-07  &  3.27  &  3.7437E-08  &  3.96  &  7.2471E-08  &  3.96\\
					$1/2^7$  & $1/2^7$  &  3.0023E-08  &  3.16  &  5.9143E-08  &  3.16  &  2.3653E-09  &  3.98  &  4.5802E-09  &  3.98\\
					$1/2^8$  & $1/2^8$  &  3.5303E-09  &  3.09  &  6.9619E-09  &  3.09  &  1.4867E-10  &  3.99  &  2.8794E-10  &  3.99\\				
					\hline
		\end{tabular}}}}
		\label{Example:3:table}
	\end{table}
	\begin{figure}[htbp]
		\centering
		\begin{subfigure}[b]{0.24\textwidth}
			\includegraphics[width=4.8cm,height=4.8cm]{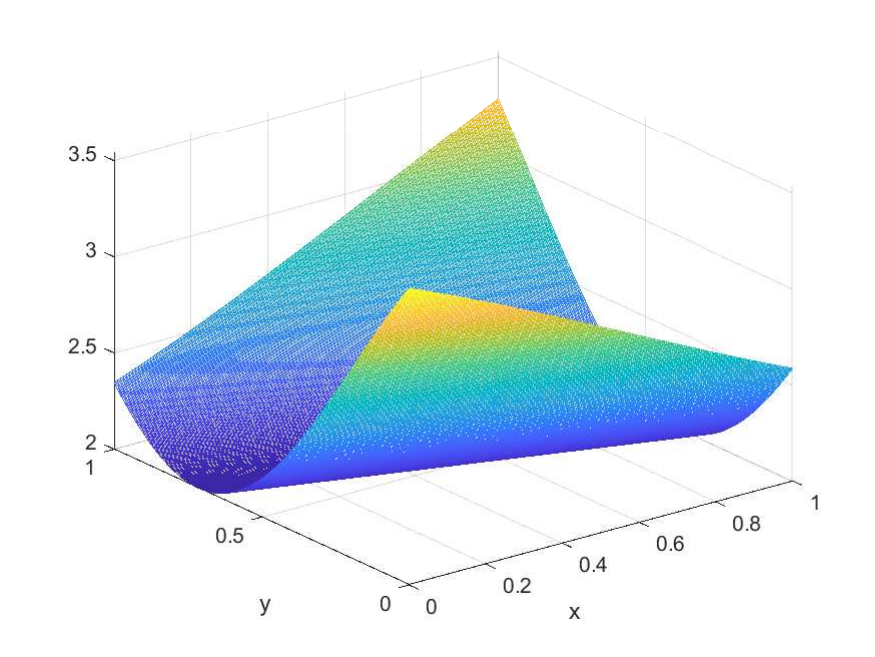}
		\end{subfigure}
		\begin{subfigure}[b]{0.24\textwidth}
			\includegraphics[width=4.8cm,height=4.8cm]{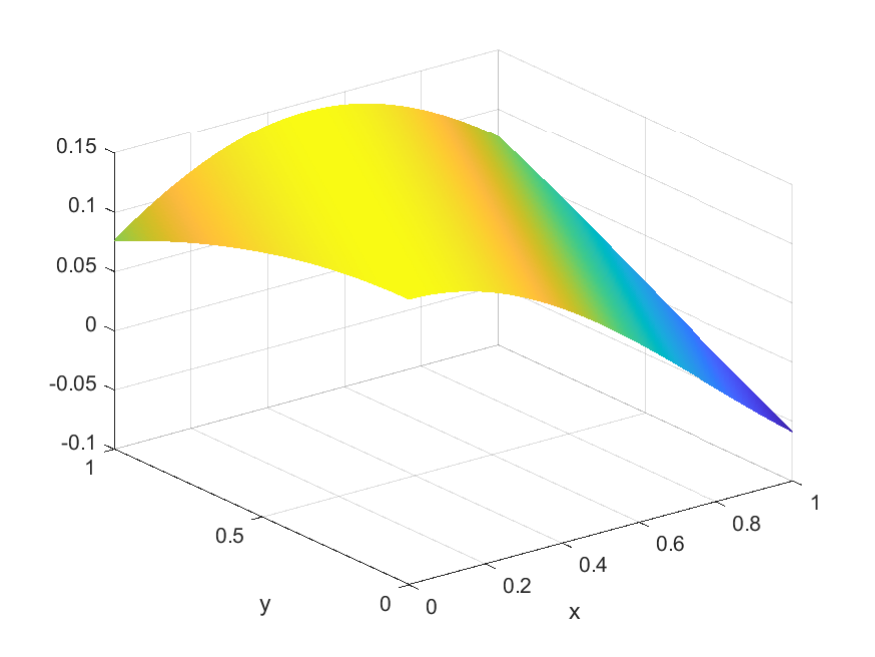}
		\end{subfigure}
		\begin{subfigure}[b]{0.24\textwidth}
			\includegraphics[width=4.8cm,height=4.8cm]{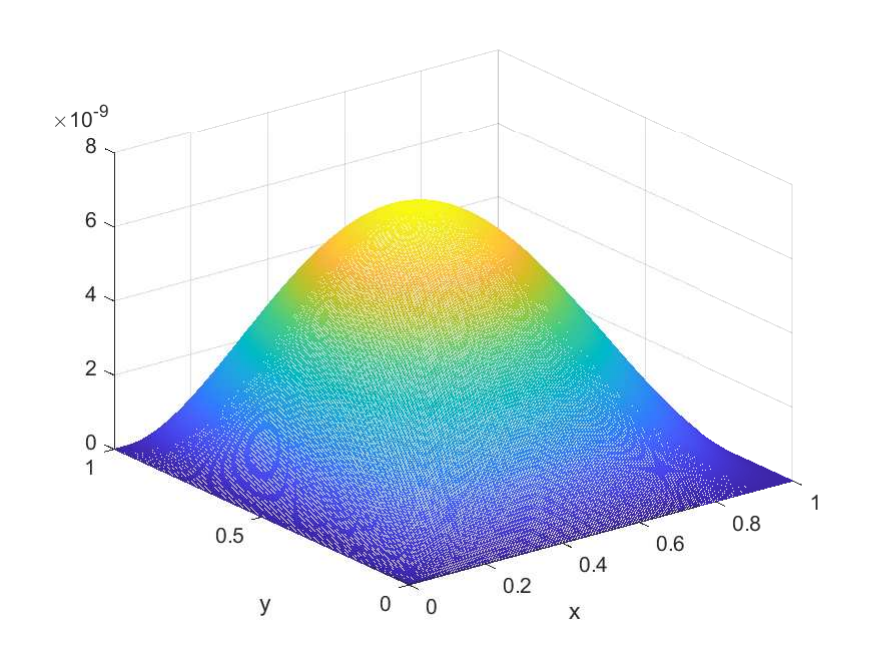}
		\end{subfigure}
		\begin{subfigure}[b]{0.24\textwidth}
			\includegraphics[width=4.8cm,height=4.8cm]{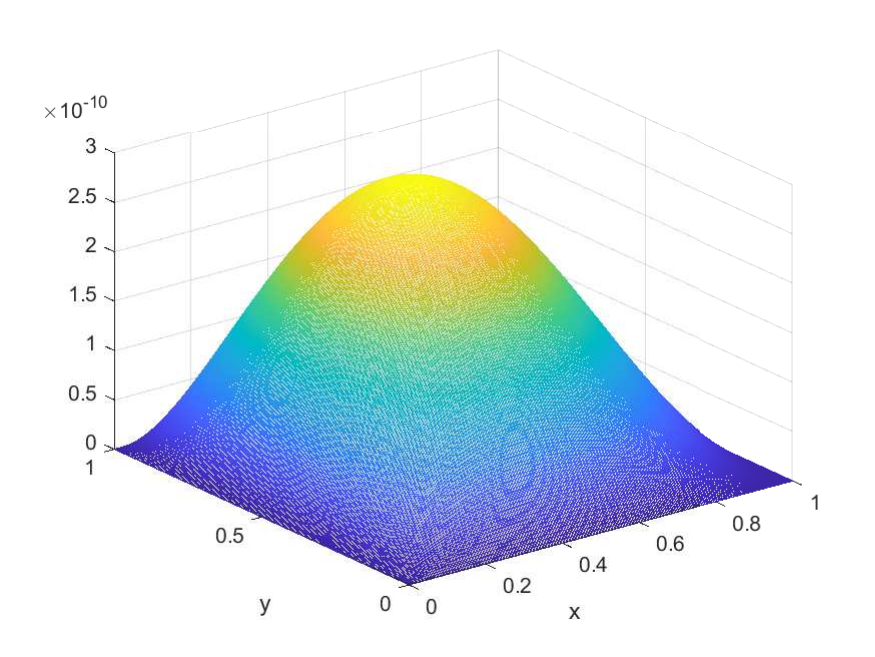}
		\end{subfigure}
		\caption
		{\cref{Example:3}: The diffusion coefficient $\kappa$ at $t=1$ (first),  the exact solution $u$ at $t=1$ (second), the error $|u_h-u|$ at $t=1$ with the numerical solution $u_h$ computed by \cref{algm3}  (third), and the error $|u_h-u|$ at $t=1$ with the numerical solution $u_h$ computed by \cref{algm4}  (fourth)  on the closure of the spatial domain  $[0,1]^2$ with $h=2^{-8}$.}
		\label{Example:3:fig}
	\end{figure}	
	To demonstrate the efficiency and accuracy of our FDM, a comparison with the results in \citep{Nguyen2009} is presented in the following \cref{Example:4}.
	\begin{example}\label{Example:4}
		\normalfont
		The exact solution, the diffusion coefficient, and the nonlinear convection term  in \eqref{Model：Original:Parabolic} are given by
		%%%
		\begin{align*}
			&u=(\exp(t)-1)xy\tanh ( (1-x)/{\kappa})\tanh ( (1-y)/{\kappa}),\qquad \kappa= 1/10,\qquad  \alpha= u^2/2, \qquad \beta= u^2/2. 
		\end{align*}
		The numerical results are presented in \cref{Example:4:table} and \cref{Example:4:fig}.	We note that the results in \cref{Example:4:table} from \citep{Nguyen2009} are computed by the BDF3 method,  while the proposed \cref{algm2} and \cref{algm3} use the CN and BDF3 methods, respectively. According to \cref{Example:4:table}, even though \cref{algm2} is second-order accurate in the temporal discretization, and uses a larger time step $\tau$ than \citep{Nguyen2009}, \cref{algm2} yields the smaller errors than those from the third-order BDF3 method in \citep{Nguyen2009}. Furthermore, if we apply the same BDF3 method for the time discretization, then the error from \citep{Nguyen2009} is approximately 63 times greater than that of \cref{algm3} when $h=1/2^6$, even the time step of \cref{algm3} is approximately 3 times larger than that of \citep{Nguyen2009}. Since we reduce the truncation errors of $\bo(h^4)$ and $\bo(h^5)$ in \cref{algm2,algm3}, the numerical orders of \cref{algm2,algm3} are higher than 2 and 3, respectively if $h\ge 1/2^7$. When $h=1/2^8$, we obtain the desired convergence rates 2 and 3. Finally,  \cref{algm2,algm3} form a nine-band matrix, but the number of nonzero bands in \citep{Nguyen2009} to generate the results in \cref{Example:4:table} is higher than 9.
	\end{example}
	\begin{table}[htbp]
		\caption{The performance in \cref{Example:4} of the proposed  \cref{algm2,algm3}. The ratios $\Qr_2$ and $\Qr_3$  are equal to  ${\|u_{h}-u\|_2}$ of \citep{Nguyen2009} divided by ${\|u_{h}-u\|_2}$ of \cref{algm2} and \cref{algm3}, respectively. }
		\centering
			{\renewcommand{\arraystretch}{1.5}
		\scalebox{1}{
			\setlength{\tabcolsep}{1.1mm}{
				\begin{tabular}{c|c|c|c|c|c|c|c|c|c|c|c}
					\hline
					\multicolumn{1}{c|}{} &
					\multicolumn{3}{c|}{  \cref{algm2} with $\tau=h/2$ (CN)}& 
					\multicolumn{3}{c|}{  \cref{algm3}  with $\tau=h$ (BDF3)}& 
					\multicolumn{3}{c|}{ \citep{Nguyen2009}  with $\tau=1/200$ (BDF3)}&
					\multicolumn{1}{c|}{  } &
					\multicolumn{1}{c}{  }  \\
					\cline{1-12}
					&    & col3     &   &  &   col6  &   &  & col9  &   &   col9/col3 &  col9/col6  \\
					\hline
					$h$&   $\tau$ &  \hspace{0.3cm}  ${\|u_{h}-u\|_2}$ \hspace{0.3cm}    & order  & $\tau$ & \hspace{0.3cm}  ${\|u_{h}-u\|_2}$ \hspace{0.3cm}    & order  &    $\tau$ & \hspace{0.3cm} ${\|u_{h}-u\|_2}$  \hspace{0.3cm}   &order  &  $\Qr_2$ & $\Qr_3$  \\
					\hline
					$1/2^3$  &  $1/16$ &  3.3269E-03  &    &  $1/8$ & 1.8605E-03  &    & $1/200$ & 5.09E-03  &    &  1.53  &  2.74\\
					$1/2^4$  &  $1/32$ &   7.1559E-04  &  2.22  &  $1/16$ & 3.9588E-04  &  2.23  & $1/200$ & 7.86E-04  &  2.69  &  1.10  &  1.99\\
					$1/2^5$  &  $1/64$ & 3.3296E-05  &  4.43  &  $1/32$ & 8.7041E-06  &  5.51  & $1/200$ & 1.01E-04  &  2.97  &  3.03  &  11.6\\
					$1/2^6$  &  $1/128$ & 4.6570E-06  &  2.84  & $1/64$ & 2.0001E-07  &  5.44  & $1/200$ & 1.26E-05  &  2.99  &  2.71  &  63.0\\
					$1/2^7$  &  $1/256$ & 1.0770E-06  &  2.11  & $1/128$ & 1.0019E-08  &  4.32  &  &   &    &    &  \\
					$1/2^8$  &  $1/512$ & 2.6870E-07  &  2.00  & $1/256$ & 1.3557E-09  &  2.89  &  &  &    &    &  \\			
					\hline
		\end{tabular}}}}
		\label{Example:4:table}
	\end{table}
	\begin{figure}[htbp]
		\centering
		\begin{subfigure}[b]{0.24\textwidth}
			\includegraphics[width=4.8cm,height=4.8cm]{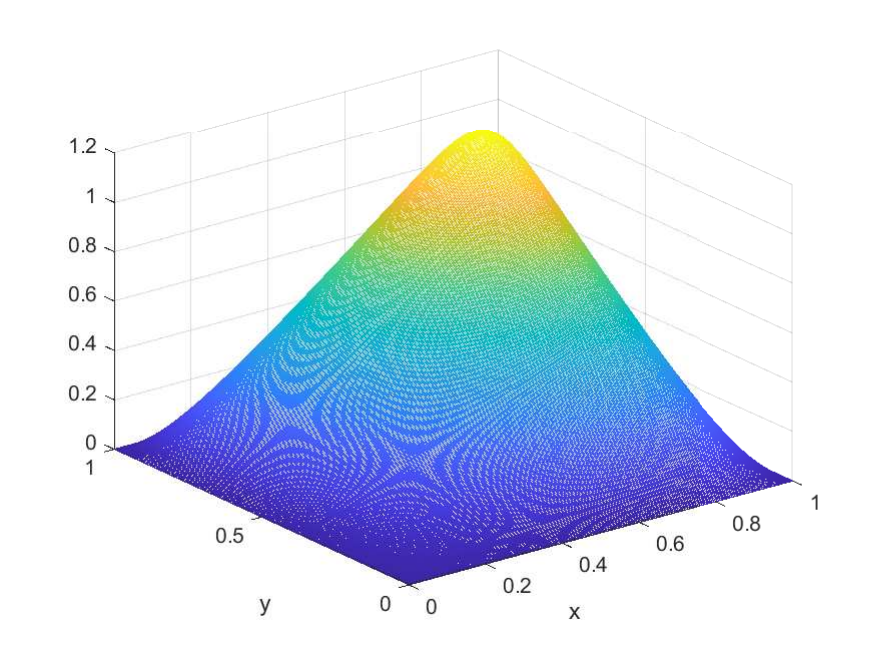}
		\end{subfigure}
		\begin{subfigure}[b]{0.24\textwidth}
			\includegraphics[width=4.8cm,height=4.8cm]{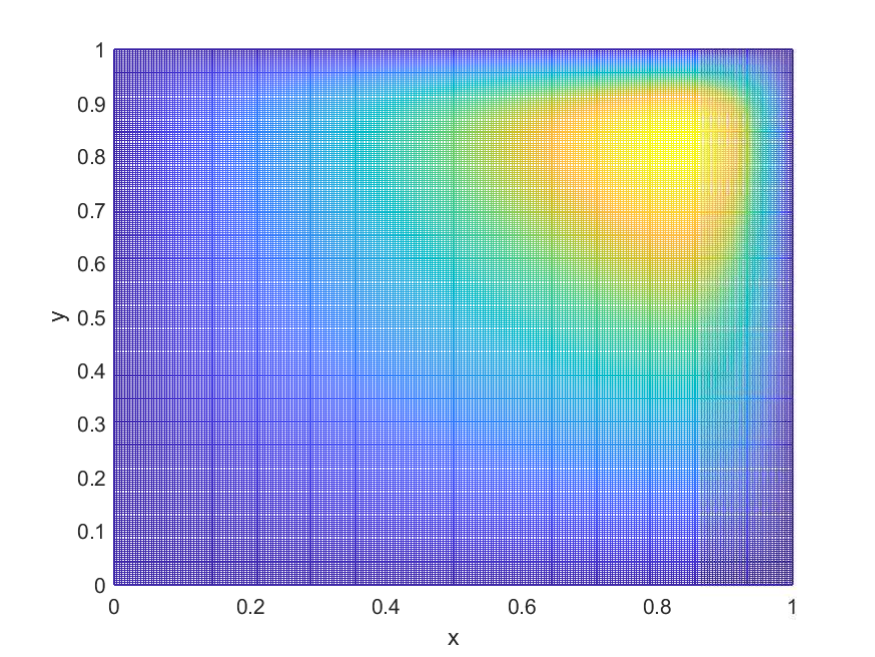}
		\end{subfigure}
		\begin{subfigure}[b]{0.24\textwidth}
			\includegraphics[width=4.8cm,height=4.8cm]{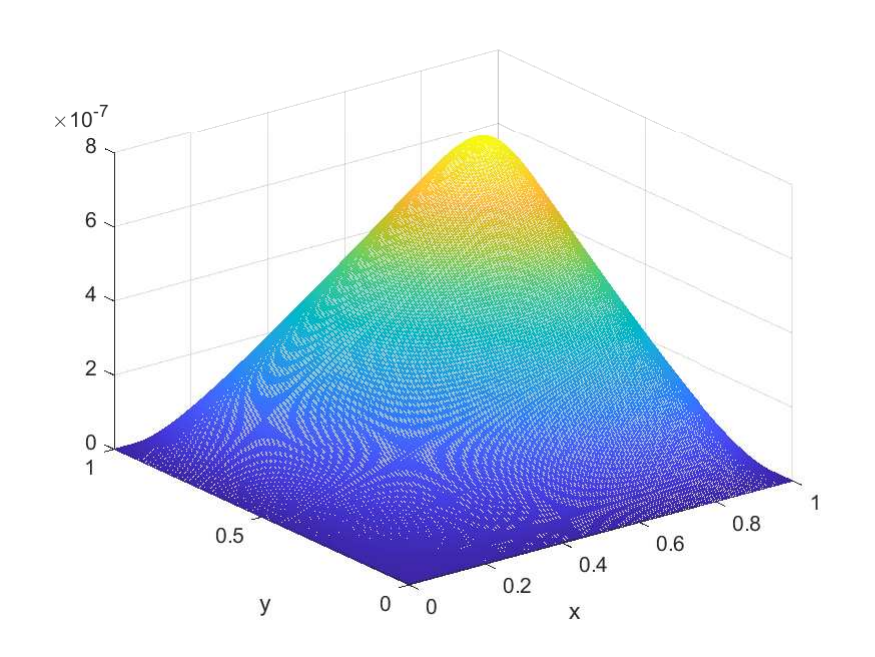}
		\end{subfigure}
		\begin{subfigure}[b]{0.24\textwidth}
			\includegraphics[width=4.8cm,height=4.8cm]{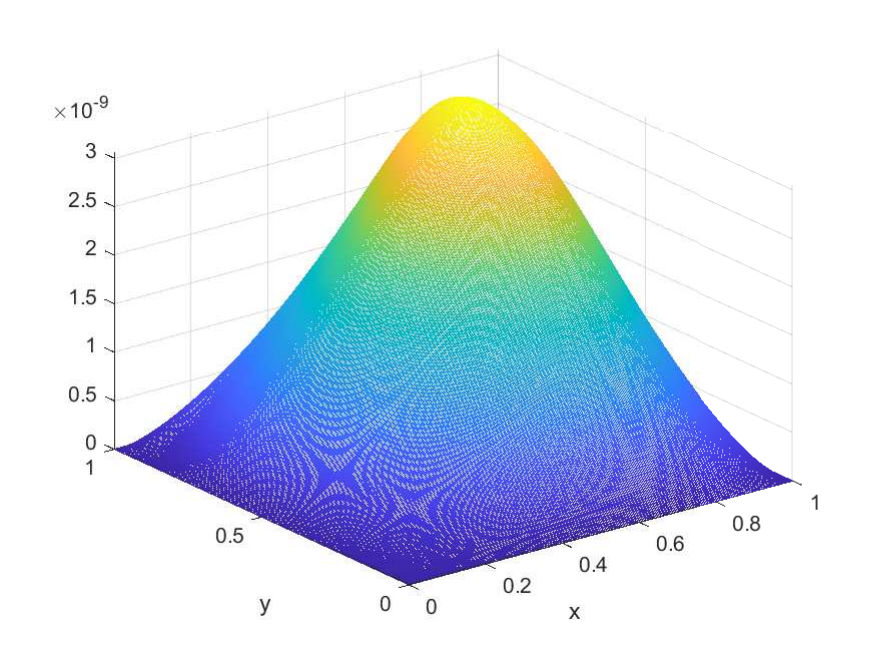}
		\end{subfigure}
		\caption
		{\cref{Example:4}:   The exact solution $u$ (first and second), the error $|u_h-u|$  with the numerical solution $u_h$ computed by \cref{algm2} (third), and  the error $|u_h-u|$  with the numerical solution $u_h$ computed by \cref{algm3} (fourth) on the closure of the spatial domain $[0,1]^2$ with $h=2^{-8}$.}
		\label{Example:4:fig}
	\end{figure}	
	\section{Contribution}\label{sec:Contribu}
	
	In this paper, we consider  the  steady and time-dependent nonlinear convection-diffusion equations  in a square domain  with the Dirichlet boundary condition. The main contributions of this paper are as follows:

	\begin{itemize}
		\item We present the fourth-order compact 9-point FDM for the steady nonlinear  equation, and derive the second-order to fourth-order compact 9-point FDMs for the  time-dependent nonlinear equation. To increase the accuracy, we modify FDMs to reduce the pollution effects. Each proposed  FDM preserves the discrete maximum principle and forms  an M-matrix, when $h$ is sufficiently small.
		\item We compare our method with the discontinuous Galerkin (DG) method in  \citep{Nguyen2009}, and the numerical results demonstrate that our proposed FDM generates the smaller errors. Precisely, when we apply the second-order CN method, our FDM scheme  produces the smaller errors than those from the third-order BDF3 and the DG methods in \citep{Nguyen2009}. Particularly, if the same BDF3 method is used, then we  achieve the error that is 1.6\% of that in \citep{Nguyen2009}.
		\item Our proposed method is accurate, robust, and stable for the variable and  time-dependent  diffusion coefficient $\kappa(x,y,t)$, and the challenging nonlinear term $\nab\cdot{\bm F} (u)$ (not limited to the Burgers equation).  The  examples  verify the accuracy and the theoretical convergence rates in the $l_2$ and $l_{\infty}$ norms. 
		\item The  matrix of the corresponding linear system constructed by our FDM only contains 9 nonzero bands. Due to the structure of the compact 9-point FDM, no special treatment is required for the grid points near the boundary. In comparison with the FEM, FVM, and DG methods, our high-order FDM avoids the numerical integration, resulting in the reduced computational cost. This advantage becomes particularly significant for the highly oscillatory $\kappa,\alpha,\beta,f$.
	\end{itemize}
	The proposed method can be naturally  extended to a 3D spatial domain and the more general nonlinear convection-diffusion-reaction equation:  $u_t-\nab\cdot (\kappa \nab u) +  \nab\cdot{\bm F} (u) +r(u) = f$, where $\kappa=\kappa(x,y,t,u)$.
	We also plan to extend our method to solve the more complicated problems of the two-phase flow in porous media in \citep{Jones2024} and the incompressible Boussinesq  equation in \citep{LiuWangJohn2003,WangLiuJohn2004}.

\section{Declarations}
\noindent \textbf{Conflict of interest:} The authors declare that they have no conflict of interest.\\
\noindent \textbf{Data availability:} Data will be made available on reasonable request.

\vspace{0.3cm}
\noindent\textbf{Acknowledgment}

Dr. John Burkardt (jvburkardt@gmail.com,  Department of Mathematics, University of Pittsburgh, Pittsburgh, PA 15260 USA) provided the editorial suggestions.

\begin{thebibliography}{99}
		
		
				

		
		
		\bibitem{Bause2012}
		M.~Bause and K.~Schwegler, Analysis of stabilized higher-order finite element approximation of nonstationary and nonlinear convection-diffusion-reaction equations. \emph{Comput. Methods Appl. Mech. Engrg.} \textbf{209-212} (2012),   184-196.
		
		
		
		\bibitem{Burkardt2020}
		J.~Burkardt and C.~Trenchea, Refactorization of the midpoint rule. \emph{Appl. Math. Lett.} \textbf{107} (2020),  106438.
		
		
		
		
		\bibitem{Burman2002}
		E.~Burman and A.~Ern, Nonlinear diffusion and discrete maximum principle for stabilized Galerkin approximations of the convection-diffusion-reaction equation. \emph{Comput. Methods Appl. Mech. Engrg.} \textbf{191} (2002), 3833-3855.
		
		
		
		
		
		\bibitem{Cances2020}
		C.~Canc\'{e}s, C.~C-Hillairet, M.~Herda, and S.~Krell,
		Large time behavior of nonlinear finite volume schemes for convection-diffusion equations. \emph{SIAM J. Numer. Anal.} \textbf{58} (2020), no. 5, 2544-2571.
		
		
		
		
		\bibitem{Clain2024} S.~Clain, D.~Lopes, R.~M.~S.~Pereira, and P.~A.~Pereira, Very high-order finite difference method on arbitrary geometries with Cartesian grids for non-linear convection diffusion reaction equations. \emph{J. Comput. Phys. }    \textbf{498} (2024),   112667.
		
		
		
		
		\bibitem{Cockburn1998}
		B.~Cockburn and C-W.~Shu,
		The local discontinuous Galerkin method for time-dependent convection-diffusion systems. \emph{SIAM J. Numer. Anal.} \textbf{35} (1998), no. 6,  2440-2463.
		
		
		
		
		\bibitem{Dolej2007}
		V.~Dolej\v{s}\'{i}, M.~Feistauer, and J.~Hozman, Analysis of semi-implicit DGFEM for nonlinear convection-diffusion problems on nonconforming meshes. \emph{Comput. Methods Appl. Mech. Engrg.} \textbf{196} (2007),  2813-2827.
		
		
		\bibitem{Dolej2005}
		V.~Dolej\v{s}\'{i}, M.~Feistauer, and V.~Sobot\'{i}kov\'{a}, Analysis of the discontinuous Galerkin method for nonlinear convection-diffusion problems. \emph{Comput. Methods Appl. Mech. Engrg.} \textbf{194} (2005),  2709-2733.
		
		
		
		\bibitem{DolejVlas2008}
		V.~Dolej\v{s}\'{i} and M.~Vlas\'{a}k, Analysis of a BDF-DGFE scheme for nonlinear convection-diffusion problems. \emph{Numer. Math.} \textbf{110} (2008), 405-447.
		
		
		
		
		
		\bibitem{Eymard2010}
		R.~Eymard, D.~Hilhorst, and M.~Vohral\'{i}k, A combined finite volume–finite element scheme for the discretization of strongly nonlinear convection-diffusion-reaction problems on nonmatching grids. \emph{Numer. Methods Partial Differ. Equ.} \textbf{26} (2010),   612-646.
		
		
		\bibitem{Feistauer1997}
		M.~Feistauer, J.~Felcman, and M.~L-Medvid'ov\'{a}, On the convergence of a combined finite volume-finite element method for nonlinear convection-diffusion problems. \emph{Numer. Methods Partial Differ. Equ.} \textbf{13} (1997),   163-190.
		
		
		
		
		
		\bibitem{Feng2022} Q.~Feng, B.~Han, and P.~Minev, Sixth order compact finite difference schemes for Poisson interface problems with singular sources. \emph{Comp. Math. Appl.} \textbf{99} (2021), 2-25.
		
		
		\bibitem{FHM2023} Q.~Feng, B.~Han, and P.~Minev, Compact 9-point finite difference methods with high accuracy order and/or M-matrix property for elliptic cross-interface problems. \emph{J. Comput. Appl. Math.} \textbf{428} (2023), 115151.
		
		
		
		\bibitem{Feng2024} Q.~Feng, B.~Han,  and P.~Minev, Sixth-order hybrid finite difference methods for elliptic interface problems with mixed boundary conditions. \emph{J. Comput. Phys. }    \textbf{497} (2024),  112635.
		
		
				
		\bibitem{Hairer1993}
		E.~Hairer, S.~P.~N\o rsett, and G.~Wanner,  Solving ordinary differential equations I: Nonstiff problems, 2nd ed. \emph{Springer-Verlag Berlin Heidelberg.} 1993.
		
		
			
	
	\bibitem{Hundsdorfer2003}
	W.~Hundsdorfer and J.~G.~Verwer, Numerical solution of time-dependent advection-diffusion-reaction equations. \emph{Springer-Verlag Berlin Heidelberg.} 2003.
	
	
		
		
		\bibitem{Jones2024}
		G.~S.~Jones and C.~Trenchea,
		Discrete energy balance equation via a symplectic second-order method for two-phase flow in porous media. \emph{Appl. Math. Comput.} \textbf{480} (2024), 128909.
		
		
		
		
		\bibitem{Kurganov2000} A.~Kurganov and E.~Tadmor, New high-resolution central schemes for nonlinear conservation laws and convection-diffusion equations. \emph{J. Comput. Phys. }    \textbf{160} (2000),   241-282.
		
		
		
		
		
		\bibitem{LiIto2001}
		Z.~Li and K.~Ito,
		Maximum principle preserving schemes for interface problems with discontinuous coefficients. \emph{SIAM J. Sci. Comput.} \textbf{23} (2001), no. 1, 339-361.
		
		
		
		\bibitem{LiZhang2020}
		H.~Li and X.~Zhang, On the  monotonicity and discrete maximum principle of the finite difference implementation of $C^0$-$Q^2$ finite element method. \emph{Numer. Math.} \textbf{145} (2020), 437-472.
		
		
	
	\bibitem{LiuWangJohn2003}
	J-G.~Liu, C.~Wang, and H.~Johnston, A fourth order scheme for incompressible Boussinesq equations. \emph{J. Sci. Comput.} \textbf{18} (2003), 253-285.
	
		
		
		
		\bibitem{Michoski2017}
		C.~Michoski, A.~Alexanderian, C.~Paillet, E.~J.~Kubatko, and C.~Dawson, Stability of nonlinear convection-diffusion-reaction systems in discontinuous Galerkin methods. \emph{J. Sci. Comput.} \textbf{70} (2017), 516-550.
		
		
		
		\bibitem{Feistauer2011}
		M.~Feistauer, V.~Ku\v{c}era, K.~Najzar, and J.~Prokopov\'{a}, Analysis of space–time discontinuous Galerkin method for nonlinear convection-diffusion problems. \emph{Numer. Math.} \textbf{117} (2011),  251-288.
		
		
		
		
		
		\bibitem{Nguyen2009} N.~C.~Nguyen, J.~Peraire, and B.~Cockburn, An implicit high-order hybridizable discontinuous Galerkin method for nonlinear convection-diffusion equations. \emph{J. Comput. Phys. }    \textbf{228} (2009),  8841-8855.
		
		
		
		
		\bibitem{Tezduyar1986}
		T.~E.~Tezduyar and Y.~J.~Park, Discontinuity-capturing finite element formulations for nonlinear convection-diffusion-reaction equations. \emph{Comput. Methods Appl. Mech. Engrg.} \textbf{59} (1986),   307-325.
		
		
		
		
	%%	\bibitem{WangLiu2003}
	%%	C.~Wang and J-G.~Liu, Fourth order convergence of compact finite difference solver for 2D incompressible flow. \emph{Commun. Appl. Anal.} \textbf{7} (2003), 171-191.
		
		
		
		
		
		\bibitem{WangLiuJohn2004}
		C.~Wang, J-G.~Liu, and H.~Johnston, Analysis of a fourth order finite difference method for the incompressible Boussinesq equations. \emph{Numer. Math.} \textbf{97} (2004), 555-594.
		
		
		
		
		
		\bibitem{XuShu2007}
		Y.~Xu and C-W.~Shu, Error estimates of the semi-discrete local discontinuous Galerkin method for nonlinear convection-diffusion and KdV equations. \emph{Comput. Methods Appl. Mech. Engrg.} \textbf{196} (2007),  3805-3822.
		
		
		
		
		\bibitem{Yan2013}
		J.~Yan, A new nonsymmetric discontinuous Galerkin method for time dependent convection diffusion equations. \emph{J. Sci. Comput.} \textbf{54} (2013), 663-683.
		
		
		


		
	\end{thebibliography}
\end{document}